\documentclass[sn-mathphys]{sn-jnl}

\jyear{2021}%

\usepackage{pifont}
\usepackage{amssymb}
\usepackage{colortbl}
\usepackage{subfigure}
\theoremstyle{thmstyleone}%
\newtheorem{remark}{Remark}%
\newtheorem{proposition}{Proposition}
\newtheorem{theorem}{Theorem}
\newtheorem{definition}{Definition}
\newtheorem{lemma}{Lemma}
\newcommand{\be}{\begin{equation}}
\newcommand{\ee}{\end{equation}}
\renewcommand{\proofname}{\textbf{Proof}}

\makeatletter
\renewcommand{\@thesubfigure}{\hskip\subfiglabelskip}
\makeatother

\usepackage{natbib}
\usepackage{url}
\title[Bilevel hyperparameter optimization for SVC]{{\huge\textmd{Bilevel hyperparameter optimization for support vector classification: theoretical analysis and a solution method}}}
\begin{document}


\author[1]{\fnm{Qingna} \sur{Li}}\email{qnl@bit.edu.cn}
\author[2]{\fnm{Zhen} \sur{Li}}\email{lizhenbeili@126.com}
\author*[3]{\fnm{Alain} \sur{Zemkoho}}\email{a.b.zemkoho@soton.ac.uk}

{\tiny{
\affil[1]{\orgdiv{School of Mathematics and Statistics}, \orgname{Beijing Key Laboratory on MCAACI/Key Laboratory of Mathematical Theory and Computation in Information Security, Beijing Institute of Technology}, \orgaddress{\city{Beijing}, \postcode{100081}, \country{P. R. China}}}
\affil[2]{\orgdiv{School of Mathematics and Statistics}, \orgname{Beijing Institute of Technology}, \orgaddress{\city{Beijing}, \postcode{100081}, \country{P. R. China}}}
\affil*[3]{\orgdiv{School of Mathematical Sciences}, \orgname{University of Southampton}, \orgaddress{\city{Southampton}, \postcode{SO17 1BJ}, \country{United Kingdom}}}
}}


\abstract{Support vector classification (SVC) is a classical and well-performed learning method for classification problems. A regularization parameter, which significantly affects the classification performance, has to be chosen and this is usually done by the cross-validation procedure. In this paper, we reformulate the  {hyperparameter} selection problem for support vector classification as a bilevel optimization problem in which the upper-level problem minimizes the average number of misclassified data points over all the cross-validation folds, and the lower-level problems are the $l_1$-loss SVC problems, with each one for each fold in T-fold cross-validation. The resulting bilevel  {optimization} model is then converted to a mathematical program with equilibrium constraints (MPEC). To solve this MPEC, we propose a global relaxation cross-validation algorithm (GR-CV) based on the well-know Sholtes-type global relaxation method (GRM). It is proven to converge to a C-stationary point. Moreover, we prove that the MPEC-tailored version of the Mangasarian-Fromovitz constraint qualification (MFCQ), which is a key property to guarantee the convergence of  {the GRM},  automatically holds at each feasible point of this MPEC. Extensive numerical results verify the efficiency of the proposed approach. In particular, compared with other methods, our algorithm enjoys superior generalization performance over almost all the data sets used in this paper.}



\maketitle
\section{Introduction}
Support vector classification (SVC) is a classical and widely used learning method for classification problems; see, e.g.,  \cite{cortes1995support,Chauhan2019problem,Vapnik2013nature}. In SVC, the selection of hyperparameters, also known as  {hyperparameter} selection, is a critical issue and has been addressed by many researchers both theoretically and practically \cite{chapelle2002choosing,duan2003evaluation,Keerthi2007Efficient,kunapuli2008bilevel1,couellan2015bi,kunapuli2008bilevel,kunapuli2008classification}. While there have been many interesting attempts to use bounds, gradient descent methods or other techniques to identify these hyperparameters \cite{chapelle2002choosing,duan2003evaluation,Keerthi2007Efficient}, one of the most widely used methods is cross-validation (CV). A classical approach for cross-validation is the grid search method \cite{momma2002pattern}, where one needs to define a grid over the hyperparameters of interest, and search for the combination of  {hyperparameters} that minimize the cross-validation error (CV error). Bennett et al. \cite{bennett2006model} emphasize that one of the drawbacks of the grid search approach is that the continuity of the hyperparameter is ignored by the discretization. A formulation of the bilevel optimization model is proposed to choose hyperparameters \cite{bennett2006model,kunapuli2008bilevel1}. Below, we will focus on the bilevel optimization approach which is the most relevant to our work. We refer to \cite{Yu2020hyper,Luo2016review} for a survey of various hyperparameters optimization methods and applications.

In terms of selecting hyperparameters through bilevel optimization, different models and approaches have been considered in the literature. For example, Okuno et al. \cite{okuno2018hyperparameter} propose a bilevel optimization model to select the best hyperparameter for a nonsmooth, possibly nonconvex, $l_{p}$-regularized problem. They then present a  smoothing-type algorithm with convergence analysis to solve this bilevel optimization model. Kunisch and Pock \cite{kunisch2013bilevel} formulate a parameter learning problem for variational image denoising model into a bilevel optimization problem. They {design} a semismooth Newton's method for solving the resulting nonsmooth bilevel optimization problems. Moore et al. \cite{mooregradient} develop an implicit gradient-type algorithm for selecting hyperparameters for linear SVM-type machine learning models which are expressed as bilevel optimization problems. Moore et al. \cite{moore2009nonsmooth} propose a nonsmooth bilevel model to select hyperparameters for support vector regression (SVR) via T-fold cross-validation. They design a proximity control approximation algorithm to solve this bilevel optimization model. Couellan et al. \cite{couellan2015bi} design a bilevel stochastic gradient algorithm for training large scale SVM with automatic selection of the hyperparameter.
We refer to  \cite{colson2007overview,dempe2002foundations, dempebilevelbook} for recent general surveys on bilevel optimization, as well as \cite{mejia2019metaheuristic,zemkoho2021theoretical,fischer2019semismooth,lin2014solving,ye2010new,ochs2016techniques,ochs2015bilevel} for some of the latest algorithms on the subject. Next, we provide a brief overview of the MPEC reformulation of the bilevel optimization problem, which will play a fundamental role in this paper.

For a bilevel program, replacing the lower-level problem by its Karush-Kuhn-Tucker (KKT) conditions will result in a mathematical program with equilibrium constraints (MPEC) \cite{luo1996mathematical}. Therefore, various algorithms for MPECs can be potentially applied to solve bilevel optimization problems, although one might want to pay attention to the fact that both problems are not necessarily equivalent. Bennett and her collaborators did a series of works \cite{bennett2006model,kunapuli2008classification,bennett2008bilevel,kunapuli2008bilevel,kunapuli2008bilevel1} on hyperparameter selection by reformulating a bilevel program into an MPEC. For example, \cite{kunapuli2008classification} considers a bilevel optimization model for selecting many hyperparameters for $l_{1}$-loss SVC problems, in which the upper-level problem has box constraints for the regularization parameter and feature selection. They {reformulate} this bilevel program into an MPEC and  {solve} it by the inexact cross-validation method. Other methods include Newton-type algorithms \cite{wu2015inexact,harder2021reformulation,lee2015global}.

Considering these works, a natural question is whether one can build up a bilevel hyperparameter selection for SVC? If yes, whether there are some special and hidden properties if we transfer the corresponding bilevel optimization problem to its corresponding MPEC and how we can solve it efficiently? This is the main motivation the work in this paper.

In this paper, we consider a bilevel optimization model for selecting the hyperparameter in SVC. This regularization  {hyperparameter} $C$ is selected to minimize the T-fold cross-validated estimation of the out-of-sample misclassification error, which is basically a $0$-$1$ loss function. Therefore, the upper-level problem minimizes the average misclassification error in T-fold cross-validation based on the optimal solution of the lower-level problem (we use the typical $l_{1}$-loss SVC model) for all the possible values of the hyperparameter $C$. There are several challenges  {to design efficient algorithms for such potentially large-scale bilevel programs}. Firstly, the objective function in the upper-level problem is a $0$-$1$ loss function, which is discontinuous and nonconvex. Secondly, the constraints for the upper-level problem involve the optimal solution set of the lower-level problem, i.e., the $l_{1}$-loss SVC optimization model, for which the optimal solution is not explicitly given. To deal with the first challenge, we reformulate the minimization of the $0$-$1$ loss function into a linear optimization problem inspired by the technique in \cite{mangasarian1994misclassification}. We then replace the lower-level problem by its optimality conditions to tackle the second challenge. This therefore leads to an MPEC.

The contributions of the paper are as follows. Firstly, we propose a bilevel optimization model for hyperparameter selection in a binary SVC and study its reformulation as an MPEC.
Secondly, we apply the  {GRM} originating from \cite{scholtes2001convergence} to solve this MPEC, which is shown to converge to a C-stationary point. The resulting algorithm is called the  {GR-CV},  {which is a concrete implementation of the
GRM for selecting the hyperparameter $C$ in SVC.} Thirdly, we prove the MPEC-Mangasarian-Fromovitz constraint qualification (MPEC-MFCQ, for short) property for each feasible point of our MPEC. The MPEC-MFCQ is a key property to guarantee the convergence of the GRM. We show that it automatically holds for our problem thanks to its special structure. Finally, we conduct extensive numerical experiments, which show that our method is very efficient; in particular, it enjoys superior generalization performance over almost all the data sets used in this paper.

The paper is organized as follows. In Section \ref{sec2}, based on T-fold cross-validation for SVC, we introduce a bilevel optimization model to select an optimal hyperparameter for SVC. We also analyze the interesting properties of the lower-level problem. In Section \ref{sec3}, we reformulate the bilevel optimization problem as an MPEC (also known as the KKT reformulation), and apply the GRM for solving the MPEC. In Section \ref{sec4}, we prove that every feasible point of this MPEC satisfies the regularity condition MPEC-MFCQ, which is  a key property to guarantee the convergence of the GRM. In Section \ref{sec5}, we present some computational experiments comparing the { {resulting GR-CV}}  {based on the GRM} with two other ones, which have been used in the literature for a similar purpose; i.e., the inexact cross-validation method (In-CV) and the grid search method (G-S). We conclude the paper in Section \ref{sec6}.\\

\noindent {\bf Notations.} {For $x \in  \mathbb{R}^{n}$, $\|x \|_{0}$ denotes the number of nonzero elements in $x$, while $\| x \|_{1}$ and $\| x \|_{2}$ correspond to the $l_{1}$-norm and $l_{2}$-norm of $x$, respectively.
{ {Also, we will use $x_{+}=((x_{1})_{+},\ \cdots,\ (x_{n})_{+}) \in \mathbb{R}^{n}, $ where
$(x_{i})_{+}=\max(x_{i},\ 0).\ \mid \! \Omega \! \mid $}} denotes the number of elements in the set $\Omega \subset \mathbb{R}^n$.  {We use $\mathbf{1}_{k}$ to denote a vector with elements all ones in $\mathbb{R}^{k}$. $I_{k}$ is the identity matrix in $\mathbb{R}^{k \times k}$, while $e^{k}_{\gamma}$ is the $\gamma$-th row vector of an identity matrix in $\mathbb{R}^{k \times k}$.} The notation $\mathbf{0}_{k \times q}$ represents a zero matrix in $\mathbb{R}^{k \times q}$ and $\mathbf{0}_{k}$ stands for a zero vector in $\mathbb{R}^{k}$. On the other hand, $\mathbf{0}_{(\tau,\ \kappa)}$ will be used for a submatrix of the zero matrix, where $\tau$ is the index set of the rows and $\kappa$ is the index set of the columns.  Similarly to the case of zero matrix, $I_{(\tau,\ \tau)}$ corresponds to a submatrix of an identity matrix indexed by both rows and columns in the set $\tau$. Finally, $\Theta_{(\tau,\ \cdot)}$ represents a submatrix of the matrix $\Theta$, where $\tau$ is the index set of the rows, and $x_{\tau}$ is a subvector of the vector $x$ corresponding to the index set $\tau$.
\section{Bilevel hyperparameter optimization for SVC}\label{sec2}
We start this section by first introducing the problem settings in relation to the T-fold cross-validation for SVC. Subsequently, we present the lower-level problem with some interesting and relevant properties for further analysis in the later parts of the paper. Finally, we introduce the upper-level problem, that is, the  {bilevel optimization model for hyperparameter} selection in SVC.

\subsection{T-fold cross-validation for SVC}
As discussed in the introduction, the most commonly used method for selecting the hyperparameter $C$ is $T$-fold cross-validation. In $T$-fold cross-validation, the data set is split into a subset $\Omega$ with $l_{1}$ points, which is used for cross-validation, and a hold-out test set $\Theta$ with $l_{2}$ points. Here, $\Omega=\{(x_{i},y_{i})\}_{i=1}^{l_{1}} \in \mathbb{R}^{n+1}$, where $x_{i} \in \mathbb{R}^{n}$ denotes a data point and $y_{i}\in \{\pm 1\}$ the corresponding label. For T-fold cross-validation, $\Omega$ is equally partitioned into $T$ pairwise disjoint subsets, one for each fold. The process is executed T iterations. For the $t$-th  iteration ($t=1, \ldots, T$), the $t$-th fold is the validation set $\Omega_{t}$, and the remaining $T-1$ folds make up the training set $\overline{\Omega}_{t}=\Omega \backslash \Omega_{t}$. Therefore, in the $t$-th iteration, the separating hyperplane is trained using the training set $\overline{\Omega}_{t}$, and the validation error is computed on the validation set $\Omega_{t}$.

 Then, the cross-validation error (CV error) is the average of the validation error over all the $T$ iterations. The value of $C$ that gives the best CV error will be selected. Finally, the final classifier is trained using all the data in $\Omega$ and the rescaled optimal $C$. The test error is computed on the test set $\Theta$. Note that the CV error and the test error are the evaluation indices for the classification performance in T-fold cross-validation. As shown in Figure \ref{figCV}, for three-fold cross-validation, the yellow part is the subset $\Omega$ which is used for three-fold cross-validation. In the first iteration, the blue part is the validation set $\Omega_{1}$, and the remaining two folds are the training set $\overline{\Omega}_{1}$. The second and third iterations have similar meanings.
\begin{figure}[h]
	\centering
	\includegraphics[width=0.55\textwidth]{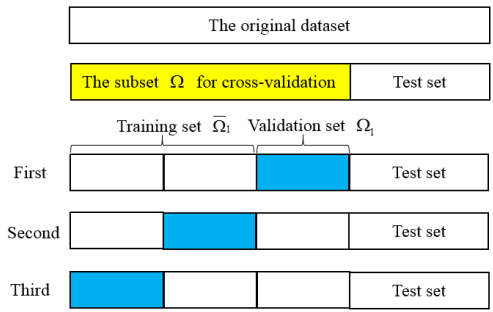}
	\caption{Three-fold cross-validation.}\label{figCV}
\end{figure}

Let $m_{1}$ be the size of the validation set $\Omega_{t}$ and $m_{2}$ the size of the training set $\overline{\Omega}_{t}$. The corresponding index sets for the validation and training sets are $\mathcal{N}_{t}$ and $\overline{\mathcal{N}}_{t}$, respectively. In T-fold cross-validation, there are $T$ validation sets. Therefore, there are totally $T m_{1}$ validation points in T-fold cross-validation. We use the index set
\be \label{eqb1}
Q_{u}:=\{i\ \mid \ i=1,\ 2,\ \cdots,\ Tm_{1}\}
\ee
to represent all the validation points in T-fold cross-validation. Similarly, there are totally $T m_{2}$ training points in T-fold cross-validation. We use the index set
\be \label{eqb2}
Q_{l}:=\{i\ \mid \ i=1,\ 2,\ \cdots,\ Tm_{2}\}
\ee
to represent all the training points in T-fold cross-validation. These two index sets will be used later.

Before analyzing different cases of the data points in the training set and the validation set, we use Figure \ref{figSVM1} to show geometric relationships of different cases in soft-margin support vector classification (without bias term) \cite{cristianini2000introduction,galli2021study}. Specifically, we  consider an $l_1$-loss SVC model as the lower-level problem.

For a sample $(x_{i},y_{i})$, the point $x_{i}$ is referred to as a positive point  {if $y_{i}=1$}; the point $x_{i}$ is referred to as a negative point if $y_{i}=-1$. In Figure \ref{figSVM1}, the plus signs `$+$' are the positive points (i.e., $y_i=1$) and the minus signs `$-$' are the negative ones (i.e., $y_i=-1$). The distance between the hyperplanes $H_{1}: w^{\top} x=1$ and $H_{2}: w^{\top} x=-1$ is called \emph{margin}. The \emph{separating hyperplane} $H$ lies between $H_{1}$ and $H_{2}$. Clearly, the hyperplanes $H_{1}$ and $H_{2}$ are the boundaries of the margin. Therefore, if a positive point lies on the hyperplane $H_{1}$ or a negative point lies on the hyperplane $H_{2}$, we call it lying on the boundary of the margin (indicated by `\ding{172}' in Figure \ref{figSVM1}). If a positive point lies between the separating hyperplane $H$ and the hyperplane $H_{1}$, or a negative point lies between the separating hyperplane $H$ and the hyperplane $H_{2}$, we call it lying between the separating hyperplane $H$ and the boundary of the margin (indicated by `\ding{173}' in Figure \ref{figSVM1}). Similarly, if a positive point lies on the correctly classified side of the hyperplane $H_{1}$, or a negative point lies on the correctly classified side of the hyperplane $H_{2}$, we call it lying on the correctly classified side of the boundary of the margin (indicated by `\ding{174}' in Figure \ref{figSVM1}).

\begin{figure}[h]
	\centering
	\includegraphics[width=0.40\textwidth]{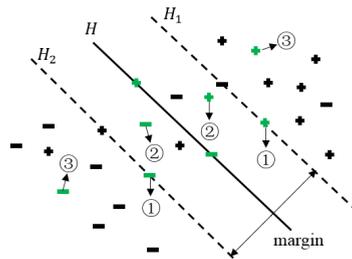}\\
	\caption{Training points in soft-margin support vector machine.}\label{figSVM1}
\end{figure}
Based on Figure \ref{figSVM1}, we have the following observations which address different cases for the data points in the training set.
\begin{proposition} \label{pro6}
Let $w^{t}$ be an optimal solution of the $t$-th lower-level problem ($l_1$-loss SVC model). For $i \in \overline{\mathcal{N}}_{t}$, consider a positive point $x_i$. Then it holds that:
\begin{itemize}
\item [{\rm{(a)}}] $x_i$ satisfies $(w^{t})^{\top} x_{i}<0$ if and only if it lies on the misclassified side of the separating hyperplane $H$, and is therefore misclassified.
\item [{\rm{(b)}}]$x_i$ satisfies $(w^{t})^{\top} x_{i}=0$ if and only if it lies on the separating hyperplane $H$, and is therefore correctly classified.
\item [{\rm{(c)}}] $x_i$ satisfies $0<(w^{t})^{\top} x_{i}<1$ if and only if it lies between the separating hyperplane $H$ and the boundary of the margin; hence, it is correctly classified.
\item [{\rm{(d)}}] $x_i$ satisfies $(w^{t})^{\top} x_{i}=1$ if and only if it lies on the boundary of the margin, and is therefore correctly classified.
\item [{\rm{(e)}}] $x_i$ satisfies $(w^{t})^{\top} x_{i}>1$ if and only if it lies on the correctly classified side of the boundary of the margin, and is therefore correctly classified.
\end{itemize}
\end{proposition}

A result analogous to Proposition \ref{pro6} can be stated for the negative points. In Figure \ref{fig_train}, any point $x_{i} \in \overline{\Omega}_{t}$ in blue is a training point in each case (notation
is the same as in Figure \ref{fig_lower}).
\begin{figure}[htbp]
\centering
\subfigure[(a) $(w^{t})^{\top} x_{i}<0.$]{
\includegraphics[width=0.30\textwidth]{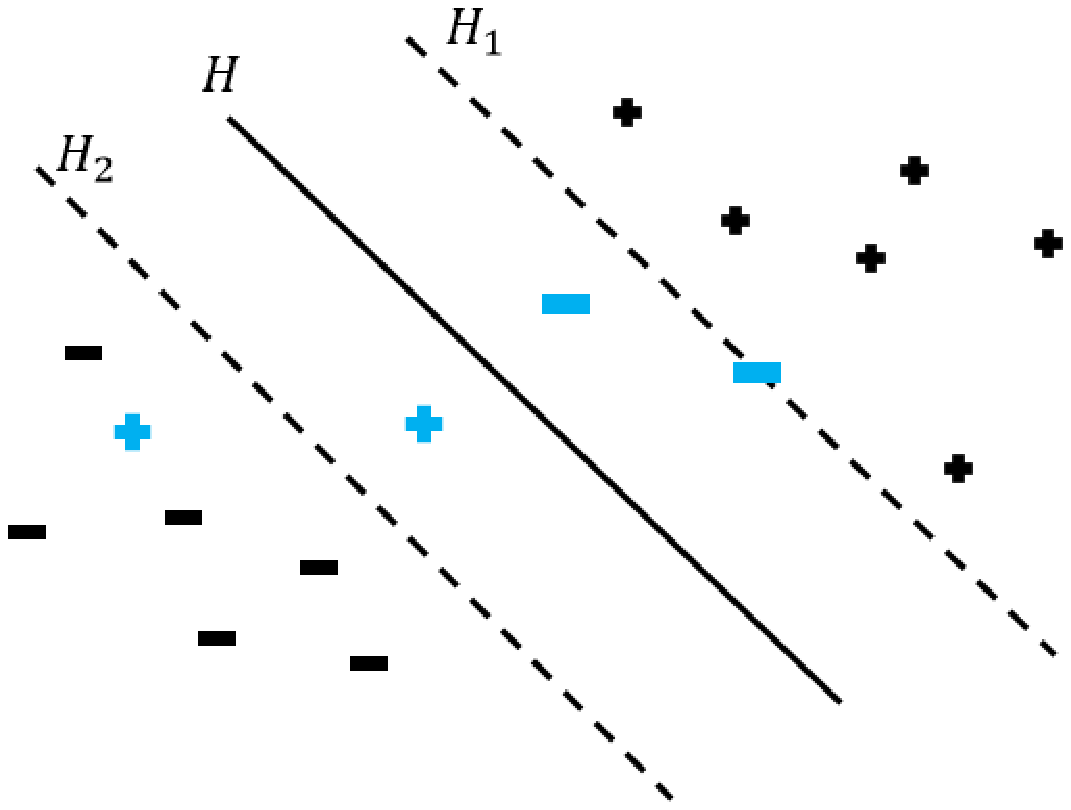}
}
\
\subfigure[(b) $(w^{t})^{\top} x_{i}=0.$]{
\includegraphics[width=0.30\textwidth]{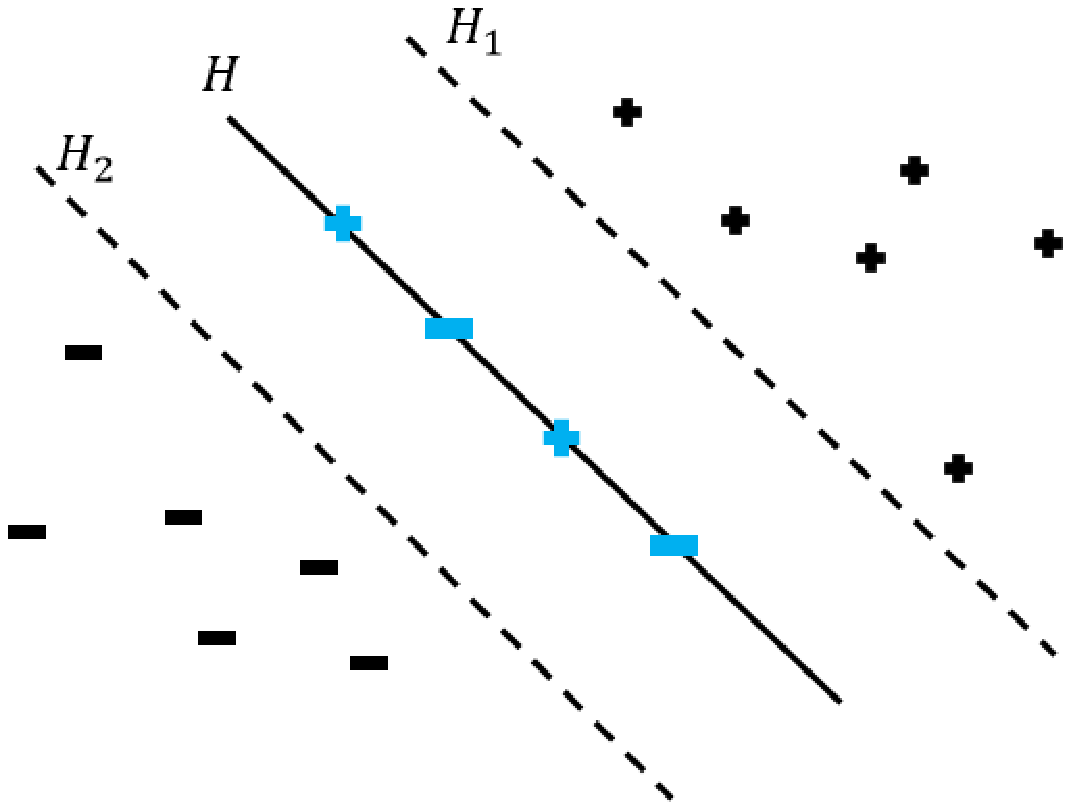}
}
\
\subfigure[(c) $0<(w^{t})^{\top} x_{i}<1.$]{
\includegraphics[width=0.30\textwidth]{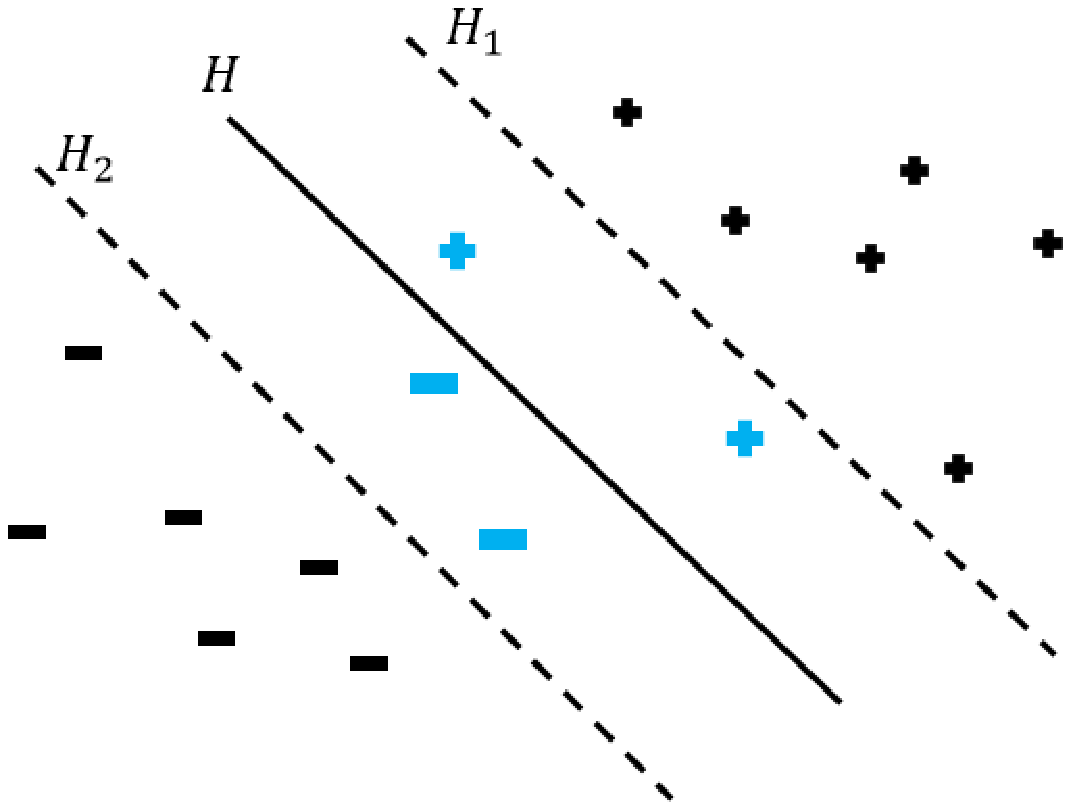}
}
\
\subfigure[(d) $(w^{t})^{\top} x_{i}=1.$]{
\includegraphics[width=0.30\textwidth]{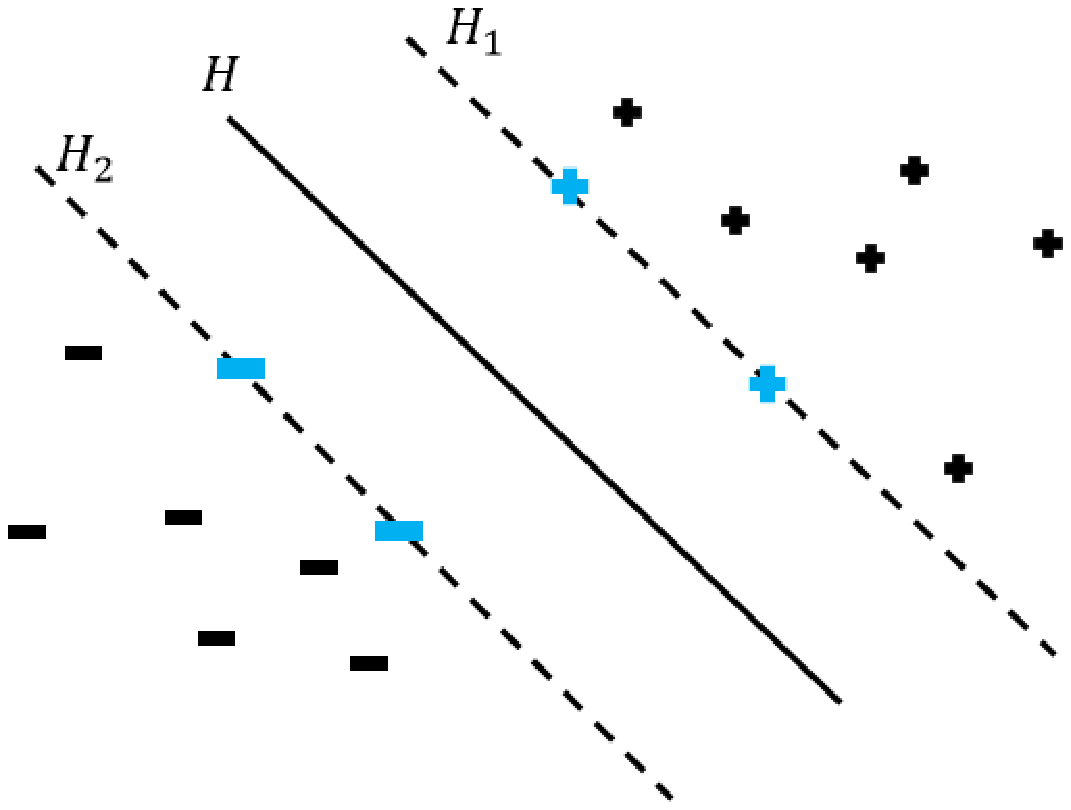}
}
\
\subfigure[(e) $(w^{t})^{\top} x_{i}>1.$]{
\includegraphics[width=0.30\textwidth]{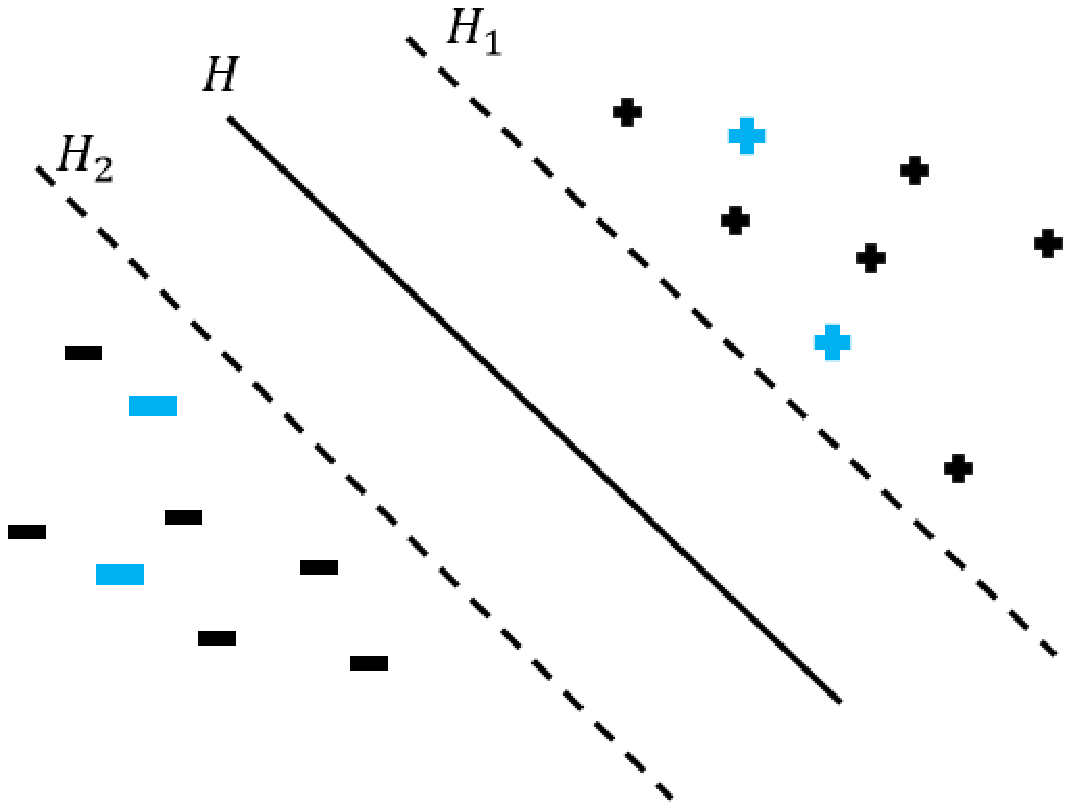}
}
\caption{Each case for different values of $(w^{t})^{\top} x_{i}$ in the training set.}\label{fig_train}
\end{figure}

As for data points in the validation set, we have the following scenarios.
\begin{proposition} \label{pro8}
Let $w^{t}$ be an optimal solution of the $t$-th lower-level problem. For $i \in \mathcal{N}_{t}$, consider a positive point $x_i$. Then it holds that:
\begin{itemize}
\item [{\rm (a)}] $x_i$ satisfies $(w^{t})^{\top} x_{i}<0$ if and only if it lies on the misclassified side of the separating hyperplane $H$, and is therefore misclassified.
\item [{\rm (b)}] $x_i$ satisfies $(w^{t})^{\top} x_{i}=0$ if and only if it lies on the separating hyperplane $H$, and is therefore correctly classified.
\item [{\rm (c)}] $x_i$ satisfies $(w^{t})^{\top} x_{i}>0$ if and only if it lies  on the correctly classified side of the separating hyperplane $H$, and it is hence correctly classified.
\end{itemize}
\end{proposition}
A result analogous to Proposition \ref{pro8} can be stated for the negative points. In Figure \ref{fig_validation}, any point $x_{i} \in \Omega_{t}$ in blue is a validation point in each case (notation is the same as in Figure \ref{fig_upper}). Note that Propositions \ref{pro6} and \ref{pro8} will be used in the proof of Propositions \ref{pro7} and \ref{pro9}.
\begin{figure}[htbp]
\centering
\subfigure[(a) $(w^{t})^{\top} x_{i}<0.$]{
\includegraphics[width=0.25\textwidth]{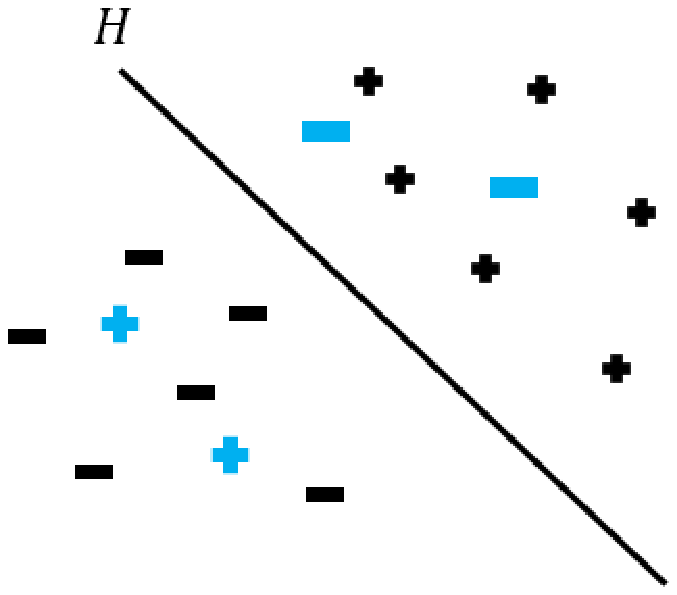}
}
\
\subfigure[(b) $(w^{t})^{\top} x_{i}=0.$]{
\includegraphics[width=0.25\textwidth]{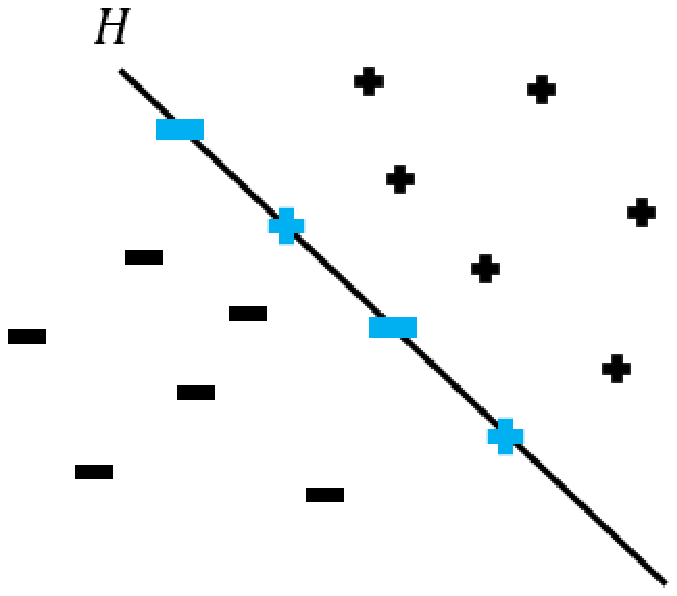}
}
\
\subfigure[(c) $(w^{t})^{\top} x_{i}>0.$]{
\includegraphics[width=0.25\textwidth]{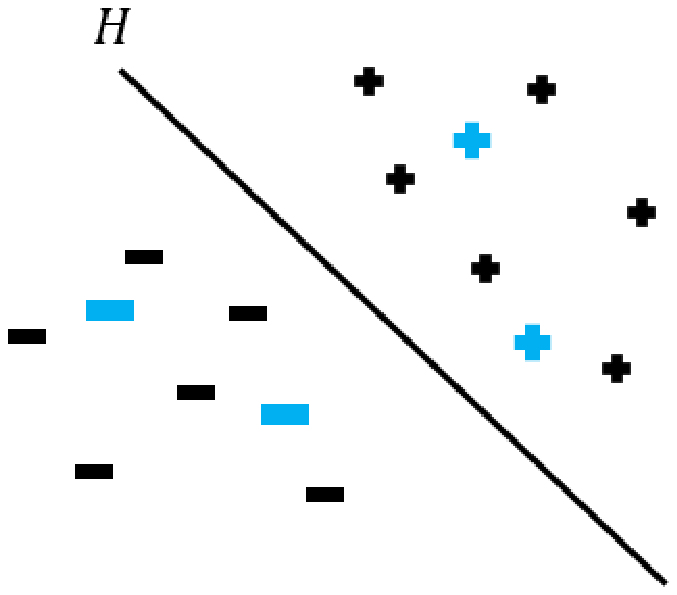}
}
\caption{Each case for different values of $(w^t)^{\top} x_{i}$ in the validation set.}\label{fig_validation}
\end{figure}

\subsection{The lower-level problem}\label{sec2.2}

In this part, we focus on the lower-level problem. That is, given  {hyperparameter} $C$ and the training set $\overline{\Omega}_{t}$, we train the dataset via $l_{1}$-loss SVC model. We will also discuss the properties of the lower-level problem.

\subsubsection{The training model: $l_{1}$-loss SVC}
In T-fold cross-validation, there are T lower-level problems. In the $t$-th lower-level problem, we train the $t$-th fold training data set $\overline{\Omega}_{t}$ by the $l_{1}$-loss SVC model without bias term \cite{galli2021study,mangasarian1994misclassification,hsieh2008dual}. That is, given $C$, we solve the following optimization problem:
$$
  \underset{ w^{t} \in \mathbb{R}^{n}}{\operatorname{min}} \frac{1}{2}\|w^{t}\|_{2}^{2}+ C \sum_{i \in \overline{\mathcal{N}}_{t}} (\mathbf{1}-y_{i}(x_{i}^{\top}w^{t}))_{+}.
$$
A popular reformulation of  {the problem above} is the convex quadratic optimization problem obtained by introducing slack variables $\xi^{t} \in \mathbb{R}^{m_{2}}$:
\be \label{eq2}
\begin{array}{l}
\underset{w^{t} \in \mathbf{R}^{n}, \xi^{t} \in \mathbf{R}^{m_{2}}}\min \quad  \frac{1}{2}\left\|w^{t}\right\|_{2}^{2}+C\sum \limits_{i=1}^{m_{2}} \xi^{t}_{i} \\
\quad \quad \ \hbox{s.t.} \quad \quad \quad
\ B^{t} w^{t} \geq \mathbf{1}-\xi^{t}, \\
\quad \quad \quad \quad \quad \quad \quad \xi^{t} \geq \mathbf{0},
\end{array}
\ee
where, for $t=1, \cdots, T$ and $k=m_{1}+ 1, \cdots, l_{1}$, we have
\[
B^{t}\!=\! \left[\begin{array}{c}
y_{t_{m_{1}+1}} x_{t_{m_{1}+1}}^{\top} \\
\vdots \\
y_{t_{l_{1}}} x_{t_{l_{1}}}^{\top}
\end{array}\right] \! \in \! \mathbb{R}^{m_{2} \times n}, \;\;\, (x_{t_k},y_{t_k}) \in  \overline{\Omega}_{t},
\]
and we use $\xi^{t}_{i}$ to denote the $i$-th element of $\xi^{t} \in \mathbb{R}^{m_{2}}$.

Let $\alpha^{t} \in \mathbb{R}^{m_{2}}$ and $\mu^{t} \in \mathbb{R}^{m_{2}}$ be the multipliers of the constraints in (\ref{eq2}). We can write the KKT conditions for the lower-level problem (\ref{eq2}) as
\begin{subequations} \label{eq_lowerKKT}
\begin{align}
 & \mathbf{0} \leq \alpha^{t} \perp  B^{t} w^{t}-\mathbf{1}+\xi^{t} \geq \mathbf{0}, \label{eq_lowerKKTa} \\ & \mathbf{0} \leq \xi^{t} \perp  \mu^{t} \geq \mathbf{0}, \label{eq_lowerKKTb}  \\
&w^{t}-(B^{t})^{\top} \alpha^{t}=\mathbf{0}, \label{eq_lowerKKTc}\\
& C\mathbf{1}-\alpha^{t}-\mu^{t}=\mathbf{0}, \label{eq_lowerKKTd}
\end{align}
\end{subequations}
where for two vectors $a$ and $b$, writing $\mathbf{0} \leq  a \perp  b \geq \mathbf{0}$ means that we have $a^{\top} b=0,\ a \geq \mathbf{0}$ and $b \geq \mathbf{0}$. Also note that each complementary constraint in (\ref{eq_lowerKKTa}) corresponds to a training point $x_{i}$ with $i \in Q_{l}$ (\ref{eqb2}). Each training point corresponds to a slack variable $\xi^{t}_i$. So we have each complementary constraint in (\ref{eq_lowerKKTb}) corresponds to a training point $x_{i}$ with $i \in Q_{l}$ (\ref{eqb2}). Therefore, there is a one-to-one correspondence between the index set of the training points $Q_{l}$ and the complementary constraints in (\ref{eq_lowerKKTa}) and   (\ref{eq_lowerKKTb}), respectively. This will be used in the definition of some index sets below.


Furthermore, we would like to emphasize the support vectors implied in (\ref{eq_lowerKKT}). From (\ref{eq_lowerKKTc}), the weight vector $w^{t}=(B^{t})^{\top} \alpha^{t}=\sum \limits_{i \in \overline{\mathcal{N}}_{t}} \alpha^{t}_{i} y_{i}x_{i}$. It implies that only the data points $x_{i} \in \overline{\Omega}_{t}$ which correspond to $\alpha^{t}_{i} \neq 0$ are involved. By $\alpha^{t}_{i} \geq 0$ in {\rm(\ref{eq_lowerKKTa})}, it means that only $x_{i} \in \overline{\Omega}_{t}$ with $\alpha^{t}_{i}>0$ are involved. It is for this reason that they are called \emph{support vectors}.
 By eliminating $\mu^{t}$ and $w^{t}$ from the system in \eqref{eq_lowerKKT}, we get the reduced KKT conditions for problem (\ref{eq2}) as follows:
\be \label{eqnew10}
\left\{\begin{aligned}
&\mathbf{0} \leq \alpha^{t} \perp  B^{t} (B^{t})^{\top} \alpha^{t}-\mathbf{1}+\xi^{t} \geq \mathbf{0}, \\
&\mathbf{0} \leq \xi^{t}  \perp  C\mathbf{1}-\alpha^{t} \geq \mathbf{0}.
\end{aligned}\right.
\ee
\subsubsection{Some properties of the lower-level problem}
Let $\alpha \in \mathbb{R}^{Tm_{2}},\ \xi \in \mathbb{R}^{Tm_{2}},\ w \in \mathbb{R}^{Tn}$, and $B \in \mathbb{R}^{Tm_{2} \times Tn}$ be defined by
\be \label{eqvar_lower}
\alpha:=\left[\begin{array}{l} \alpha^{1} \\ \alpha^{2}\\ \vdots \\ \alpha^{T}
\end{array}\right],\, \xi:=\left[\begin{array}{l} \xi^{1} \\ \xi^{2} \\ \vdots \\ \xi^{T}
\end{array}\right],\, w:=\left[\begin{array}{l} w^{1} \\ w^{2} \\ \vdots \\ w^{T}
\end{array}\right],\, \mbox{ and }\, B:=\left[\begin{array}{cccc} B^{1} & \mathbf{0} & \cdots & \mathbf{0} \\ \mathbf{0} &B^{2} & \cdots& \mathbf{0} \\ \vdots &\vdots& \ddots & \vdots \\
 \mathbf{0} & \mathbf{0} & \cdots & B^{T}
\end{array}\right],
\ee
respectively. The KKT conditions in (\ref{eqnew10}) can be decomposed as 
\begin{eqnarray}
\Lambda_{1}&:=&\{i \in Q_{l}\ \mid \ \alpha_{i}=0,\ (B B^{\top} \alpha-\mathbf{1}+\xi)_{i}=0,\ \xi_{i}=0\}, \label{eq6a}\\
\Lambda_{2}&:=&\{i \in Q_{l}\ \mid \ \alpha_{i}=0,\ (B B^{\top} \alpha-\mathbf{1}+\xi)_{i}>0,\ \xi_{i}=0\}, \label{eq6b}\\
\Lambda_{3}&:=&\{i \in Q_{l}\ \mid \ 0< \alpha_{i} \leq C,\ (B B^{\top} \alpha-\mathbf{1}+\xi)_{i}=0,\ \xi_{i}=0\},\label{eqLambda3}\\
\Lambda_{4}&:=&\{i \in Q_{l}\ \mid \ \alpha_{i}=C,\ (B B^{\top} \alpha-\mathbf{1}+\xi)_{i}=0,\ 0<\xi_{i}<1\}, \label{eq6e}\\
\Lambda_{5}&:=&\{i \in Q_{l}\ \mid \ \alpha_{i}=C,\ (B B^{\top} \alpha-\mathbf{1}+\xi)_{i}=0,\ \xi_{i}=1\}, \label{eq6f}\\
\Lambda_{6}&:=&\{i \in Q_{l}\ \mid \ \alpha_{i}=C,\ (B B^{\top} \alpha-\mathbf{1}+\xi)_{i}=0,\ \xi_{i}>1\}.\label{eq6g}
\end{eqnarray}
Obviously, the intersection of any pair of these index sets $\Lambda_{i}$ for $i=1, \ldots, 6$ is empty. An illustrative representation of data points corresponding to these index sets is given in Figure \ref{fig_lower}.
\begin{figure}[htbp]
\centering
\subfigure[(a) Points with indices in $\Lambda_{1}$]{
\includegraphics[width=0.28\textwidth]{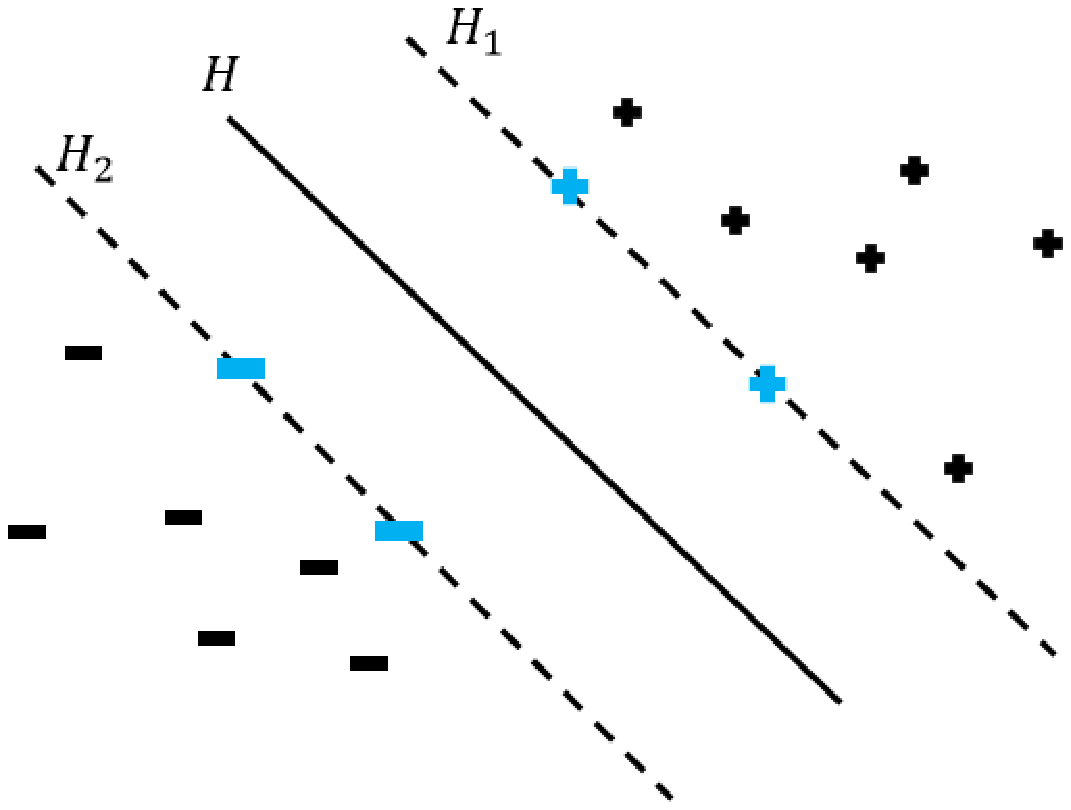}
}
\
\subfigure[(b) Points with indices in $\Lambda_{2}$]{
\includegraphics[width=0.28\textwidth]{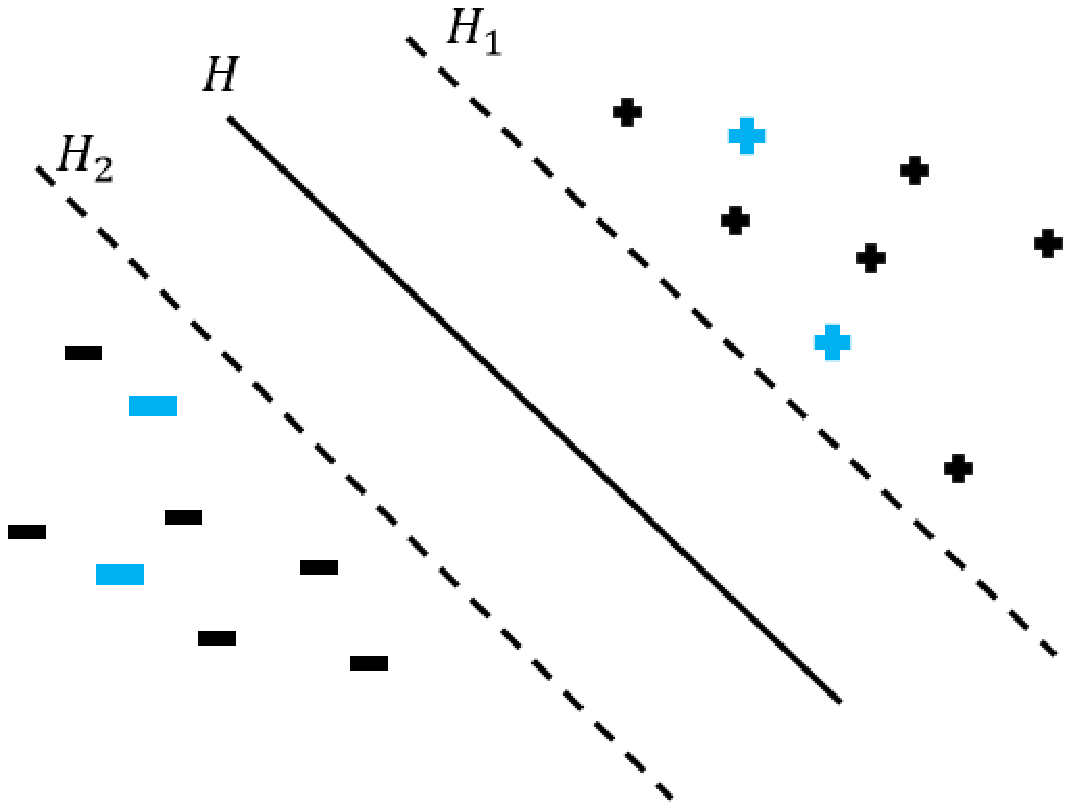}
}
\
\subfigure[(c) Points with indices in $\Lambda_{3}$]{
\includegraphics[width=0.28\textwidth]{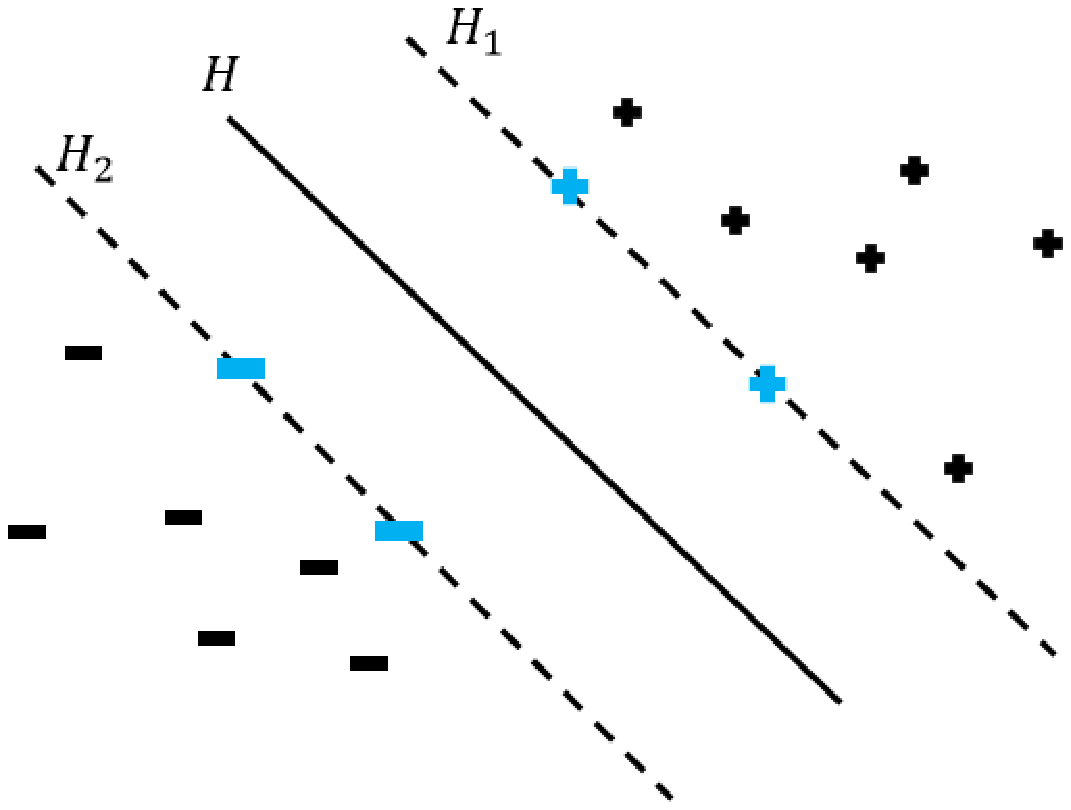}
}
\
\subfigure[(d) Points with indices in $\Lambda_{4}$]{
\includegraphics[width=0.28\textwidth]{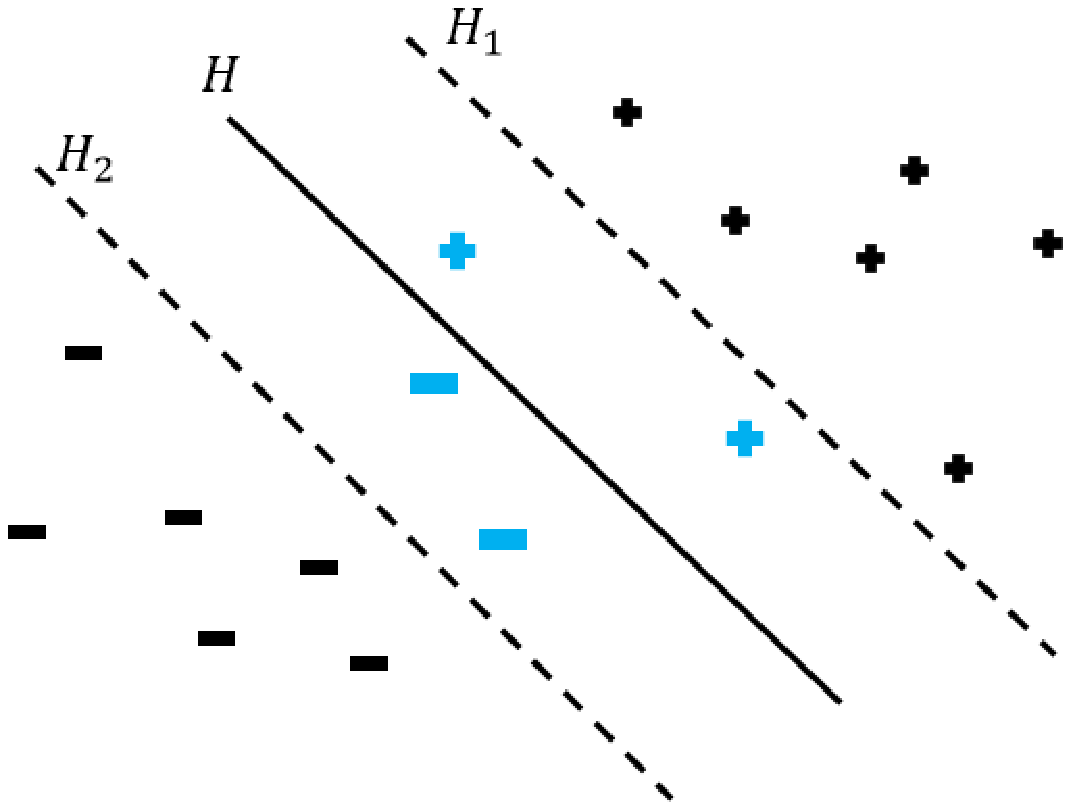}
}
\
\subfigure[(e) Points with indices in $\Lambda_{5}$]{
\includegraphics[width=0.28\textwidth]{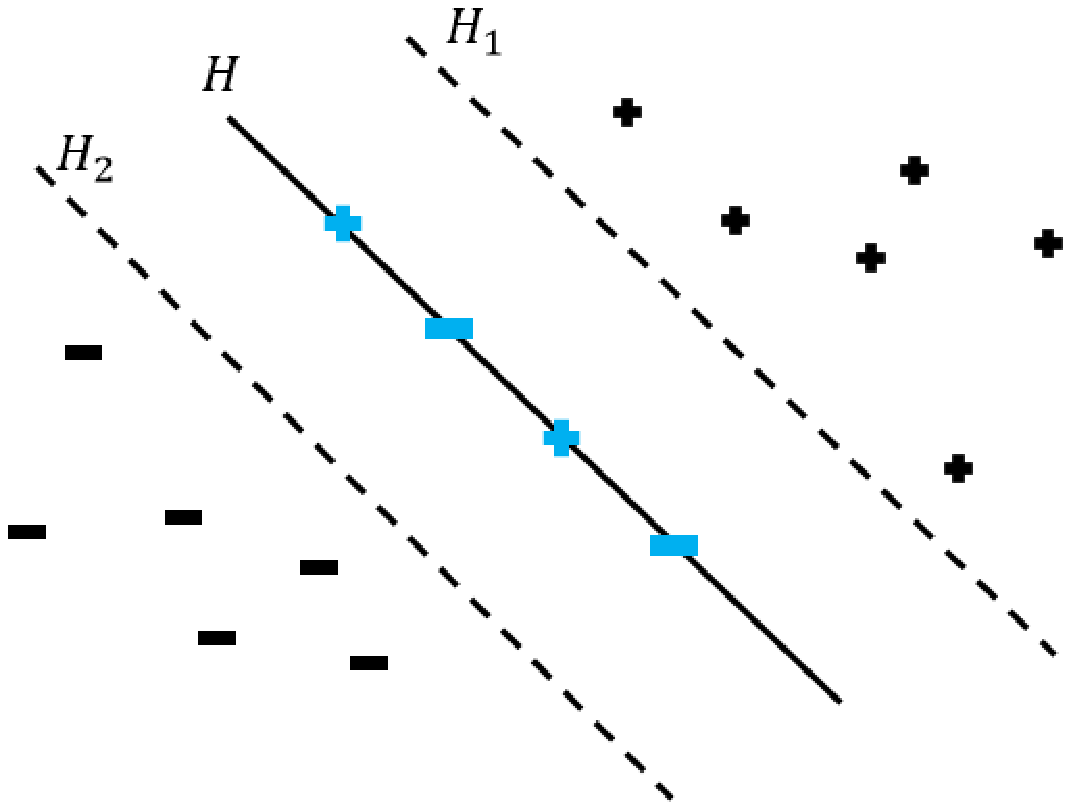}
}
\
\subfigure[(f) Points with indices in $\Lambda_{6}$]{
\includegraphics[width=0.28\textwidth]{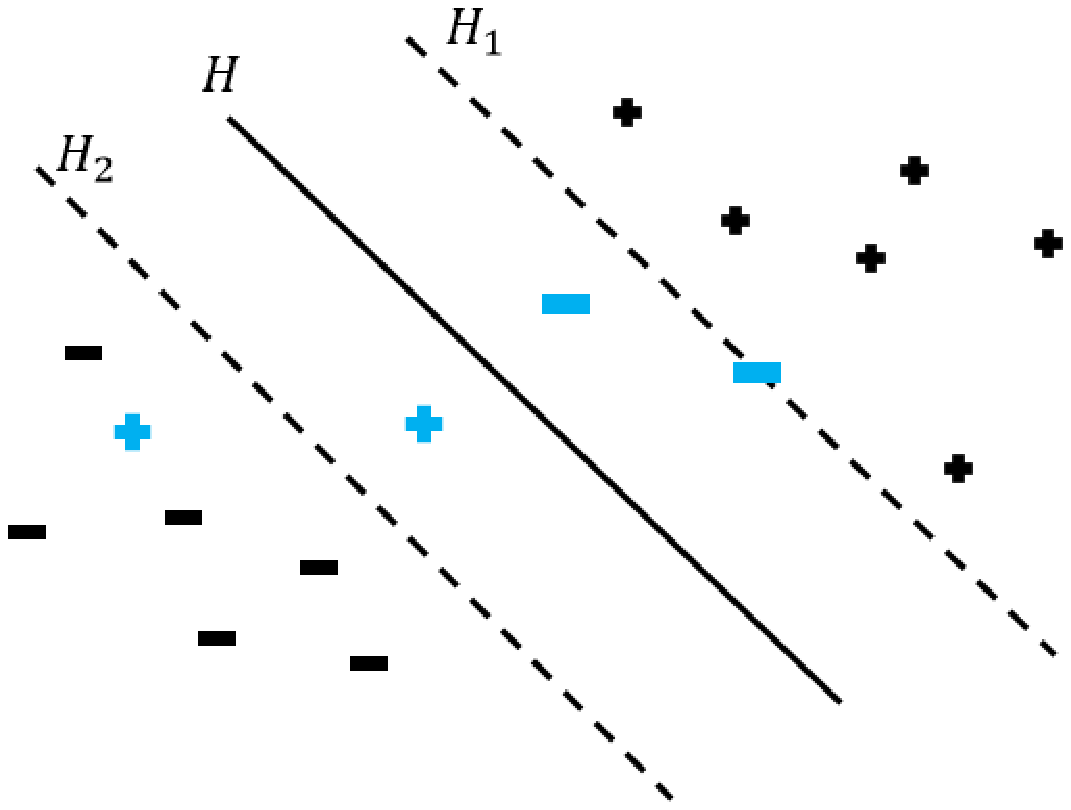}
}
\caption{Representation of points with index sets  $\Lambda_j$, $j=1, \ldots, 6$}\label{fig_lower}
\end{figure}
\begin{proposition}\label{pro7}
Considering the training points corresponding to $Q_{l}$ in {\rm(\ref{eqb2})}, let $(\alpha,\, \xi)$ satisfy the conditions in {\rm(\ref{eqnew10})}. Then, the following statements hold true:
\begin{itemize}
\item [{\rm  {(a)}}] The points $\{x_i\}_{i\in \Lambda_{1}}$ lie on the boundary of the margin; they are correctly classified points, but are not support vectors.
\item [{\rm  {(b)}}]  The points $\{x_i\}_{i\in \Lambda_{2}}$ lie on the correctly classified side of the boundary of the margin; they are correctly classified points, but are not support vectors.
\item [{\rm  {(c)}}]  The points $\{x_i\}_{i\in \Lambda_{3}}$ lie on the boundary of the margin; they are correctly classified points and are support vectors.
\item [{\rm  {(d)}}]  The points $\{x_i\}_{i\in \Lambda_{4}}$ lie between the separating hyperplane $H$ and the boundary of the margin; they are correctly classified therefore support vectors.
\item [{\rm  {(e)}}]  The points $\{x_i\}_{i\in \Lambda_{5}}$ lie on the separating hyperplane $H$; they are correctly classified points and are support vectors.
\item [{\rm  {(f)}}] The points $\{x_i\}_{i\in \Lambda_{6}}$ lie on the misclassified side of the separating hyperplane $H$; they are misclassified points and are support vectors.
\end{itemize}
\end{proposition}
\begin{proof}
We take positive points for example. The same analysis can be applied to negative ones. Since $w=B^{\top} \alpha$ in (\ref{eq_lowerKKTc}), we get $(B B^{\top} \alpha -\mathbf{1}+\xi)_{i}=(Bw -\mathbf{1}+\xi)_{i}$.
\begin{itemize}
\item[{\rm (a)}] For the points $\{x_i\}_{i\in \Lambda_{1}}$, since $\xi_{i}=0$ in {\rm (\ref{eq6a})}, we have $(Bw-\mathbf{1}+\xi)_{i}=(Bw-\mathbf{1})_{i}=0$, that is, $y_i(w^{\top} x_i) -1=0$. For a positive point, $y_{i}=1$, it implies that $w^{\top} x_i=1$. It corresponds to (d) in Proposition \ref{pro6}. Therefore, it means that the point $x_i$ lies on the boundary of the margin. It is correctly classified, and it is not a support vector, since $\alpha_{i}=0$.
\item[{\rm (b)}] For the points $\{x_i\}_{i\in \Lambda_{2}}$, since $\xi_{i}=0$ in {\rm (\ref{eq6b})}, we have $(Bw-\mathbf{1}+\xi)_{i}=(Bw-\mathbf{1})_{i}>0$, that is, $y_i(w^{\top} x_i) -1>0$. For a positive point, $y_{i}=1$, it implies that $w^{\top} x_i>1$. It corresponds to (e) in Proposition \ref{pro6}. Therefore, it means that the point $x_i$ lies on the correctly classified side of the boundary of the margin. It is correctly classified, but not a support vector, as $\alpha_{i}=0$.
\item[{\rm (c)}] For the points $\{x_i\}_{i\in \Lambda_{3}}$, since $\xi_{i}=0$ in {\rm (\ref{eqLambda3})}, we have $(Bw-\mathbf{1}+\xi)_{i}=(Bw-\mathbf{1})_{i}=0$, that is, $y_i(w^{\top} x_i) -1=0$. For a positive point, $y_{i}=1$, it implies that $w^{\top} x_i=1$. It corresponds to (d) in Proposition \ref{pro6}. Therefore, it means that the point $x_i$ lies on the boundary of the margin. It is correctly classified, and it is a support vector, since $\alpha_{i} >0$.
\item[{\rm (d)}] For the points $\{x_i\}_{i\in \Lambda_{4}}$, since $0<\xi_{i}<1$ in {\rm (\ref{eq6e})}, we have $0<(Bw)_{i}<1$, that is, $0<y_i(w^{\top} x_i)<1$. For a positive point, $y_{i}=1$, it implies that $0<w^{\top} x_i<1$. It corresponds to (c) in Proposition \ref{pro6}. Therefore, $x_i$ lies between the separating hyperplane $H$ and the boundary of the margin. It is correctly classified, and it is a support vector, since $\alpha_{i} >0$.
\item[{\rm (e)}] For the points $\{x_i\}_{i\in \Lambda_{5}}$, since $\xi_{i}=1$ in {\rm (\ref{eq6f})}, we have $(Bw-\mathbf{1}+\xi)_{i}=(Bw)_{i}=0$, that is, $y_i(w^{\top} x_i)=0$. For a positive point, $y_{i}=1$, it implies that $w^{\top} x_i=0$. It corresponds to (b) in Proposition \ref{pro6}. Therefore, it means that the point $x_i$ lies on the separating hyperplane $H$. It is correctly classified, and it is a support vector, since $\alpha_{i} >0$.
\item[{\rm (f)}] For the points $\{x_i\}_{i\in \Lambda_{6}}$, since $\xi_{i}>1$ in {\rm (\ref{eq6g})}, we have $(Bw)_{i}<0$, that is, $y_i(w^{\top} x_i)<0$. For a positive point, $y_{i}=1$, it implies that $w^{\top} x_i<0$. It corresponds to (a) in Proposition \ref{pro6}. Therefore, it means that the point $x_i$ lies on the misclassified side of the separating hyperplane $H$. It is misclassified, and it is a support vector, since $\alpha_{i} >0$.
\end{itemize}
\end{proof}

\subsection{The upper-level problem}
In this part, we introduce the upper-level problem, that is, the bilevel  {optimization} model for  {hyperparameter} selection in SVC under the settings of T-fold cross-validation. Note that the aim of the upper-level problem is to minimize the T-fold cross-validation error (CV error) measured on the validation sets based on the optimal solutions of the lower-level problems. Specifically, the basic bilevel  {optimization} model for selecting the hyperparameter $C$ in SVC is formulated as
\be \label{eq31}
\begin{aligned}
\min \limits _{C \in \mathbb{R},\ w^{t} \in \mathbb{R}^{n},\  j=1, \cdots,T} & \frac{1}{T}\sum_{t=1}^{T} \frac{1}{m_{1}} \sum_{i \in \mathcal{N}_{t}} \| \left( -y_{i} \left( x_{i}^{\top} w^{t}\right) \right)_{+}\|_{0}\\
\hbox{s.t.} \quad \quad \quad \ & C \geq 0,\\
 \quad \quad \quad \ & \text{and for} \quad t=1,\cdots,T: \\
\quad \quad & w^{t} \in \underset{ w \in \mathbb{R}^{n}}{\operatorname{argmin}} \left\{\frac{1}{2}\|w\|_{2}^{2}+ C \sum_{i \in \overline{\mathcal{N}}_{t}}  \left( \mathbf{1}-y_{i}\left(x_{i}^{\top}w\right)\right)_{+}
\right\}.
\end{aligned}
\ee
Here, the expression $\sum_{i \in \mathcal{N}_{t}}\| \left( -y_{i} \left( x_{i}^{\top} w^{t}\right) \right)_{+}\|_{0}$ basically counts the number of data points that are misclassified in the validation set $\Omega_{t}$, while the outer summation (i.e., the objective function in (\ref{eq31})) averages the misclassification error over all the folds.

Problem (\ref{eq31}) can be equivalently written in the matrix form as follows
\be \label{eq1}
\setlength{\abovedisplayskip}{3pt}
\begin{aligned}
\min \limits _{C \in \mathbb{R},\ w^{t} \in \mathbb{R}^{n},\ j=1,\cdots,T} & \frac{1}{T} \sum_{t=1}^{T} \frac{1}{m_{1}} \|  \left( -A^{t} w^{t} \right)_{+} \|_{0}\\
\hbox{s.t.}  \quad \quad \quad \ & C \geq 0,\\
\quad \quad \quad \ & \text{and for} \quad t=1,\cdots,T: \\
\quad \quad & w^{t} \in \underset{w \in \mathbb{R}^{n}}{\operatorname{argmin}} \left\{\frac{1}{2}\|w\|_{2}^{2}+ C \|  \left(\mathbf{1}-B^{t} w\right)_{+} \|_{1} \right\},
\end{aligned}
\ee
where, for $t=1,\ \cdots,\ T$ and $k=1,\ \cdots,\ m_{1}$, we have
$$
 A^{t}=\left[\begin{array}{c}
y_{t_1} x_{t_1}^{\top} \\
\vdots \\
y_{t_{m_{1}}} x_{t_{m_{1}}}^{\top}
\end{array}\right] \in \mathbb{R}^{m_{1} \times n} \;\, \mbox{ and }\;\, (x_{t_k},y_{t_k}) \in \Omega_{t}.
\setlength{\abovedisplayskip}{3pt}
$$
\begin{remark}
Compared with the model in {\rm \cite{kunapuli2008classification}}, for example, we consider a simpler bilevel  {optimization} model, without the box constraints in the upper-level problem,
\end{remark}
\section{Single-level reformulation and method}\label{sec3}

 In this section, we first reformulate the bilevel  {optimization} problem as a single-level optimization problem, precisely, we write the problem as an MPEC. Then we present the properties of this single-level problem. Finally, we discuss the  {GRM} to solve the MPEC problem.
\subsection{The MPEC reformulation}
Recall the upper-level objective function in (\ref{eq1}) is a measure
of misclassification error based on the $T$ out-of-sample validation sets, which we minimize. The
measure used here is the classical cross-validation error (CV error) for classification, the average number
of the data points misclassified. It is clear that $\|(\cdot)_{+}\|_{0}$ is discontinuous and nonconvex. However, the function $\|(\cdot)_{+}\|_{0}$ can be characterized as the sum of all elements of the solution to the following linear optimization problem as demonstrated in \cite{mangasarian1994misclassification}, i.e.,
\[
\setlength{\abovedisplayskip}{1.5pt}
 \| r_{+} \|_{0}=\left\{\sum \limits_{i=1}^{m_1} \zeta_{i}:\; \zeta =\underset{u}{\operatorname{argmin}} \left\{ -u^{\top} r \ :\ \mathbf{0} \leq u \leq \mathbf{1} \right\}\right\}.
\]
Therefore, for each fold, $ \| \left(-A^{t} w^{t} \right)_{+} \|_{0}$ is the sum of all elements of the solution to the following linear optimization problem:
 \be \label{eq5}
\begin{aligned}
& \min \limits_{\zeta^{t} \in \mathbb{R}^{m_{1}}}  \ -(\zeta^{t})^{\top} (-A^{t} w^{t}) \\
&\ \ \begin{array}{ll} \hbox{s.t.} \ \ \zeta^{t} \geq \mathbf{0}, \\
 \quad \quad \mathbf{1}- \zeta^{t} \geq \mathbf{0}.
\end{array}
\end{aligned}
\ee
This implies that $ \| \left(-A^{t} w^{t} \right)_{+} \|_{0}= \sum \limits_{i=1}^{m_{1}}\zeta^{t}_{i}$ in each fold. According to Proposition \ref{pro8}, there are two cases for the validation points:
\begin{itemize}
\item [1.] If the validation point $(x_{i},y_{i}) \in \Omega_{t}$ is misclassified, then $y_{i} \left( x_{i}^{\top} w^{t}\right)<0 $. That is, $ (-A^{t} w^{t})_{i} >0$, which corresponds to $(\left( -A^{t} w^{t}\right)_{+})_{i}>0$.
\item [2.] If the validation point $(x_{i},y_{i}) \in \Omega_{t}$ is correctly classified, we have $y_{i} \left( x_{i}^{\top} w^{t}\right)\geq 0$. There are two cases.
    Firstly, $x_{i}$ lies on the separating hyperplane $H$, that is,  $y_{i} \left( x_{i}^{\top} w^{t}\right)=0$. For $y_{i}=1$, there is $(-A^{t} w^{t})_{i} =0$, which corresponds to $(\left( -A^{t} w^{t}\right)_{+})_{i}=0$.
    Secondly, $x_{i}$ lies on the correctly classified side of the separating hyperplane $H$, that is, $y_{i} \left( x_{i}^{\top} w^{t}\right)>0$. For $y_{i}=1$, there is $ (-A^{t} w^{t})_{i} <0$, which corresponds to $(\left( -A^{t} w^{t}\right)_{+})_{i}=0$.
\end{itemize}
Combining with $ \| \left(-A^{t} w^{t} \right)_{+} \|_{0}= \sum \limits_{i=1}^{m_{1}}\zeta^{t}_{i}$, it means that
  \be \label{eq57}
  \setlength{\abovedisplayskip}{1.5pt}
  \zeta^{t}_{i}=\left\{\begin{array}{ll}
1, \quad \quad \text{if} \quad (x_{i},y_{i}) \in \Omega_{t} \ \text{is misclassified}, \\
0, \quad \quad \text{if} \quad (x_{i},y_{i}) \in \Omega_{t}\ \text{is correctly classified},
\end{array}\right.
\setlength{\belowdisplayskip}{1.5pt}
  \ee
  where $\zeta^{t}_{i}$ is the $i$-th element of $\zeta^{t}$ in the $t$-th fold.

The linear programs (LPs) (\ref{eq5}), for $t= 1, \ldots, T$, are inserted into the bilevel optimization problem in order to recast the discontinuous upper-level objective function into a continuous one. Each LP in the form of (\ref{eq5}) can also be replaced with its KKT conditions as follows
$$
\left\{\begin{aligned}
& \begin{array}{l}
\mathbf{0} \leq \zeta^{t} \perp  \lambda^{t} \geq \mathbf{0}, \\ \mathbf{0} \leq z^{t}  \perp  \mathbf{1}-\zeta^{t} \geq \mathbf{0}, \\
A^{t} w^{t}-\lambda^{t}+z^{t}=\mathbf{0}.  \end{array}
\end{aligned}\right.
$$
By eliminating $\lambda^{t}$ and $w^{t}$ with $w^{t}=(B^{t})^{\top} \alpha^{t}$ in (\ref{eq_lowerKKTc}), we get the reduced KKT conditions for problem (\ref{eq5}) with
\begin{subequations} \label{eq8}
\begin{align}
&\mathbf{0} \leq \zeta^{t} \perp  A^{t} (B^{t})^{\top} \alpha^{t}+z^{t} \geq \mathbf{0}, \label{eq8a} \\
& \mathbf{0} \leq z^{t}  \perp  \mathbf{1}-\zeta^{t} \geq \mathbf{0}. \label{eq8b}
\end{align}
\end{subequations}
Note that each complementary constraint in (\ref{eq8a}) corresponds to a validation point $x_{i}$ with $i \in Q_{u}$ (\ref{eqb1}). Each validation point corresponds to a variable $\zeta^{t}_i$. So we have each complementary constraint in (\ref{eq8b}) corresponds to a validation point $x_{i}$ with $i \in Q_{u}$ (\ref{eqb1}). Therefore, there is a one-to-one correspondence between the index set of the validation points $Q_{u}$ and the complementary constraints in (\ref{eq8a}) and (\ref{eq8b}), respectively. 

Combining the systems in (\ref{eqnew10}) and (\ref{eq8}), we can transform the bilevel optimization problem (\ref{eq1}) into the single-level optimization problem
\be \label{eq9}
\setlength{\abovedisplayskip}{3pt}
\begin{aligned}
\min_{\begin{subarray}{l}
 \quad \quad \ C \in \mathbb{R} \\
\zeta^{t} \in \mathbb{R}^{m_{1}},\ z^{t} \in \mathbb{R}^{m_{1}}\\
\alpha^{t} \in \mathbb{R}^{m_{2}},\ \xi^{t} \in \mathbb{R}^{m_{2}}\\
\quad \ t=1,\ \cdots,\ T
 \end{subarray}} & \frac{1}{T m_{1}}  \sum_{i=1}^{m_{1}} \sum_{t=1}^{T} \zeta^{t}_{i}\\
 \quad \quad \quad \hbox{s.t.} \quad \quad \ & C \geq 0, \\
 \quad \quad \quad & \text{and for} \quad t=1,\ \cdots,\ T: \\
\quad \quad & \left\{\begin{array}{ll} \mathbf{0} \leq  \zeta^{t} \perp A^{t} (B^{t})^{\top} \alpha^{t}+z^{t} \geq \mathbf{0}, \\
 \mathbf{0} \leq  z^{t}  \perp  \mathbf{1}-\zeta^{t} \geq \mathbf{0},  \\
 \mathbf{0} \leq  \alpha^{t} \perp  B^{t} (B^{t})^{\top} \alpha^{t}- \mathbf{1}+\xi^{t} \geq \mathbf{0},\\
 \mathbf{0} \leq  \xi^{t}  \perp  C \mathbf{1}-\alpha^{t} \geq \mathbf{0}.
  \end{array}\right.
\end{aligned}
\ee
Note that the constraints $C\mathbf{1}-\alpha^{t} \geq \mathbf{0}$ and $\alpha^{t} \geq \mathbf{0}$ imply $C \geq 0$. Therefore, we remove the redundant constraint $C \geq 0$, and get an equivalent form of the problem above as follows
\be \label{eq54}
\setlength{\abovedisplayskip}{3pt}
\begin{aligned}
\min_{\begin{subarray}{l}
 \quad \quad \ C \in \mathbb{R} \\
\zeta^{t} \in \mathbb{R}^{m_{1}},\ z^{t} \in \mathbb{R}^{m_{1}}\\
\alpha^{t} \in \mathbb{R}^{m_{2}},\ \xi^{t} \in \mathbb{R}^{m_{2}}\\
\quad \ t=1,\ \cdots,\ T
 \end{subarray}} & \frac{1}{T m_{1}}  \sum_{i=1}^{m_{1}} \sum_{t=1}^{T} \zeta^{t}_{i}\\
\quad \quad \quad \hbox{s.t.} \quad \quad \ & \text{for} \quad t=1,\ \cdots,\ T: \\
\quad \quad & \left\{\begin{array}{ll} \mathbf{0} \leq  \zeta^{t} \perp A^{t} (B^{t})^{\top} \alpha^{t}+z^{t} \geq \mathbf{0}, \\
 \mathbf{0} \leq  z^{t}  \perp  \mathbf{1}-\zeta^{t} \geq \mathbf{0},  \\
 \mathbf{0} \leq  \alpha^{t} \perp  B^{t} (B^{t})^{\top} \alpha^{t}- \mathbf{1}+\xi^{t} \geq \mathbf{0},\\
 \mathbf{0} \leq  \xi^{t}  \perp  C \mathbf{1}-\alpha^{t} \geq \mathbf{0}.
  \end{array}\right.
\end{aligned}
\ee
 The presence of the equilibrium constraints makes problem (\ref{eq54}) an instance of an MPEC, which is sometimes labelled as an extension of a bilevel optimization problem \cite{luo1996mathematical}. The optimal hyperparameter is now well defined as a global optimal solution to the MPEC \cite{lee2015global}.
Now that we have transformed a bilevel classification model into the MPEC (\ref{eq54}), we can rewrite it in a compact form
\be \label{eq15}
\begin{aligned}
& \min_{\begin{subarray}{c}
C \in \mathbb{R} \\
\zeta \in \mathbb{R}^{Tm_{1}},\ z \in \mathbb{R}^{Tm_{1}}\\
\alpha \in \mathbb{R}^{Tm_{2}},\ \xi \in \mathbb{R}^{Tm_{2}}
\end{subarray}}  \frac{1}{Tm_{1}} \mathbf{1}^{\top} \zeta\\
& \quad  \quad \ \begin{array}{ll}\hbox{s.t.}\ \ \mathbf{0} \leq  \zeta \perp  A B^{\top} \alpha+z  \geq \mathbf{0}, \\
\quad \quad \mathbf{0} \leq  z  \perp  \mathbf{1}-\zeta  \geq \mathbf{0},  \\
\quad \quad \mathbf{0} \leq  \alpha  \perp  B B^{\top} \alpha- \mathbf{1}+\xi \geq \mathbf{0},\\
\quad \quad \mathbf{0} \leq  \xi  \perp  C \mathbf{1}-\alpha  \geq \mathbf{0},
  \end{array}
\end{aligned}
\ee
where $\alpha$, $\xi$ and $B$ are defined in (\ref{eqvar_lower}), while
\[
 \setlength{\abovedisplayskip}{3pt}
 \begin{array}{c}
 \zeta\!:=\!\left[\begin{array}{l} \zeta^{1} \\ \zeta^{2} \\ \vdots \\ \zeta^{T}
\end{array}\right] \! \in \! \mathbb{R}^{Tm_{1}},\ z \!:=\!\left[\begin{array}{l} z^{1} \\ z^{2} \\ \vdots \\ z^{T}
\end{array}\right] \! \in \! \mathbb{R}^{Tm_{1}},\
 A:=\left[\begin{array}{cccc} A^{1} & \mathbf{0} & \cdots & \mathbf{0} \\ \mathbf{0} &A^{2} & \cdots& \mathbf{0} \\ \vdots &\vdots& \ddots & \vdots \\
 \mathbf{0} & \mathbf{0} & \cdots & A^{T}
\end{array}\right] \in \mathbb{R}^{Tm_{1} \times {Tn}},
 \end{array}
\]
From now on, all our analysis is going to be based on the model in (\ref{eq15}).

\subsection{Some properties of the MPEC reformulation}
Observe that the last two constraints of problem (\ref{eq15}) correspond to the complementarity systems that are part of the KKT conditions of the lower-level problem in (\ref{eqnew10}). As the latter conditions are carefully studied in Proposition \ref{pro7}, it remains to analyze the first two complementarity systems describing the feasible set of problem  (\ref{eq15}). Hence, we partition them as follows
\begin{eqnarray}
\Psi_{1}&:=&\left\{i \in Q_{u}\ \mid \ 0\leq \zeta_{i}<1,\ (A B^{\top} \alpha+z)_{i}=0,\ z_{i}=0\right\}, \label{eqPsi1}\\
\Psi_{2}&:=&\left\{i \in Q_{u}\ \mid \ \zeta_{i}=0,\ (A B^{\top} \alpha+z)_{i}>0,\ z_{i}=0\right\}, \label{eq7b}\\
\Psi_{3}&:=&\left\{i \in Q_{u}\ \mid \ \zeta_{i}=1,\ (A B^{\top} \alpha+z)_{i}=0,\ z_{i}>0\right\}. \label{eq7d}
\end{eqnarray}
Similarly to (\ref{eq6a})--(\ref{eq6g}), the intersection of any pair of the index sets $\Psi_{j}$ for $j=1, \,2, \, 3$ is empty. In the same vein, an illustrative representation of data points corresponding to the  index sets $\Psi_{j}$ for $j=1, \,2, \, 3$ is given in Figure \ref{fig_upper}.
\begin{figure}[htbp]
\centering
\subfigure[(a) Points with indices in $\Psi_{1}$]{
\includegraphics[width=0.28\textwidth]{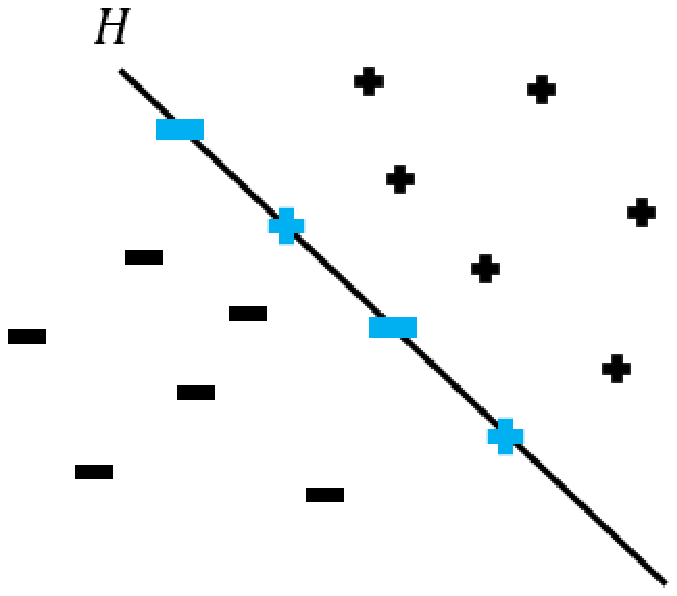}
}
\
\subfigure[(b) Points with indices in $\Psi_{2}$]{
\includegraphics[width=0.28\textwidth]{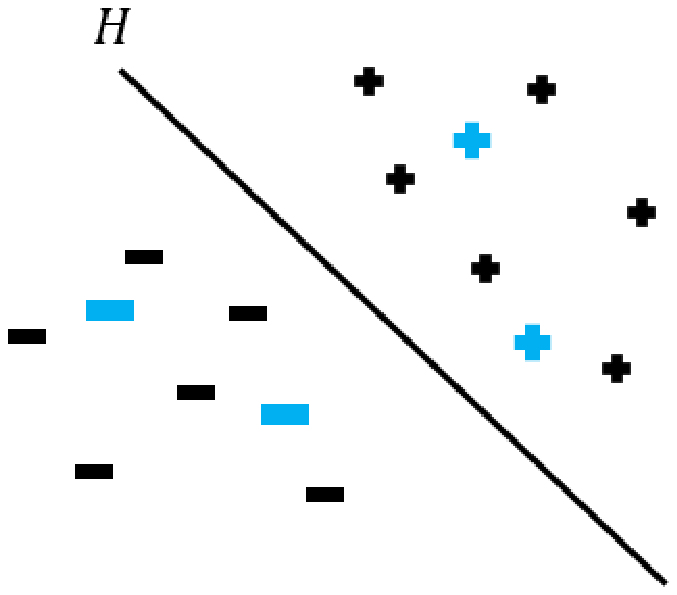}
}
\
\subfigure[(c) Points with indices in $\Psi_{3}$]{
\includegraphics[width=0.28\textwidth]{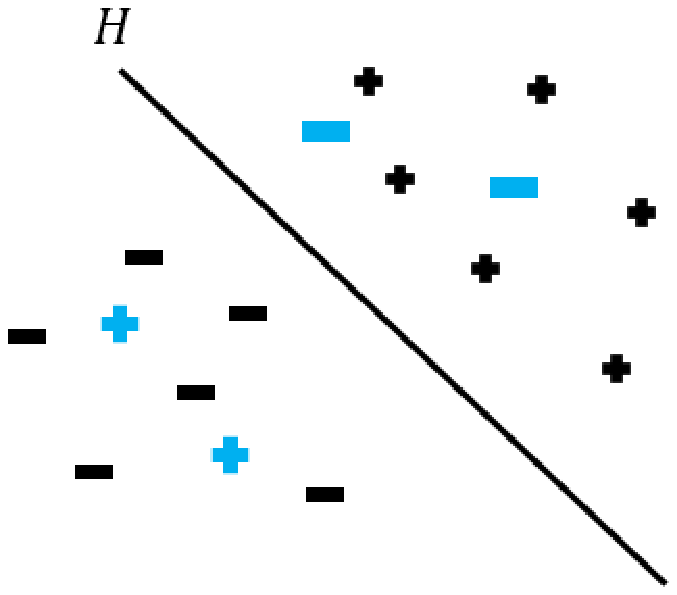}
}
\caption{Representation of points with index sets  $\Psi_j$, $j=1,\, 2, \, 3$}\label{fig_upper}
\end{figure}
\begin{proposition}\label{pro9}
Considering the validation points corresponding to $Q_{u}$ in {\rm(\ref{eqb1})}, let $(\zeta,\ z, \, \alpha)$ satisfy the first two complementarity systems describing the feasible set of problem  {\rm(\ref{eq15})}. Then, the following statements hold true:
\begin{itemize}
\item[ {{\rm (a)}}] The points $\{x_i\}_{i\in \Psi_{1}}$ lie on the separating hyperplane $H$ and are therefore correctly classified.
\item[ {{\rm (b)}}] The points $\{x_i\}_{i\in \Psi_{2}}$ lie on the correctly classified side of the separating hyperplane $H$ and are therefore correctly classified.
\item[ {{\rm (c)}}] The points $\{x_i\}_{i\in \Psi_{3}}$ lie on the misclassified side of the separating hyperplane $H$ and are therefore misclassified.
\end{itemize}
\end{proposition}
\begin{proof} We take positive points for example. The same analysis can be applied to negative ones. Since $w=B^{\top} \alpha$ in (\ref{eq_lowerKKTc}), we get $(A B^{\top} \alpha +z)_{i}=(A w +z)_{i}$.
\begin{itemize}
\item [{\rm (a)}] For the points $\{x_i\}_{i\in \Psi_{1}}$, since $z_{i}=0$ in {\rm (\ref{eqPsi1})}, we have $(A w+z)_{i}=(Aw)_{i}=0$, that is, $y_i(w^{\top} x_i)=0$. For a positive point, $y_{i}=1$, it implies that $w^{\top} x_i=0$. It corresponds to (b) in Proposition \ref{pro8}. Therefore, it means that the point $x_i$ lies on the separating hyperplane $H$. It is correctly classified.
\item[{\rm (b)}] For the points $\{x_i\}_{i\in \Psi_{2}}$, since $z_{i}=0$ in {\rm (\ref{eq7b})}, we have $(Aw+z)_{i}=(Aw)_{i}>0$, that is, $y_i(w^{\top} x_i)>0$. For a positive point, $y_{i}=1$, it implies that $w^{\top} x_i>0$. It corresponds to (c) in Proposition \ref{pro8}. Therefore, it means that the point $x_i$ lies on the correctly classified side of the separating hyperplane $H$. It is correctly classified.
\item[{\rm (c)}]
  For the points $\{x_i\}_{i\in \Psi_{3}}$, since $z_{i}>0$ in {\rm (\ref{eq7d})}, we have $(Aw)_{i}<0$, that is, $y_{i}(w^{\top} x_{i})<0$. For a positive point, $y_{i}=1$, it implies that $w^{\top} x_i<0$. It corresponds to (a) in Proposition \ref{pro8}. Therefore, it means that the point $x_i$ lies on the misclassified side of the separating hyperplane $H$.
\end{itemize}
%
\end{proof}
In Section \ref{sec4},  Proposition \ref{pro9} will be combined with Proposition \ref{pro7} to prove Proposition \ref{pro4}. It might also be important to note that if a validation point $x_{i}$ lies on the separating hyperplane $H$, then we will have $0 \leq \zeta_{i} < 1$.

 \subsection{The global relaxation method (GRM)}\label{sec6}
Here, we present a numerical algorithm to solve the MPEC \eqref{eq15}. There are various methods for solving MPECs, we refer to \cite{dempe2003annotated,luo1996mathematical} for some surveys on the problem and to \cite{jane2005necessary,flegel2005constraint,wu2015inexact,harder2021reformulation,guo2015solving,jara2018study,judice2012algorithms,li2015superlinearly,yu2019solving,dempe2003annotated,anitescu2000solving,facchinei2007finite,fletcher2006local,fukushima2002implementable} for some of the latest methods to solve the problem. Among methods to solve MPECs, one of the most popular ones is the relaxation method due to Scholtes \cite{scholtes2001convergence}. Recently, Kanzow et al. \cite{hoheisel2013theoretical} provided comparisons of five relaxation methods for solving MPECs, where it appears that the GRM has the best theoretical (in terms of requiring weaker assumptions for convergence) and numerical performance. Therefore, we will apply the GRM to solve our MPEC \eqref{eq15}.

To simplify the presentation of the method, we now write problem \eqref{eq15}  {into} further compact format. Let $v = \left[C,\ \zeta^{\top},\ z^{\top},\ \alpha^{\top},\ \xi^{\top}\right]^{\top}  \in  \mathbb{R}^{\overline{m}+1}$
with $\overline{m}= 2 T (m_{1}+m_{2})$
and define the functions
\be \label{eqe1}
F(v)\! = \!M^{\top}v, \;\; G(v)\!= \!Pv+a, \; \mbox{ and }\; H(v) \!= \!Qv,
\ee
where
 $$
  M\!=\!\frac{1}{T m_{1}} \left[\begin{array}{c} 0\\ \mathbf{1}_{T m_1} \\ \mathbf{0}_{T m_1 } \\ \mathbf{0}_{T m_2 } \\ \mathbf{0}_{T m_2 } \end{array}\right] \! \in \! \mathbb{R}^{\overline{m}+1},\ a \! = \! \left[\begin{array}{l} \mathbf{0}_{T m_1} \\ \mathbf{1}_{T m_1} \\-\mathbf{1}_{T m_2} \\ \mathbf{0}_{T m_2} \end{array}\right] \! \in \! \mathbb{R}^{\overline{m}},\ Q \!= \!\left[\begin{array}{cc} \mathbf{0}_{\overline{m}}& I_{\overline{m}} \end{array}\right] \! \in \! \mathbb{R}^{\overline{m} \times (\overline{m}+1)},
  $$
  $$ P \!= \!\left[\begin{array}{ccccc}\mathbf{0}_{T m_1 }& \mathbf{0}_{T m_1 \times T m_1} & I_{T m_1} & A B^{\top} & \mathbf{0}_{T m_1 \times T m_2} \\ \mathbf{0}_{T m_1} &-I_{T m_1 } & \mathbf{0}_{T m_1 \times T m_1} & \mathbf{0}_{T m_1 \times T m_2} & \mathbf{0}_{T m_1 \times T m_2}\\ \mathbf{0}_{T m_2} &\mathbf{0}_{T m_2 \times T m_1} & \mathbf{0}_{T m_2 \times T m_1} & B B^{\top} & I_{ T m_2} \\
\mathbf{1}_{T m_2 } &\mathbf{0}_{T m_2 \times T m_1} & \mathbf{0}_{T m_2 \times T m_1} & -I_{T m_2 } & \mathbf{0}_{T m_2 \times T m_2 }\end{array}\right]  \! \in \! \mathbb{R}^{\overline{m} \times (\overline{m}+1)}. $$
Problem (\ref{eq15}) can then be written in the form
\be \label{eq16}
\begin{aligned}
& \min_{v \in \mathbb{R}^{\overline{m}+1}} F(v)\\
& \quad \hbox{s.t.}\ \  \mathbf{0} \leq  H(v)  \perp  G(v)  \geq \mathbf{0}.
\end{aligned}
\ee

The basic idea of the GRM is as follows. Let $\{t_{k}\} \downarrow 0$. At each iteration, we replace the MPEC (\ref{eq16}) by the nonlinear program (NLP) of the following form, parameterized in $t_k$:
\be \label{eq50}
\begin{aligned}
& \min_{v}\ F(v)\\
& \begin{array}{ll}
\hbox{s.t.}\ \ G_{i}(v) \geq 0 \quad  \forall \ i=1,\ \cdots,\ \overline{m},\\
\quad \quad H_{i}(v) \geq 0 \quad  \forall \ i=1,\ \cdots,\ \overline{m},\\
\quad \quad G_{i}(v)H_{i}(v) \leq t_{k} \quad  \forall \ i=1,\ \cdots,\ \overline{m}.
  \end{array}
\end{aligned} \tag{NLP-$t_{k}$}
\ee
The details of the GRM are shown in Algorithm \ref{alo2}.
\begin{algorithm}
 \caption{The Global Relaxation Method (GRM) $(v_0,\ t_0,\ \sigma,\ t_{min})$}\label{alo2}
\begin{algorithmic}[1]
\State {\bf Require} a starting vector $v_0$, an initial relaxation parameter $t_0$, and parameters $\sigma \in  (0, 1),\ t_{\min} > 0$.
\State Set $k:=0$.
\State {\bf while} {$t_k> t_{\min}$} {\bf do}
\State Find an approximate solution $v^{k+1}$ of the relaxed problem (\ref{eq50}) \quad using $v^{k}$ as a starting point.
\State Let $t_{k+1} \leftarrow \sigma \cdot t_{k}$ and $k \leftarrow k+1$.
\State {\bf end while}
\State {\bf Return} the final iterate $v_{opt} := v^{k}$, the corresponding function value $F (v_{opt} )$, and the maximum constraint violation Vio($v_{opt}$).
\end{algorithmic}
\end{algorithm}

\noindent Here, the maximum violation of all constraints Vio defined by
\be \label{eqc1}
\operatorname{Vio}\left(v_{o p t}\right)=\|\min\{G(v_{opt}),\ H(v_{opt}) \}\|_{\infty}
\ee
is used to measure the feasibility of the final iterate $v_{opt}$, where $\|\cdot \|_{\infty}$ denotes the $l_{\infty}$ norm.
We use the GRM in Algorithm \ref{alo2} to solve the MPEC (\ref{eq16}), and get the optimal hyperparameter $C$ and  the corresponding function value $F(v_{opt})$ which is the cross-validation error (CV error) measured on the validation sets in T-fold cross-validation.
To analyze the convergence of the GRM, we need the concept of C-stationarity, which we define next.

To proceed, let $v$ be a feasible point for the MPEC $(\ref{eq16})$ and recall that $F(v), G(v)$ and $H(v)$ defined in (\ref{eqe1}). Based on $v$, let
$$
\begin{aligned}
& I_{G}&:= \;\; & \{i \ \mid \ G_{i}(v)=0,\ H_{i}(v)>0 \}, \\
& I_{GH}&:=\;\;  & \{i \ \mid \ G_{i}(v)=0,\ H_{i}(v)=0 \},\\
& I_{H}&:=\;\;  & \{i \ \mid \ G_{i}(v)>0,\ H_{i}(v)=0 \}.
\end{aligned}
$$
\begin{definition}{\rm (C-stationarity)}
Let $v$ be a feasible point for the MPEC {\rm (\ref{eq16})}. Then $v$ is said to be a C-stationary point, if there are multipliers $\gamma,\ \nu \in \mathbb{R}^{\overline{m}}$, such that
$$
\nabla F\left(v\right)-\sum_{i=1}^{\overline{m}} \gamma_{i} \nabla G_{i}\left(v\right) -\sum_{i=1}^{\overline{m}} \nu_{i} \nabla H_{i}\left(v\right)=0,
$$
and $\gamma_{i}=0$ for $ i \in I_{H},\ \nu_{i}=0$ for $i \in I_{G}$, and $\gamma_{i} \nu_{i} \geq 0$ for $i \in I_{GH}$.
\end{definition}
Note that for problem \eqref{eq16}, C-stationarity holds at any local optimal solution that satisfies the MPEC-MFCQ, which can be defined as follows \cite{hoheisel2013theoretical}.

\begin{definition}\label{pro1}
 A feasible point $v$ for problem $(\ref{eq16})$ satisfies the MPEC-MFCQ if and only if the set of gradient vectors
\be \label{eq22}
\begin{aligned}
\left\{\nabla G_{i}\left(v\right) \mid i \in  I_{G} \cup I_{GH} \right\}
\cup \left\{\nabla H_{i}\left(v\right) \mid i \in  I_{H} \cup I_{GH}\right\}
\end{aligned}
\ee
is positive-linearly independent.
\end{definition}
Recall that the set of gradient vectors in {\rm (\ref{eq22})} is said to be positive-linearly \emph{dependent} if there exist scalars $\{\delta_{i}\}_{i \in I_{G}  \cup I_{GH}}$ and $ \{\beta_{i}\}_{i \in I_{H} \cup I_{GH}}$ with $\delta_{i} \geq 0$ for $i \in I_{G}  \cup I_{GH}$, $\beta_{i} \geq 0$ for $i \in I_{H} \cup I_{GH}$, not all of them being zero, such that $\Sigma_{i \in I_{G}  \cup I_{GH}} \delta_{i} \nabla G_{i}(v)+\Sigma_{i \in I_{H}  \cup I_{GH}} \beta_{i} \nabla H_{i}(v)=0$. Otherwise, we say that this set of gradient vectors is positive-linearly \emph{independent}.

Also note that various other stationarity concepts can be defined for problem \eqref{eq16}; for more details on this, interested readers are referred to \cite{dempe2012karush,flegel2005constraint}.

The following result establishes the well-definiteness  of Algorithm \ref{alo2}, as it provides a framework ensuring that a solutions (or a stationary points, to be precise) exist for problem {\rm (\ref{eq50})} as required.
\begin{theorem}\label{thr2}{\rm \cite{hoheisel2013theoretical}}
Let $v$ be a feasible point for the MPEC $(\ref{eq16})$ such that MPEC-MFCQ is satisfied at $v$. Then there exists a neighborhood $N$ of $v$ and $\overline{t} > 0$ such that standard MFCQ for {\rm (\ref{eq50})} at $t_k=t$ is satisfied at all feasible points of {\rm (\ref{eq50})} at $t_k=t$ in this neighborhood $N$ for all $t \in (0,\ \overline{t})$.
\end{theorem}
Subsequently, we have the following convergence result, which ensures that a sequence of stationary points of problem (\ref{eq50}), computed by Algorithm \ref{alo2}, converges to a C-stationary point of problem \eqref{eq16}.
\begin{theorem} \label{thr1}{\rm \cite{hoheisel2013theoretical}}
Let $\{t_{k}\} \downarrow 0$ and let $v^k$ be a stationary point of {\rm (NLP-$t_{k}$)} with $v^{k} \rightarrow v$ such that MPEC-MFCQ holds at the feasible point $v$. Then $v$ is a C-stationary point of the MPEC $(\ref{eq16})$.
\end{theorem}

Clearly, the MPEC-MFCQ is crucial for the analysis of problem \eqref{eq16}, as it not only ensures that the C-stationarity condition can hold at a locally optimal point, but also helps in establishing the two fundamental results in Theorems \ref{thr2} and \ref{thr1}. Considering this importance of the condition, we carefully analyze it in the next section, and show, in particular, that it automatically holds at any feasible point of problem $(\ref{eq16})$.

\section{Fulfilment of the MPEC-MFCQ}\label{sec4}

\noindent In this section, we prove that every point in the feasible set of the MPEC (\ref{eq16}) satisfies the MPEC-MFCQ. The rough idea of our proof is as follows. Firstly, by analyzing the relationship of different index sets (Proposition \ref{pro4}), we reach a reduced form of the MPEC-MFCQ (Proposition \ref{pro5}). Then based on the positive-linear independence of three submatrices (Lemma \ref{lem1}-Lemma \ref{lem3}), we eventually show the MPEC-MFCQ in Theorem \ref{thr3}. The roadmap of the proof is summarized in Figure \ref{liucheng}.
\begin{figure}[h]
	\centering
	\includegraphics[width=1\textwidth]{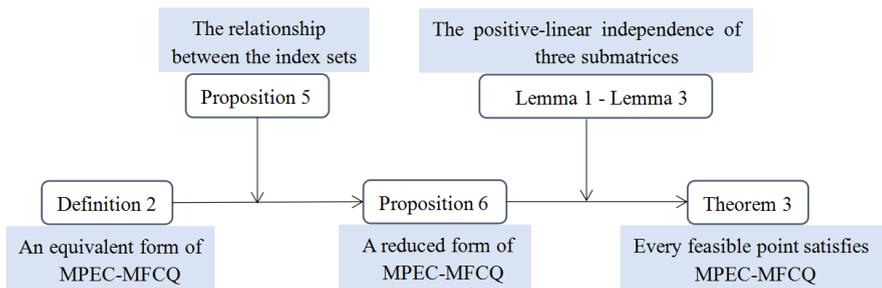}
	\caption{The roadmap of the proof of the MPEC-MFCQ.}\label{liucheng}
\end{figure}
\subsection{Relationships between the index sets}

In this part, we first explore more properties about the index sets $I_{H},\ I_{G},\ I_{GH}$, as they are key to the analysis of the positive-linear independence of the vectors in (\ref{eq22}).
  Let $I_{H}:=\underset{k=1} {\overset{4}{\cup}}I_{H_{k}},\ I_{G}:=\underset{k=1} {\overset{4}{\cup}}I_{G_{k}},$ and $I_{GH}:=\underset{k=1} {\overset{4}{\cup}}I_{GH_{k}}$, where
\begin{subequations} \label{eq23}
\begin{align}
& I_{H_{1}}\ \; := \; \{i \in Q_{u} \ \mid \ \zeta_{i}=0,\ (AB^{\top} \alpha +z)_{i}>0\}, \label{eq23a} \\
& I_{H_{2}} \ \; :=\; \{i \in Q_{u} \ \mid \ z_{i}=0,\ 1-\zeta_{i}>0\}, \label{eq23b} \\
& I_{H_{3}} \ \; := \; \{i \in Q_{l} \ \mid \ \alpha_{i}=0,\ (BB^{\top} \alpha -\mathbf{1}+\xi)_{i}>0\}, \label{eq23c}\\
& I_{H_{4}} \ \; := \; \{i \in Q_{l} \ \mid \ \xi_{i}=0,\ C-\alpha_{i}>0\}, \label{eq23d}\\
& I_{G_{1}} \ \; :=\; \{i \in Q_{u} \ \mid \ \zeta_{i}>0,\ (AB^{\top} \alpha +z)_{i}=0\},\label{eq23e}\\
& I_{G_{2}} \ \; :=\; \{i \in Q_{u} \ \mid \ z_{i}>0,\ 1-\zeta_{i}=0\},\label{eq23f}\\
& I_{G_{3}} \ \; :=\; \{i \in Q_{l} \ \mid \ \alpha_{i}>0,\ (BB^{\top} \alpha -\mathbf{1}+\xi)_{i}=0\},\label{eq23g}\\
& I_{G_{4}}\ \; :=\; \{i \in Q_{l} \ \mid \ \xi_{i}>0,\ C-\alpha_{i}=0\}, \label{eq23h}\\
& I_{GH_{1}} :=\; \{i \in Q_{u} \ \mid \ \zeta_{i}=0,\ (AB^{\top} \alpha +z)_{i}=0\},\label{eq23i}\\
& I_{GH_{2}} :=\; \{i \in Q_{u} \ \mid \ z_{i}=0,\ 1-\zeta_{i}=0\},\label{eq23j} \\
& I_{GH_{3}} :=\; \{i \in Q_{l} \ \mid \ \alpha_{i}=0,\ (BB^{\top} \alpha -\mathbf{1}+\xi)_{i}=0\}, \label{eq23k}\\
& I_{GH_{4}} :=\; \{i \in Q_{l} \ \mid \ \xi_{i}=0,\ C-\alpha_{i}=0\}.\label{eq23l}
\end{align}
\end{subequations}
Here, $Q_{u},\ Q_{l}$ are defined in (\ref{eqb1}) and (\ref{eqb2}), respectively. Furthermore, let
 {
$$ \setlength{\abovedisplayskip}{1.5pt} I^{k}:=I_{H_{k}}\cup I_{G_{k}} \cup I_{GH_{k}},\ k=1,\ 2,\ 3,\ 4.
\setlength{\belowdisplayskip}{1.5pt} $$}
It can be observed that each index set $I^{k},\ k=1,\ 2,\ 3,\ 4$ corresponds to the union of the three components in the partition involved in the corresponding part of the complementarity systems in (\ref{eq15}); that is,
\begin{itemize}
\item [{}]  Part 1: $I^{1}$ for the partition of the system $\mathbf{0} \leq  \zeta  \perp \ A B^{T} \alpha+z \geq \mathbf{0}$;
\item [{}]  Part 2: $I^{2}$ for the partition of the system  $\mathbf{0} \leq  z  \perp  \mathbf{1}-\zeta  \geq \mathbf{0}$;
\item [{}]  Part 3: $I^{3}$ for the partition of the system $\mathbf{0} \leq  \alpha  \perp  B B^{T} \alpha- \mathbf{1}+\xi  \geq \mathbf{0}$;
\item [{}]  Part 4: $I^{4}$ for the partition of the system $\mathbf{0} \leq  \xi  \perp  C \mathbf{1}-\alpha \geq \mathbf{0}$.
\end{itemize}
In the previous section, we have clarified a one-to-one correspondence between the index set of the validation points $Q_u$ in (\ref{eqb1}) and the complementary constraints in Part 1 and Part 2, respectively. It is clearly that $I^{1}=I^{2}=Q_{u}$. Similarly, we have $I^{3}=I^{4}=Q_{l}$.
\begin{figure}[htbp]
\centering
\subfigure[(a) Part 1.]{
\includegraphics[width=0.41\textwidth]{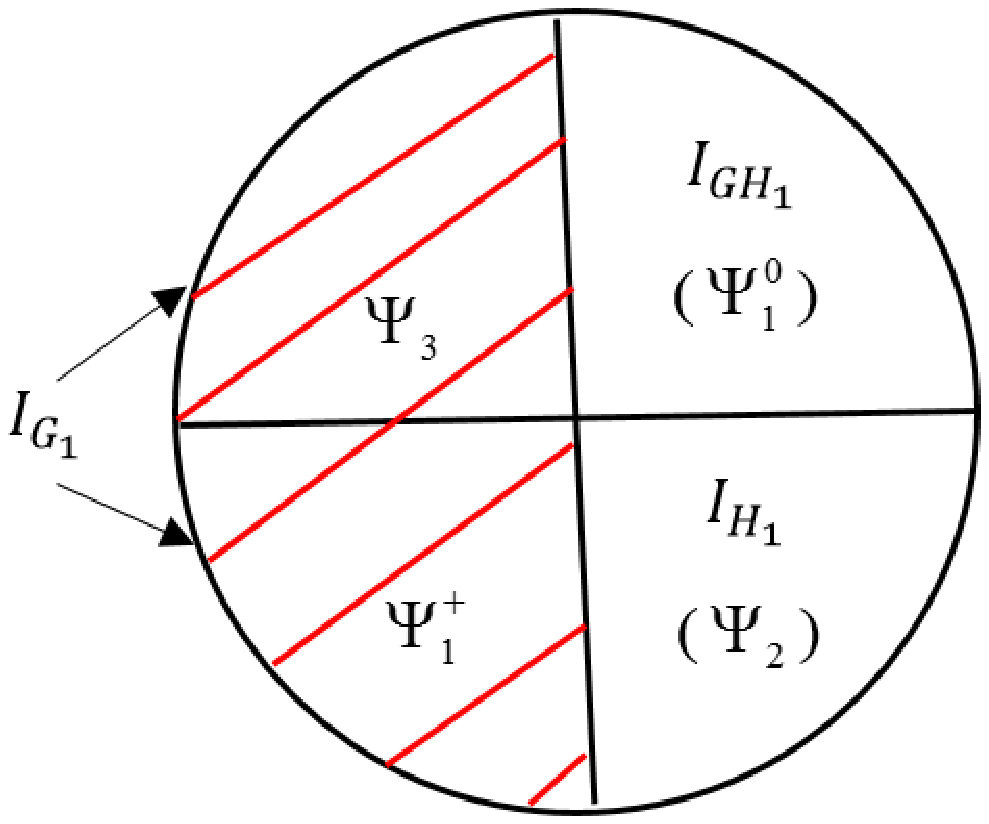}
}
\
\subfigure[(b) Part 2.]{
\includegraphics[width=0.35\textwidth]{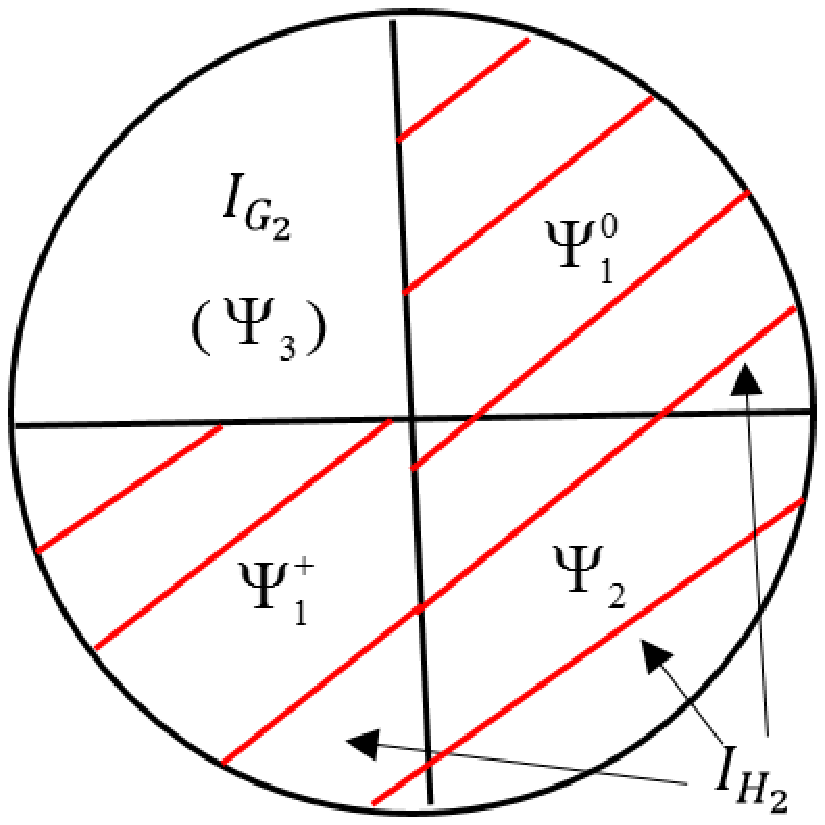}
}
\
\subfigure[(c) Part 3.]{
\hspace{0.1cm}
\includegraphics[width=0.39\textwidth]{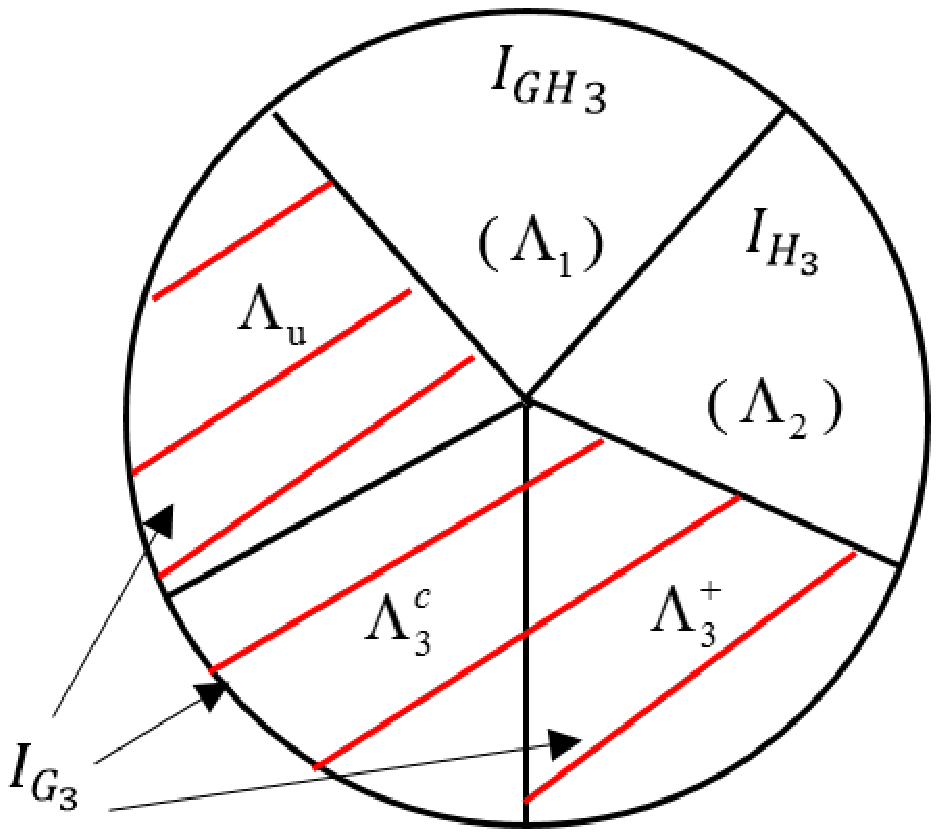}
}
\
\subfigure[(d) Part 4.]{
\hspace{0.2cm}
\includegraphics[width=0.43\textwidth]{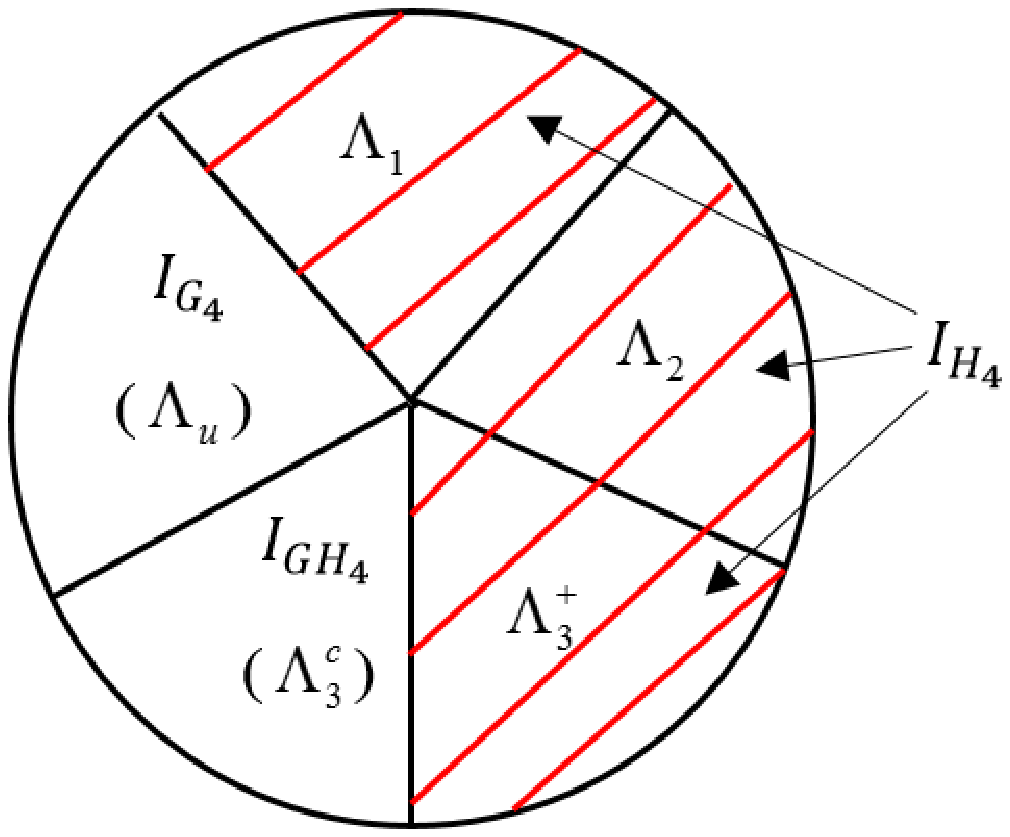}
}
\caption{The index sets corresponding to the complementarity constraints in Parts 1--4.}\label{fig_relation}
\end{figure}
Next, we give the relationships between the index sets in (\ref{eq23}); recall that we already have some index sets described in Propositions \ref{pro7} and \ref{pro9}. For the convenience of the analysis, we divide the index set $\Lambda_3$ in (\ref{eqLambda3}) into two subsets $\Lambda^{+}_{3}$ and $\Lambda^{c}_{3}$, as well as  $\Psi_{1}$ in (\ref{eqPsi1}) into  $\Psi^{0}_{1}$ and $\Psi^{+}_{1}$:
\begin{eqnarray}\label{eqdivideset}
  \Lambda^{+}_{3} &:=&\{i \in Q_{l}\ \mid \ 0<\alpha_{i}<C,\ (B B^{\top} \alpha-\mathbf{1}+\xi)_{i}=0,\ \xi_{i}=0\},  \label{eq6c}\\
 \Lambda^{c}_{3} &:=& \{i \in Q_{l}\ \mid \ \alpha_{i}=C,\ (B B^{\top} \alpha-\mathbf{1}+\xi)_{i}=0,\ \xi_{i}=0\}, \label{eq6d}\\
\Psi^{0}_{1} &:=& \{i \in Q_{u}\ \mid \ \zeta_{i}=0,\ (A B^{\top} \alpha+z)_{i}=0,\ z_{i}=0\}, \label{eq7a}\\
\Psi^{+}_{1}&:=& \{i \in Q_{u}\ \mid \ 0<\zeta_{i}<1,\ (A B^{\top} \alpha+z)_{i}=0,\ z_{i}=0\}. \label{eq7c}
\end{eqnarray}

\begin{proposition}\label{pro4}
The index sets in $(\ref{eq23})$ and the index sets in Proposition $\ref{pro7}$ and Proposition $\ref{pro9}$ have the following relationship:
\begin{itemize}
\item [{\rm (a)}] In Part $1,\ I_{H_{1}}=\Psi_{2},\ I_{G_{1}}=\Psi^{+}_{1} \cup \Psi_{3},\ I_{GH_{1}}=\Psi^{0}_{1}.$
\item [{\rm (b)}] In Part $2,\ I_{H_{2}}=\Psi_{1} \cup \Psi_{2},\ I_{G_{2}}=\Psi_{3},\ I_{GH_{2}}=\emptyset.$
\item [{\rm (c)}] In Part $3,\ I_{H_{3}}=\Lambda_{2},\ I_{G_{3}}=\Lambda_{3} \cup \Lambda_{u},\ I_{GH_{3}}=\Lambda_{1}.$
\item [{\rm (d)}] In Part $4,\ I_{H_{4}}=\Lambda_{1} \cup \Lambda_{2} \cup \Lambda^{+}_{3},\ I_{G_{4}}=\Lambda_{u},\ I_{GH_{4}}=\Lambda^{c}_{3}.$
\end{itemize}
Here, $\Lambda_{u}$ is defined as follows
\be \label{eqa4}
\Lambda_{u}:=\left\{i \in Q_{l}\ \mid \ \alpha_{i}=C,\ (B B^{\top} \alpha-\mathbf{1}+\xi)_{i}=0,\ \xi_{i}>0\right\}=\Lambda_{4} \cup \Lambda_{5} \cup \Lambda_{6}.
\ee
\end{proposition}
\vspace{-0.5cm}
\begin{proof}
According to the definition of the index sets in $(\ref{eq23})$ and the index sets in Proposition $\ref{pro7}$ and Proposition $\ref{pro9}$, we have the following analysis.
\begin{itemize}
\item [{\rm (a)}] In Part 1, for $i \in I_{H_{1}}$, compared with the index set $\Psi_{2}$ in {\rm (\ref{eq7b})}, it follows that we have $z_{i}=0$ and $I_{H_{1}}=\Psi_{2}$. For $i \in I_{G_{1}}$, compared with the index sets $\Psi^{+}_{1}$ in {\rm (\ref{eq7c})} and $\Psi_{3}$ in {\rm (\ref{eq7d})}, we get $0<\zeta_{i}<1,\ z_{i}=0$ or $\zeta_{i}=1,\ z_{i}>0$, and $I_{G_{1}}=\Psi^{+}_{1} \cup \Psi_{3}$. For $i \in I_{GH_{1}}$, compared with the index set $\Psi^{0}_{1}$ in {\rm (\ref{eq7a})}, we get $z_{i}=0$ and $I_{GH_{1}}=\Psi^{0}_{1}$.
\item [{\rm (b)}] In Part 2, for $i \in I_{H_{2}}$, compared with the index sets $\Psi_{1}$ in {\rm (\ref{eqPsi1})} and $\Psi_{2}$ in {\rm (\ref{eq7b})}, we get $I_{H_{2}}=\Psi_{1} \cup \Psi_{2}.$ For $i \in I_{G_{2}}$, compared with the index set  $\Psi_{3}$ in {\rm (\ref{eq7d})}, we get $(A B^{\top} \alpha+z)_{i}=0$ and $I_{G_{2}}=\Psi_{3}$. For $i \in I_{GH_{2}}$, there is no index set in Proposition \ref{pro9} corresponds to the index set $I_{GH_{2}}$. Therefore, $I_{GH_{2}}=\emptyset.$
\item [{\rm (c)}] In Part 3, for $i \in I_{H_{3}}$, compared with the index set $\Lambda_{2}$ in {\rm (\ref{eq6b})}, we get $\xi_{i}=0$ and $I_{H_{3}}=\Lambda_{2}$. For $i \in I_{G_{3}}$, compared with the index sets $\Lambda_{3}$ in {\rm (\ref{eqLambda3})} and $\Lambda_{u}$ in {\rm (\ref{eqa4})}, we get $I_{G_{3}}=\Lambda_{3} \cup \Lambda_{u}.$ For $i \in I_{GH_{3}}$, compared with the index set $\Lambda_{1}$ in {\rm (\ref{eq6a})}, we get $\xi_{i}=0$ and $I_{GH_{3}}=\Lambda_{1}$.
\item [{\rm (d)}] In Part 4, for $i \in I_{H_{4}}$, compared with the index sets $\Lambda_{1}$ in {\rm (\ref{eq6a})}, $\Lambda_{2}$ in {\rm(\ref{eq6b})} and $\Lambda^{+}_{3}$ in {\rm (\ref{eq6c})}, we get $I_{H_{4}}=\Lambda_{1} \cup \Lambda_{2} \cup \Lambda^{+}_{3}.$ For $i \in I_{G_{4}}$, compared with the index set $\Lambda_{u}$ in {\rm (\ref{eqa4})}, we get $(Bw-\mathbf{1}+\xi)_{i}=0$ and $I_{G_{4}}=\Lambda_{u}.$ For $i \in I_{GH_{4}}$, compared with the index set $\Lambda^{c}_{3}$ in {\rm (\ref{eq6d})}, it results that we have $(Bw-\mathbf{1}+\xi)_{i}=0$ and $I_{GH_{4}}=\Lambda^{c}_{3}$.
\end{itemize}
\end{proof}
The results in Proposition \ref{pro4} are demonstrated in Figure \ref{fig_relation}. For example, for (a) in Proposition \ref{pro4}, the index sets of complementarity constraints in Part 1 are shown in Figure \ref{fig_relation} (a), which is about the relationship of $I_{H_{1}},\ I_{G_{1}},\ I_{GH_{1}}$ in (\ref{eq23}) and the index sets (\ref{eqPsi1})--(\ref{eq7d}). In Figure \ref{fig_relation} (a), the red shaded part represents the index set $I_{G_{1}}$, which contains the index sets $\Psi^{+}_{1}$ and $\Psi_{3}$. (b)--(d) in Proposition \ref{pro4} are demonstrated in Figure \ref{fig_relation} (b)--(d). Specially, in Figure \ref{fig_relation} (b),
the red shaded part represents the index set $I_{H_{2}}$, which contains the index sets  {$\Psi_{1}$ (or $\Psi^{0}_{1} \cup \Psi^{+}_{1}$)} and $\Psi_{2}$. In Figure \ref{fig_relation} (c), the red shaded part represents the index set $I_{G_{3}}$, which contains the index sets  {$\Lambda_{3}$ (or $\Lambda^{+}_{3} \cup \Lambda^{c}_{3}$)} and $\Lambda_{u}$. In Figure \ref{fig_relation} (d), the red shaded part represents the index set $I_{H_{4}}$, which contains the index sets $\Lambda_{1},\ \Lambda_{2}$, and $\Lambda^{+}_{3}$.

\subsection{The reduced form of the MPEC-MFCQ}
Here we provide a matrix representation of the union of gradients in {\rm(\ref{eq22})}.
\begin{proposition}\label{pro5}
 The set of gradient vectors in {\rm(\ref{eq22})} at a feasible point $v$ for the MPEC {\rm (\ref{eq16})} can be written in the matrix form
\be \label{eq24}
\Gamma=\left[\begin{array}{ccccc}\mathbf{0}_{(I_{G_{1}},\ L_{1})}& \mathbf{0}_{(I_{G_{1}},\ L_{2})} &\cellcolor{green!50}\Gamma_{a}^{3} & (A B^{\top})_{(I_{G_{1}},\ \cdot \ )} & \mathbf{0}_{(I_{G_{1}},\ L_{5})} \\
\mathbf{0}_{(I_{GH_{1}},\ L_{1})}& \mathbf{0}_{(I_{GH_{1}},\ L_{2})} &\cellcolor{green!50}\Gamma_{b}^{3} & (A B^{\top})_{(I_{GH_{1}},\ \cdot \ ) } & \mathbf{0}_{(I_{GH_{1}},\ L_{5})} \\
\mathbf{0}_{(I_{GH_{1}},\ L_{1})} &\cellcolor{blue!50}\Gamma_{c}^{2} & \mathbf{0}_{(I_{GH_{1}},\ L_{3}) } & \mathbf{0}_{(I_{GH_{1}},\ L_{4})} & \mathbf{0}_{(I_{GH_{1}},\ L_{5})}\\
\mathbf{0}_{(I_{H_{1}},\ L_{1})} &\cellcolor{blue!50}\Gamma_{d}^{2} & \mathbf{0}_{(I_{H_{1}},\ L_{3}) } & \mathbf{0}_{(I_{H_{1}},\ L_{4})} & \mathbf{0}_{(I_{H_{1}},\ L_{5})}\\
\mathbf{0}_{(I_{G_{2}},\ L_{1})} & \cellcolor{blue!50} \Gamma_{e}^{2}& \mathbf{0}_{(I_{G_{2}},\ L_{3}) } & \mathbf{0}_{(I_{G_{2}},\ L_{4})} & \mathbf{0}_{(I_{G_{2}},\ L_{5})}\\
\mathbf{0}_{(I_{H_{2}},\ L_{1})}& \mathbf{0}_{(I_{H_{2}},\ L_{2})} &\cellcolor{green!50}\Gamma_{f}^{3} & \mathbf{0}_{(I_{H_{2}},\ L_{4}) } & \mathbf{0}_{(I_{H_{2}},\ L_{5})} \\
\mathbf{0}_{(I_{G_{3}},\ L_{1})} &\mathbf{0}_{(I_{G_{3}},\ L_{2})} & \mathbf{0}_{(I_{G_{3}},\ L_{3})} & \cellcolor{yellow!50}(B B^{\top})_{(I_{G_{3}},\ \cdot \ )} & \cellcolor{yellow!50}\Gamma_{g}^{5} \\
\mathbf{0}_{(I_{GH_{3}},\ L_{1})} &\mathbf{0}_{(I_{GH_{3}},\ L_{2})} & \mathbf{0}_{(I_{GH_{3}},\ L_{3})} & \cellcolor{yellow!50}(B B^{\top})_{(I_{GH_{3}},\ \cdot \ )} & \cellcolor{yellow!50}\Gamma_{h}^{5} \\
\mathbf{0}_{(I_{GH_{3}},\ L_{1})} &\mathbf{0}_{(I_{GH_{3}},\ L_{2})} & \mathbf{0}_{(I_{GH_{3}},\ L_{3})} & \cellcolor{yellow!50}\Gamma_{i}^{4} & \cellcolor{yellow!50}\mathbf{0}_{(I_{GH_{3}},\ L_{5})}\\
\mathbf{0}_{(I_{H_{3}},\ L_{1})} &\mathbf{0}_{(I_{H_{3}},\ L_{2})} & \mathbf{0}_{(I_{H_{3}},\ L_{3})} & \cellcolor{yellow!50}\Gamma_{j}^{4} & \cellcolor{yellow!50}\mathbf{0}_{(I_{H_{3}},\ L_{5})}\\
\mathbf{1}_{(I_{G_{4}},\ L_{1})} &\mathbf{0}_{(I_{G_{4}},\ L_{2})} & \mathbf{0}_{(I_{G_{4}},\ L_{3})} & \Gamma_{k}^{4} & \mathbf{0}_{(I_{G_{4}},\ L_{5}) }\\
\mathbf{1}_{(I_{GH_{4}},\ L_{1})} &\mathbf{0}_{(I_{GH_{4}},\ L_{2})} & \mathbf{0}_{(I_{GH_{4}},\ L_{3})} & \Gamma_{l}^{4} & \mathbf{0}_{(I_{GH_{4}},\ L_{5}) }\\
\mathbf{0}_{(I_{GH_{4}},\ L_{1})} &\mathbf{0}_{(I_{GH_{4}},\ L_{2})} & \mathbf{0}_{(I_{GH_{4}},\ L_{3})}  & \cellcolor{yellow!50}\mathbf{0}_{(I_{GH_{4}},\ L_{4})}& \cellcolor{yellow!50}\Gamma_{m}^{5}\\
\mathbf{0}_{(I_{H_{4}},\ L_{1})} &\mathbf{0}_{(I_{H_{4}},\ L_{2})} & \mathbf{0}_{(I_{H_{4}},\ L_{3})}  & \cellcolor{yellow!50}\mathbf{0}_{(I_{H_{4}},\ L_{4}) } & \cellcolor{yellow!50}\Gamma_{n}^{5}
\end{array}\right],
\ee
where $L_{q},\ q=1,\ \cdots,\ 5$ are the index sets of columns corresponding to the variables $C$, $\zeta$, $z$, $\alpha$, and $\xi$, respectively, and
\begin{equation}\label{eq25}
\left.
\begin{array}{lll}
\Gamma_{a}^{3} & := &
\left[
\begin{array}{cc}
\mathbf{0}_{(I_{G_{1}},\ \Psi^{0}_{1} \cup \Psi_{2})} & I_{(I_{G_{1}} ,\ \Psi^{+}_{1}\cup \Psi_{3})}
\end{array}
\right]\\
\Gamma_{b}^{3}& := &
\left[\begin{array}{cc} I_{(I_{GH_{1}},\ \Psi^{0}_{1})}& \mathbf{0}_{(I_{GH_{1}},\ \Psi^{+}_{1} \cup \Psi_{2}\cup \Psi_{3})}
\end{array} \right]\\
\Gamma_{c}^{2} & := &\left[\begin{array}{cc} I_{(I_{GH_{1}},\ \Psi^{0}_{1})}& \mathbf{0}_{(I_{GH_{1}},\ \Psi^{+}_{1} \cup \Psi_{2} \cup  \Psi_{3})} \end{array} \right]\\
 \Gamma_{d}^{2}&:=&\left[\begin{array}{cc} \mathbf{0}_{(I_{H_{1}},\ \Psi_{1} \cup \Psi_{3})}&  I_{(I_{H_{1}},\ \Psi_{2})} \end{array} \right]\\
 \Gamma_{e}^{2}& := &\left[\begin{array}{cc} \mathbf{0}_{(I_{G_{2}},\  \Psi_{1} \cup \Psi_{2})}& -I_{(I_{G_{2}},\ \Psi_{3})} \end{array} \right]\\
 \Gamma_{f}^{3} &:=&\left[\begin{array}{cc} I_{(I_{H_{2}},\ \Psi_{1} \cup \Psi_{2})}& \mathbf{0}_{(I_{H_{2}},\ \Psi_{3})} \end{array} \right]\\
\Gamma_{g}^{5}&:=&\left[\begin{array}{cc}\mathbf{0}_{(I_{G_{3}},\ \Lambda_{1} \cup \Lambda_{2})}& I_{(I_{G_{3}},\ \Lambda_{3} \cup \Lambda_{u})} \end{array} \right]\\
\Gamma_{h}^{5}&:=&\left[\begin{array}{cc}I_{(I_{GH_{3}},\  \Lambda_{1})}&\mathbf{0}_{(I_{GH_{3}},\ \Lambda_{2}\cup \Lambda_{3} \cup \Lambda_{u})} \end{array} \right] \\
\Gamma_{i}^{4}&:=&\left[\begin{array}{cc} I_{(I_{GH_{3}},\ \Lambda_{1}) }& \mathbf{0}_{(I_{GH_{3}},\ \Lambda_{2} \cup \Lambda_{3} \cup \Lambda_{u})} \end{array} \right]\\
 \Gamma_{j}^{4}&:=&\left[\begin{array}{cc} \mathbf{0}_{(I_{H_{3}},\  \Lambda_{1} \cup \Lambda_{3} \cup \Lambda_{u})}& I_{(I_{H_{3}},\ \Lambda_{2})} \end{array} \right]\\
 \Gamma_{k}^{4}&:=&\left[\begin{array}{cc} \mathbf{0}_{(I_{G_{4}},\ \Lambda_{1} \cup \Lambda_{2} \cup \Lambda_{3})} &  -I_{(I_{G_{4}},\ \Lambda_{u})} \end{array} \right]\\
 \Gamma_{l}^{4}&:=&\left[\begin{array}{cc} \mathbf{0}_{(I_{GH_{4}},\ \Lambda_{1} \cup \Lambda_{2} \cup \Lambda^{+}_{3} \cup \Lambda_{u})} & -I_{(I_{GH_{4}},\  \Lambda^{c}_{3})} \end{array} \right]\\
\Gamma_{m}^{5}&:=&\left[\begin{array}{cc} \mathbf{0}_{(I_{GH_{4}},\ \Lambda_{1} \cup \Lambda_{2} \cup \Lambda^{+}_{3} \cup \Lambda_{u})} & I_{(I_{GH_{4}},\ \Lambda^{c}_{3})} \end{array} \right]\\
 \Gamma_{n}^{5}&:=&\left[\begin{array}{cc} I_{(I_{H_{4}},\ \Lambda_{1} \cup \Lambda_{2} \cup \Lambda^{+}_{3})} & \mathbf{0}_{(I_{H_{4}},\  \Lambda^{c}_{3} \cup \Lambda_{u}) }  \end{array} \right]
\end{array}\right\}.
\end{equation}
\end{proposition}
 \begin{proof} Based on Definition \ref{pro1}, we can write the system of gradient vectors in (\ref{eq22}), at a feasible point $v$, in the equivalent matrix form
 $$
\Gamma = \left[\begin{array}{l}
\nabla G(v)_{I_{G_1}} \\ \nabla G(v)_{I_{GH_1}} \\ \nabla H(v)_{I_{GH_1}}  \\ \nabla H(v)_{I_{H_1}}\\
\nabla G(v)_{I_{G_2}} \\ \nabla H(v)_{I_{H_2}} \\ \nabla G(v)_{I_{G_3}}\\ \nabla  G(v)_{I_{GH_3}} \\ \nabla  H(v)_{I_{GH_3}} \\ \nabla H(v)_{I_{H_3}} \\ \nabla G(v)_{I_{G_4}} \\ \nabla  G(v)_{I_{GH_4}} \\ \nabla  H(v)_{I_{GH_4}} \\ \nabla H(v)_{I_{H_4}}
\end{array}\right]
$$
that we can easily show to be equivalent to (\ref{eq24}). To proceed, first note that from Proposition \ref{pro4} (a) and (b), we have
$$
\begin{array}{l}
I_{H_{1}}\!=\! \Psi_{2},\ I_{G_{1}}\!=\! \Psi^{+}_{1}\cup \Psi_{3},\ I_{GH_{1}}\!=\!\Psi^{0}_{1},\ I_{H_{2}}\!=\!\Psi_{1} \cup \Psi_{2},\\[1ex] I_{G_{2}}\!=\!\Psi_{3},\ Q_u\!=\!\Psi_{1} \cup \Psi_{2}\cup \Psi_{3}.
\end{array}
$$
So, we get $\Gamma_{a}^{3},\ \Gamma_{b}^{3},\ \Gamma_{c}^{2},\ \Gamma_{d}^{2},\ \Gamma_{e}^{2}$, and $\Gamma_{f}^{3}$ in (\ref{eq25}). On the other hand, it follows from Proposition \ref{pro4} (c) and (d), we have
$$\begin{array}{l}
I_{H_{3}}\!=\!\Lambda_{2},\ I_{G_{3}}\!=\!\Lambda_{3} \cup \Lambda_{u},\ I_{GH_{3}}\!=\!\Lambda_{1},\ I_{H_{4}}\!=\!\Lambda_{1} \cup \Lambda_{2} \cup \Lambda^{+}_{3},\\[1ex]
 I_{G_{4}}\!=\!\Lambda_{u},\ I_{GH_{4}}\!=\!\Lambda^{c}_{3}, \;\; Q_{l}\!=\!\Lambda_{1} \cup \Lambda_{2} \cup \Lambda_{3} \cup \Lambda_{u}.
\end{array}$$
Subsequently, it follows that $\Gamma_{g}^{5},\ \Gamma_{h}^{5},\ \Gamma_{i}^{4},\ \Gamma_{j}^{4},\ \Gamma_{k}^{4},\ \Gamma_{l}^{4},\ \Gamma_{m}^{5}$, and $\Gamma_{n}^{5}$ in (\ref{eq25}). Therefore, we obtain the form of the matrix $\Gamma$ in (\ref{eq24}).
\end{proof}

\subsection{Three important lemmas}

Due to the complicated form of $\Gamma$ in (\ref{eq24}), in this part, we first present three lemmas, addressing the positive-linear independence of three submatrices in $\Gamma$ marked by blue, green and yellow, respectively. To proceed from here on, we define the size of each index set in (\ref{eq23}) and Propositions \ref{pro7}-\ref{pro9} as follows. We denote the size of the index set $I_{G_1}$ by $S_1$, that is, $ \mid \! I_{G_1}\mid =S_1$. Similarly,
$$
\begin{array}{lllll} \mid \! I_{G_2}\! \mid=S_2,& \mid \! I_{G_3}\! \mid=S_3,&
\mid \! I_{G_4} \! \mid=S_4, &{} \\
\mid \! I_{H_1}\! \mid=U_1, & \mid \! I_{H_2}\! \mid=U_2, & \mid \! I_{H_3}\! \mid=U_3, & \mid \! I_{H_4}\! \mid=U_4,&{}\\
\mid \! I_{GH_1}\! \mid=W_1, & \mid \! I_{GH_3}\! \mid =W_2,& \mid \! I_{GH_4}\! \mid=W_3, & {} &{}\\
\mid \! \Lambda_{1}\! \mid=D_1, & \mid \! \Lambda_{2}\! \mid=D_2,& \mid \! \Lambda^{+}_{3}\! \mid=D_3, & \mid \! \Lambda^{c}_{3}\! \mid=D_4, & \mid \! \Lambda_{u}\! \mid=D_5,\\
\mid \! \Psi^{0}_{1}\! \mid=N_1, & \mid \! \Psi^{+}_{1}\! \mid=N_2, & \mid \! \Psi_{2}\! \mid =N_3, & \mid \! \Psi_{3}\! \mid=N_4.&{}
\end{array}
$$
Further, we denote the index corresponding to each row in the matrices $\Gamma_{a}^{3},\ \cdots\ \Gamma_{n}^{5}$ in (\ref{eq25}) by $a_{s},\ \cdots\ n_{s}$, respectively.
\begin{lemma}\label{lem1}
The row vectors in the following matrix
\be \label{eqb34}
\left[\begin{array}{l}\Gamma_{c}^{2}\\ \Gamma_{d}^{2} \\ \Gamma_{e}^{2}
\end{array}\right]=\left[\begin{array}{cccc}I_{(I_{GH_{1}},\ \Psi^{0}_{1})}&  \mathbf{0}_{(I_{GH_{1}},\ \Psi^{+}_{1})}& \mathbf{0}_{(I_{GH_{1}},\ \Psi_{2})} & \mathbf{0}_{(I_{GH_{1}},\  \Psi_{3})} \\
 \mathbf{0}_{(I_{H_{1}},\ \Psi^{0}_{1})}&  \mathbf{0}_{(I_{H_{1}},\  \Psi^{+}_{1})}&I_{(I_{H_{1}},\  \Psi_{2})} & \mathbf{0}_{(I_{H_{1}},\ \Psi_{3})} \\
 \mathbf{0}_{(I_{G_{2}},\  \Psi^{0}_{1})}& \mathbf{0}_{(I_{G_{2}},\  \Psi^{+}_{1})}& \mathbf{0}_{(I_{G_{2}},\  \Psi_{2})} & -I_{(I_{G_{2}},\ \Psi_{3})} \\
\end{array}\right]
\ee
 are positive-linearly independent.
\end{lemma}
\begin{proof} Assume that there exist $\overline{\rho}^{c} \in \mathbb{R}^{W_{1}}\ \text{and}\ \overline{\rho}^{c} \geq \mathbf{0},\ \overline{\rho}^{d} \in \mathbb{R}^{U_{1}} \ \text{and}\ \overline{\rho}^{d} \geq \mathbf{0},\ \overline{\rho}^{e} \in \mathbb{R}^{S_{2}} \ \text{and}\ \overline{\rho}^{e} \geq \mathbf{0},$ such that
$$
 \sum \limits_{s=1}^{W_1} \rho_{s}^{c}\left[\begin{array}{c}
 e^{W_{1}}_{c_{s}} \\
 \mathbf{0}_{N_2}\\
 \mathbf{0}_{N_3}\\
 \mathbf{0}_{N_4}
 \end{array}\right]+\sum \limits_{s=1}^{U_1} \rho_{s}^{d}\left[\begin{array}{c}
 \mathbf{0}_{N_{1}}\\
 \mathbf{0}_{N_2}\\
 e^{U_1}_{d_{s}}\\
 \mathbf{0}_{N_{4}}\\
 \end{array}\right]+\sum \limits_{s=1}^{S_2} \rho_{s}^{e}\left[\begin{array}{c}
 \mathbf{0}_{N_{1}}\\
 \mathbf{0}_{N_2}\\
 \mathbf{0}_{N_3}\\
 -e^{S_2}_{e_{s}}
 \end{array}\right]=\mathbf{0}.
$$
 {The above equation} is equivalent to the following system
\be \label{eq37}
 \left[\begin{array}{c}\overline{\rho}^{c} \\  \mathbf{0}_{N_2} \\ \overline{\rho}^{d}\\
-\overline{\rho}^{e}
\end{array}\right]=\mathbf{0}.
\ee
Since $\overline{\rho}^{c} \geq \mathbf{0},\ \overline{\rho}^{d} \geq \mathbf{0},\ \overline{\rho}^{e} \geq \mathbf{0}$, we get $\overline{\rho}^{c} = \mathbf{0},\ \overline{\rho}^{d} = \mathbf{0},\ \overline{\rho}^{e} = \mathbf{0}$ from Equation (\ref{eq37}).
Therefore, the row vectors in the matrix {\rm (\ref{eqb34})} are positive-linearly independent.
\end{proof}
\begin{lemma}\label{lem2}
The row vectors in the following matrix
\be \label{eq38}
\left[\begin{array}{l}\Gamma_{a}^{3}\\ \Gamma_{b}^{3}\\ \Gamma_{f}^{3}
\end{array}\right]=\left[\begin{array}{cccc}\mathbf{0}_{(\Psi^{+}_{1},\ \Psi^{0}_{1}) }&  I_{(\Psi^{+}_{1},\  \Psi^{+}_{1})}& \mathbf{0}_{(\Psi^{+}_{1},\ \Psi_{2})} & \mathbf{0}_{(\Psi^{+}_{1},\ \Psi_{3})} \\
\mathbf{0}_{(\Psi_{3},\ \Psi^{0}_{1})}&  \mathbf{0}_{(\Psi_{3},\ \Psi^{+}_{1})}& \mathbf{0}_{(\Psi_{3},\ \Psi_{2})} & I_{(\Psi_{3},\  \Psi_{3})}\\
I_{(I_{GH_{1}},\ \Psi^{0}_{1}) }& \mathbf{0}_{(I_{GH_{1}},\ \Psi^{+}_{1})}& \mathbf{0}_{(I_{GH_{1}},\  \Psi_{2})} & \mathbf{0}_{(I_{GH_{1}},\ \Psi_{3})}\\
I_{(\Psi^{0}_{1},\ \Psi^{0}_{1})}& \mathbf{0}_{(\Psi^{0}_{1},\ \Psi^{+}_{1})}& \mathbf{0}_{(\Psi^{0}_{1},\ \Psi_{2})} & \mathbf{0}_{(\Psi^{0}_{1},\  \Psi_{3})}\\
\mathbf{0}_{(\Psi^{+}_{1},\ \Psi^{0}_{1}) }& I_{(\Psi^{+}_{1},\  \Psi^{+}_{1})}& \mathbf{0}_{(\Psi^{+}_{1},\  \Psi_{2})} & \mathbf{0}_{(\Psi^{+}_{1},\  \Psi_{3})}\\
\mathbf{0}_{(\Psi_{2},\ \Psi^{0}_{1})}&  \mathbf{0}_{(\Psi_{2},\ \Psi^{+}_{1})}&I_{(\Psi_{2},\ \Psi_{2})} & \mathbf{0}_{(\Psi_{2},\ \Psi_{3})}
\end{array}\right]
\ee
 are positive-linearly independent.
\end{lemma}
 \begin{proof} Assume that there exist $\overline{\rho}^{a} \in \mathbb{R}^{S_{1}}\ \text{and}\ \overline{\rho}^{a} \geq \mathbf{0},\ \overline{\rho}^{b} \in \mathbb{R}^{W_{1}} \ \text{and}\ \overline{\rho}^{b} \geq \mathbf{0},\ \overline{\rho}^{f} \in \mathbb{R}^{U_{2}} \ \text{and}\ \overline{\rho}^{f} \geq \mathbf{0},$ such that
 $$
 \sum \limits_{s=1}^{S_1} \rho_{s}^{a}\left[\begin{array}{c}
 \mathbf{0}_{N_1+N_3}\\
 e^{S_{1}}_{a_{s}}
 \end{array}\right]+\sum \limits_{s=1}^{W_1} \rho_{s}^{b}\left[\begin{array}{c}
 e^{W_{1}}_{b_{s}}\\
 \mathbf{0}_{N_2+N_3+N_{4}}
 \end{array}\right]+\sum \limits_{s=1}^{U_2} \rho_{s}^{f}\left[\begin{array}{c}
 e^{U_{2}}_{f_{s}}\\
 \mathbf{0}_{N_4}
 \end{array}\right]=\mathbf{0}.
$$
 {The above equation} is equivalent to the following system
\be \label{eq41}
 \left[\begin{array}{c}\overline{\rho}^{b}+\overline{\rho}_{\Psi^{0}_1}^{f}\\
 \overline{\rho}_{\Psi^{+}_{1}}^{a}+\overline{\rho}_{\Psi^{+}_{1}}^{f} \\
 \overline{\rho}_{\Psi_2}^{f}\\
 \overline{\rho}_{\Psi_3}^{a}
\end{array}\right]=\mathbf{0}.
\ee
Since $\overline{\rho}^{a} \geq \mathbf{0},\ \overline{\rho}^{b} \geq \mathbf{0},\ \overline{\rho}^{f} \geq \mathbf{0}$, we get $\overline{\rho}^{a} = \mathbf{0},\ \overline{\rho}^{b} = \mathbf{0},\ \overline{\rho}^{f} = \mathbf{0}$ from Equation (\ref{eq41}).
Therefore, the row vectors in the matrix {\rm (\ref{eq38})} are positive-linearly independent.
\end{proof}
\begin{lemma}\label{lem3}
The row vectors in the matrix $\Gamma_{sub}$ defined by
\be
\label{eqb3}
\begin{array}{c}
\Gamma_{sub} =\left[\begin{array}{cc}(B B^{\top})_{(I_{G_{3}},\ \cdot \ )} & \Gamma_{g}^{5}\\ (B B^{\top})_{(I_{GH_{3}},\ \cdot \ )} & \Gamma_{h}^{5}\\ \Gamma_{i}^{4} & \mathbf{0}_{(I_{GH_{3}},\ L_{5} )} \\ \Gamma_{j}^{4} & \mathbf{0}_{(I_{H_{3}},\ L_{5}) } \\
\mathbf{0}_{(I_{GH_{4}},\ L_{4})} & \Gamma_{m}^{5} \\ \mathbf{0}_{(I_{H_{4}},\ L_{4})} & \Gamma_{n}^{5}
\end{array}\right]
\end{array}
\ee
are positive-linearly independent.
\end{lemma}
\begin{proof}
 For the convenience of analysis, note that
\[
\left[\begin{array}{c}
\Gamma_{g}^{5} \\ \Gamma_{h}^{5} \\ \Gamma_{m}^{5} \\ \Gamma_{n}^{5}
\end{array} \right]=\left[\begin{array}{ccccc}
\mathbf{0}_{(\Lambda^{+}_{3},\ \Lambda_{1})} & \mathbf{0}_{(\Lambda^{+}_{3},\  \Lambda_{2})} & I_{(\Lambda^{+}_{3},\ \Lambda^{+}_{3})} & \mathbf{0}_{(\Lambda^{+}_{3},\ \Lambda^{c}_{3})} & \mathbf{0}_{(\Lambda^{+}_{3},\ \Lambda_{u})}\\
\mathbf{0}_{(\Lambda^{c}_{3},\ \Lambda_{1})} & \mathbf{0}_{(\Lambda^{c}_{3},\  \Lambda_{2})} & \mathbf{0}_{(\Lambda^{c}_{3},\ \Lambda^{+}_{3})} & I_{(\Lambda^{c}_{3},\ \Lambda^{c}_{3})} & \mathbf{0}_{(\Lambda^{c}_{3},\ \Lambda_{u})}\\
\mathbf{0}_{(\Lambda_{u},\ \Lambda_{1})} & \mathbf{0}_{(\Lambda_{u},\ \Lambda_{2})} & \mathbf{0}_{(\Lambda_{u},\ \Lambda^{+}_{3})} & \mathbf{0}_{(\Lambda_{u},\ \Lambda^{c}_{3})} & I_{(\Lambda_{u},\ \Lambda_{u})}\\
I_{(I_{GH_{3}},\ \Lambda_{1})} & \mathbf{0}_{(I_{GH_{3}},\ \Lambda_{2})} & \mathbf{0}_{(I_{GH_{3}},\ \Lambda^{+}_{3})} & \mathbf{0}_{(I_{GH_{3}},\ \Lambda^{c}_{3})} & \mathbf{0}_{(I_{GH_{3}},\  \Lambda_{u})}\\
\mathbf{0}_{(I_{GH_{4}},\ \Lambda_{1})} & \mathbf{0}_{(I_{GH_{4}},\ \Lambda_{2})} & \mathbf{0}_{(I_{GH_{4}},\ \Lambda^{+}_{3})} & I_{(I_{GH_{4}},\ \Lambda^{c}_{3})} & \mathbf{0}_{(I_{GH_{4}},\ \Lambda_{u})}\\
I_{(\Lambda_{1},\ \Lambda_{1})} & \mathbf{0}_{(\Lambda_{1},\ \Lambda_{2})} & \mathbf{0}_{(\Lambda_{1},\ \Lambda^{+}_{3})} & \mathbf{0}_{(\Lambda_{1},\ \Lambda^{c}_{3})} & \mathbf{0}_{(\Lambda_{1},\ \Lambda_{u})}\\
\mathbf{0}_{(\Lambda_{2},\ \Lambda_{1})} & I_{(\Lambda_{2},\ \Lambda_{2})} & \mathbf{0}_{(\Lambda_{2},\ \Lambda^{+}_{3})} & \mathbf{0}_{(\Lambda_{2},\ \Lambda^{c}_{3})} & \mathbf{0}_{(\Lambda_{2},\ \Lambda_{u})}\\
\mathbf{0}_{(\Lambda^{+}_{3},\ \Lambda_{1})} & \mathbf{0}_{(\Lambda^{+}_{3},\ \Lambda_{2})} & I_{(\Lambda^{+}_{3},\ \Lambda^{+}_{3})} & \mathbf{0}_{(\Lambda^{+}_{3},\ \Lambda^{c}_{3})} & \mathbf{0}_{(\Lambda^{+}_{3},\  \Lambda_{u})}\\
 \end{array} \right],
\]
and assume that we can find some vectors $\overline{\rho}^{g} \in \mathbb{R}^{S_{3}}$ and $\overline{\rho}^{g} \geq \mathbf{0}$, $\overline{\rho}^{h} \in \mathbb{R}^{W_{2}}$ and $\overline{\rho}^{h} \geq \mathbf{0}$, $\overline{\rho}^{i} \in \mathbb{R}^{W_{2}}$ and $\overline{\rho}^{i} \geq \mathbf{0}$, $\overline{\rho}^{j} \in \mathbb{R}^{U_{3}}$ and $\overline{\rho}^{j} \geq \mathbf{0}$, $\overline{\rho}^{m} \in \mathbb{R}^{W_{3}}$ and $\overline{\rho}^{m} \geq \mathbf{0}$, and $\overline{\rho}^{n} \in \mathbb{R}^{U_{4}}$ and $\overline{\rho}^{n} \geq \mathbf{0}$, such that 
\begin{small}
$$
\begin{array}{l}
 \sum \limits_{s=1}^{S_3} \rho_{s}^{g}\left[\begin{array}{c} \left(BB^{\top}\right)_{(g_{s},\ \cdot \ )}^{\top}\\ \left[\begin{array}{c}\mathbf{0}_{D_{1}+D_{2}}\\ e^{S_3}_{g_{s}} \end{array}\right]
 \end{array}\right]\!+\!\sum \limits_{s=1}^{W_2} \rho_{s}^{h}\left[\begin{array}{c}\left(BB^{\top}\right)_{(h_{s},\ \cdot \ )}^{\top} \\ \left[\begin{array}{c} e^{W_{2}}_{h_{s}}\\
 \mathbf{0}_{T m_{2} -D_{1}} \end{array}\right]\end{array}\right]
 \!+\!\sum \limits_{s=1}^{W_2} \rho_{s}^{i}
 \left[\begin{array}{c}
 \left[\begin{array}{c}
 e^{W_2}_{i_{s}} \\
 \mathbf{0}_{T m_{2} -D_{1}} \end{array}\right] \\ \mathbf{0}_{T m_{2}}\end{array}\right]\!+\!\\
 \sum \limits_{s=1}^{U_3} \rho_{s}^{j}\left[\begin{array}{c}
 \left[\begin{array}{c}
 \mathbf{0}_{D_{1}} \\ e^{U_3}_{j_{s}}\\ \mathbf{0}_{D_{3}+D_{4}+D_{5}}\end{array}\right] \\ \mathbf{0}_{T m_{2}}\end{array}\right]\!+\!\sum \limits_{s=1}^{W_3} \rho_{s}^{m}\left[\begin{array}{c} \mathbf{0}_{T m_{2}}\\ \left[\begin{array}{c} \mathbf{0}_{(D_{1}+D_{2}+D_{3})} \\ e^{W_3}_{m_{s}}\\
 \mathbf{0}_{D_{5}} \\ \end{array}\right]
 \end{array}\right]\!+\!\sum \limits_{s=1}^{U_4} \rho_{s}^{n}\left[\begin{array}{c} \mathbf{0}_{T m_{2}}\\ \left[\begin{array}{c} e^{U_{4}}_{n_{s}}\\
 \mathbf{0}_{D_{4}+D_{5}} \\ \end{array}\right]
 \end{array}\right]\!= \!\mathbf{0}.
 \end{array}
$$
\end{small}
 {The above equation} is equivalent to the compact system
\[
\left[\begin{array}{c}
\sum \limits_{s=1}^{S_3}\rho_{s}^{g}\left(\left(BB^{\top}\right)_{(g_{s},\ \cdot)}\right)^{\top}+\sum \limits_{s=1}^{W_2} \left(\left(BB^{\top}\right)_{(h_{s},\ \cdot)}\right)^{\top}+\left[\begin{array}{c} \overline{\rho}^{i}\\ \overline{\rho}^{j}\\ \mathbf{0}_{D_3+D_4+D_5 }
\end{array}\right] \\
\overline{\rho}^{h}+\overline{\rho}_{\Lambda_1}^{n} \\
\overline{\rho}_{\Lambda_2}^{n}\\
\overline{\rho}_{\Lambda^{+}_{3}}^{g}+\overline{\rho}_{\Lambda^{+}_{3}}^{n} \\
\overline{\rho}_{\Lambda^{c}_{3}}^{g}+\overline{\rho}^{m}\\ \overline{\rho}_{\Lambda_{u}}^{g}
\end{array}\right]=\mathbf{0},
\]
which leads to $\overline{\rho}^{g} = \mathbf{0},\ \overline{\rho}^{h} = \mathbf{0},\ \overline{\rho}^{i} = \mathbf{0},\ \overline{\rho}^{j} = \mathbf{0},\ \overline{\rho}^{m} = \mathbf{0},\ \overline{\rho}^{n} = \mathbf{0}$, given that $\overline{\rho}^{g} \geq \mathbf{0},\ \overline{\rho}^{h} \geq \mathbf{0},\ \overline{\rho}^{i} \geq \mathbf{0},\ \overline{\rho}^{j} \geq \mathbf{0},\ \overline{\rho}^{m} \geq \mathbf{0},\ \overline{\rho}^{n} \geq \mathbf{0}$.
Therefore, the row vectors in the matrix $\Gamma_{sub}$ are positively-linearly independent.
\end{proof}
\subsection{The main result}
Based on the above lemmas, we are ready to present the main theorem on the MPEC-MFCQ.
\begin{theorem} \label{thr3} Let $v=(C,\zeta,z,\alpha,\xi)$ be any feasible point for the MPEC {\rm (\ref{eq16})}, then  $v$  satisfies the MPEC-MFCQ.
\end{theorem}
\begin{proof}
Assume there exist $\overline{\rho}^{a} \in \mathbb{R}^{S_{1}}\ \text{and}\ \overline{\rho}^{a} \geq \mathbf{0},\ \overline{\rho}^{b} \in \mathbb{R}^{W_{1}}\ \text{and}\ \overline{\rho}^{b} \geq \mathbf{0},\ \overline{\rho}^{c} \in \mathbb{R}^{W_{1}}\ \text{and}\ \overline{\rho}^{c} \geq \mathbf{0},\ \overline{\rho}^{d} \in \mathbb{R}^{U_{1}}\ \text{and}\ \overline{\rho}^{d} \geq \mathbf{0},\ \overline{\rho}^{e} \in \mathbb{R}^{S_{2}}\ \text{and}\ \overline{\rho}^{e} \geq \mathbf{0},\ \overline{\rho}^{f} \in \mathbb{R}^{U_{2}}\ \text{and}\ \overline{\rho}^{f} \geq \mathbf{0},\ \overline{\rho}^{g} \in \mathbb{R}^{S_{3}}\ \text{and}\ \overline{\rho}^{g} \geq \mathbf{0},\ \overline{\rho}^{h} \in \mathbb{R}^{W_{2}} \ \text{and}\ \rho^{h} \geq \mathbf{0},\ \overline{\rho}^{i} \in \mathbb{R}^{W_{2}} \ \text{and}\ \overline{\rho}^{i} \geq \mathbf{0},\ \overline{\rho}^{j} \in \mathbb{R}^{U_{3}}\ \text{and}\ \overline{\rho}^{j} \geq \mathbf{0},\ \overline{\rho}^{k} \in \mathbb{R}^{S_{4}}\ \text{and}\ \overline{\rho}^{k} \geq \mathbf{0},\ \overline{\rho}^{l} \in \mathbb{R}^{W_{3}}\ \text{and}\ \overline{\rho}^{l} \geq \mathbf{0},\ \overline{\rho}^{m} \in \mathbb{R}^{W_{3}} \ \text{and}\ \overline{\rho}^{m} \geq \mathbf{0},\ \overline{\rho}^{n} \in \mathbb{R}^{U_{4}} \ \text{and}\ \overline{\rho}^{n} \geq \mathbf{0},$ such that the following holds
$$
\sum \limits_{s=1}^{S_1}\rho_{s}^{a}\left[\begin{array}{c}
0 \\
\mathbf{0}_{T m_{1}} \\ \left(\Gamma_{a}^{3}\right)_{(a_{s},\ \cdot \ )}^{\top} \\ \left(AB^{\top}\right)_{(a_{s},\ \cdot \ )}^{\top}\\ \mathbf{0}_{T m_{2}}
\end{array}\right]+
\sum \limits_{s=1}^{W_1}\rho_{s}^{b}\left[\begin{array}{c}
0 \\
\mathbf{0}_{T m_{1}} \\ \left(\Gamma_{b}^{3}\right)_{(b_{s},\ \cdot \ )}^{\top} \\ \left(AB^{\top}\right)_{(b_{s},\ \cdot \ )}^{\top} \\ \mathbf{0}_{T m_{2}}
\end{array}\right]
+\sum \limits_{s=1}^{W_1}\rho_{s}^{c}\left[\begin{array}{c}
0 \\ \left(\Gamma_{c}^{2}\right)_{(c_{s},\ \cdot \ )}^{\top} \\ \mathbf{0}_{T m_{1}} \\ \mathbf{0}_{T m_{2}} \\ \mathbf{0}_{T m_{2}}
\end{array}\right]
$$
\be \label{eq43}
\begin{array}{l}
+\sum \limits_{s=1}^{U_1}\rho_{s}^{d}\left[\begin{array}{c}
0 \\ \left(\Gamma_{d}^{2}\right)_{(d_{s},\ \cdot \ )}^{\top} \\ \mathbf{0}_{T m_{1} } \\ \mathbf{0}_{T m_{2} } \\ \mathbf{0}_{T m_{2} }
\end{array}\right]+\sum \limits_{s=1}^{S_2}\rho_{s}^{e}\left[\begin{array}{c}
0 \\ \left(\Gamma_{e}^{2}\right)_{(e_{s},\ \cdot \ )}^{\top} \\ \mathbf{0}_{T m_{1}} \\ \mathbf{0}_{T m_{2} } \\ \mathbf{0}_{T m_{2} }
\end{array}\right]
+\sum \limits_{s=1}^{U_2}\rho_{s}^{f}\left[\begin{array}{c}
0 \\
\mathbf{0}_{T m_{1} } \\ \left(\Gamma_{f}^{3}\right)_{(f_{s},\ \cdot \ )}^{\top} \\ \mathbf{0}_{T m_{2} } \\ \mathbf{0}_{T m_{2}}
\end{array}\right]\\
+\sum \limits_{s=1}^{S_3}\rho_{s}^{g}\left[\begin{array}{c}
0\\
\mathbf{0}_{T m_{1}} \\ \mathbf{0}_{T m_{1}} \\ \left(BB^{\top}\right)_{(g_{s},\ \cdot \ )}^{\top} \\ \left(\Gamma_{g}^{5}\right)_{(g_{s},\ \cdot \ )}^{\top}
\end{array}\right]+
\sum \limits_{s=1}^{W_2}\rho_{s}^{h}\left[\begin{array}{c}
0 \\
\mathbf{0}_{T m_{1} } \\ \mathbf{0}_{T m_{1} } \\ \left(BB^{\top}\right)_{(h_{s},\ \cdot \ )}^{\top} \\ \left(\Gamma_{h}^{5}\right)_{(h_{s},\ \cdot \ )}^{\top}
\end{array}\right] +
\sum \limits_{s=1}^{W_2}\rho_{s}^{i}\left[\begin{array}{c}
0 \\
\mathbf{0}_{T m_{1} } \\ \mathbf{0}_{T m_{1}} \\ \left(\Gamma_{i}^{4}\right)_{(i_{s},\ \cdot \ )}^{\top} \\ \mathbf{0}_{T m_{2}}
\end{array}\right]\\
+\sum \limits_{s=1}^{U_3}\rho_{s}^{j}\left[\begin{array}{c}
0 \\
\mathbf{0}_{T m_{1} } \\ \mathbf{0}_{T m_{1} } \\ \left(\Gamma_{j}^{4}\right)_{(j_{s},\ \cdot \ )}^{\top} \\ \mathbf{0}_{T m_{2} }
\end{array}\right] +\sum \limits_{s=1}^{S_4}\rho_{s}^{k}\left[\begin{array}{c}
1 \\
\mathbf{0}_{T m_{1}} \\ \mathbf{0}_{T m_{1}} \\ \left(\Gamma_{k}^{4}\right)_{(k_{s},\ \cdot \ )}^{\top} \\ \mathbf{0}_{T m_{2}}
\end{array}\right]
+\sum \limits_{s=1}^{W_3}\rho_{s}^{l}\left[\begin{array}{c}
1 \\
\mathbf{0}_{T m_{1} } \\ \mathbf{0}_{T m_{1}} \\ \left(\Gamma_{l}^{4}\right)_{(l_{s},\ \cdot \ )}^{\top} \\ \mathbf{0}_{T m_{2}}
\end{array}\right]\\
+\sum \limits_{s=1}^{W_3}\rho_{s}^{m}\left[\begin{array}{c}
0 \\
\mathbf{0}_{T m_{1}} \\ \mathbf{0}_{T m_{1}} \\ \mathbf{0}_{T m_{2} } \\ \left(\Gamma_{m}^{5}\right)_{(m_{s},\ \cdot \ )}^{\top}
\end{array}\right]
+\sum \limits_{s=1}^{U_4}\rho_{s}^{n}\left[\begin{array}{c}
0 \\
\mathbf{0}_{T m_{1} } \\ \mathbf{0}_{T m_{1}} \\ \mathbf{0}_{T m_{2} } \\ \left(\Gamma_{n}^{5}\right)_{(n_{s},\ \cdot \ )}^{\top}
\end{array}\right]=\mathbf{0}.
\end{array}
\ee
From the first row in Equation (\ref{eq43}), we get $\sum \limits_{s=1}^{S_4} \rho_{s}^{k}+\sum \limits_{s=1}^{W_3} \rho_{s}^{l}=0$. Together with the fact that $\overline{\rho}^{k} \geq \mathbf{0},\ \overline{\rho}^{l}\geq \mathbf{0}$, we get $\overline{\rho}^{k}=\mathbf{0}$ and $\overline{\rho}^{l}=\mathbf{0}$. From Lemma \ref{lem1}, we get $\overline{\rho}^{c} = \mathbf{0},\ \overline{\rho}^{d} = \mathbf{0},\ \overline{\rho}^{e} = \mathbf{0}$ in Equation (\ref{eq43}). From Lemma \ref{lem2}, we get $\overline{\rho}^{a} = \mathbf{0},\ \overline{\rho}^{b} = \mathbf{0},\ \overline{\rho}^{f} = \mathbf{0}$ in Equation (\ref{eq43}). From Lemma \ref{lem3}, we get $\overline{\rho}^{g} = \mathbf{0},\ \overline{\rho}^{h} = \mathbf{0},\ \overline{\rho}^{i} = \mathbf{0},\ \overline{\rho}^{j} = \mathbf{0},\ \overline{\rho}^{m} = \mathbf{0},\ \overline{\rho}^{n} = \mathbf{0}$ in Equation (\ref{eq43}).

In summary, the row vectors in the matrix $\Gamma$ (\ref{eq24}) are positive-linearly independent at every feasible point $v$ for the MPEC (\ref{eq16}). That is to say, every feasible point $v$ for the MPEC  (\ref{eq16}) satisfies the MPEC-MFCQ.
\end{proof}
\section{Numerical results}\label{sec5}
In this section, we { {present the GR-CV, which is a concrete implementation of the GRM in Algorithm \ref{alo2} for selecting the hyperparameter $C$ in SVC, as shown in Algorithm \ref{alo1}. We show
numerical results of the proposed GR-CV, and compare it with other approaches.}}
\begin{algorithm}
\caption{The Global Relaxation Cross-Validation Algorithm (GR-CV)}\label{alo1}
\begin{algorithmic}[1]
\State Given $T$, split the data set into a subset $\Omega$ with $l_{1}$ points and a hold-out test set $\Theta$ with $l_{2}$ points. The set $\Omega$ is equally partitioned into $T$ pairwise disjoint subsets, one for each fold. 
\State {\bf Select} an optimal hyperparameter $\widehat{C}$ by the GRM in Algorithm \ref{alo2}.
\State {\bf Post-processing procedure.} The regularization  {hyperparameter} $\widehat{C}$ is rescaled by a factor $\frac{T}{T-1}$.
Then, an $l_{1}$-loss SVC problem is solved on the subset $\Omega$ using $\frac{T}{T-1} \widehat{C}$ by ALM-SNCG algorithm in \cite{Yan2020efficient}. This gives the final classifier $\widehat{w}$.
\end{algorithmic}
\end{algorithm}

All the numerical tests are conducted in Matlab R2018a on a Windows 7 Dell Laptop with an Intel(R) Core(TM) i5-6500U CPU at 3.20GHz and 8 GB of RAM. All the data sets are collected from the LIBSVM library: \href{csie.ntu.edu.tw/cjlin/
libsvmtools/datasets}{https://www.csie.ntu.edu.tw/cjlin/
libsvmtools/datasets/}. Each data set is split
into a subset $\Omega$ with $l_{1}$ points (it is used for cross-validation) and a hold-out test set $\Theta$ with $l_{2}$ points. The data descriptions are shown in Table 1.
\begin{table}[!htbp]
\centering
\caption{Descriptions of data sets.}\label{table1}
\begin{tabular}{cccccccc}
\toprule
Data set & $l_{1}$  & $l_{2}$ & n & Data set & $l_{1}$  & $l_{2}$ & n\\
\midrule
heart & 189 & 81 & 13  & splice & 300 & 700 & 60 \\
breast & 240 & 172 & 10 & fourclass & 300 &562 & 2  \\
colon-cancer & 36 &26 &2000 & w1a & 240 &260 & 300 \\
ionosphere & 246 & 105 & 34 & w2a & 300 & 500 & 300\\
australian & 270 & 420 & 14 & a1a & 300 & 200 & 119  \\
diabetes & 270 & 498 & 8 & german.number &207 &793 & 24 \\
\bottomrule
\end{tabular}
\end{table}

We compare { {our GR-CV}} with two other approaches: the inexact cross-validation method (In-CV) and the grid search algorithm (G-S). In-CV \cite{kunapuli2008classification} is a relaxation method based on the relaxation of the complementarity constraints by a prescribed tolerance parameter $\mathbf{tol} > 0$. That is, solving (\ref{eq50}) with $t_{k}=\mathbf{tol}$ as a fixed tolerance rather than decreasing $t_{k}$ gradually.

The parameters of three methods are set as follows. For GR-CV, we set the initial values as $v_{0}=\left[1,\ \mathbf{0}_{1 \times \overline{m}}\right]^{\top}$, $t_{0}=1,\ t_{\min}=10^{-8},\ \sigma=0.01.$ The relaxed subproblems (\ref{eq50}) are solved by the \texttt{snsolve} function, which is part of the SNOPT solver \cite{philip2002user}.
 For In-CV, we use the same $v_{0}$ as in GR-CV and $\mathbf{tol}=10^{-4}$. For G-S, a typical grid range for $C$ would be $C \in \{10^{-4},\ 10^{-3},\ 10^{-2}$, $10^{-1},\ 1,\ 10^{1},\ 10^{2},\ 10^{3},\ 10^{4}\}$ \cite{kunapuli2008classification}. In each training process, the ALM-SNCG algorithm from \cite{Yan2020efficient}, which is outstanding and competitive with the most popular methods in LIBLINEAR (\href{csie.ntu.edu.tw/~cjlin/liblinear}{https://www.csie.ntu.edu.tw/~cjlin/liblinear/}) in both speed and accuracy, is used to solve the $l_{1}$-loss SVC problem.

We compare the aforementioned methods in the following three aspects:
 \begin{itemize}
 \item [1.] Test error ($E_{t}$) as defined by
$$
E_{t}=\frac{1}{l_{2}} \sum_{(x, y) \in \Theta} \frac{1}{2} \mid \operatorname{sign}\left(\widehat{w}^{\top} x\right)-y \mid,
$$
which is a measure of the ability of generalization.
 \item [2.] CV error ($E_{C}$) as defined in the objective function of problem (\ref{eq31}).
 \item [3.] The number of iterations $k$ for an algorithm, and the total number of iterations $it$ for solving the subproblems  (short for $(k,\ it$)).
 \end{itemize}
We also report the maximum violation of all constraints defined as in (\ref{eqc1}), to measure the feasibility of the final solution given by GR-CV and In-CV.

The results are reported in Table \ref{table2}, where we mark the winners of test error $E_{t}$, CV error $E_{C}$ and the maximum violation of all constraints Vio in bold. We also show the comparisons { {of the three}} methods for different data sets on test error $E_{t}$ and CV error $E_{C}$ in Figures \ref{fig_test} and \ref{fig_CV}, respectively. The data sets on the horizontal axis are arranged in the order shown in Table \ref{table2}.

From Figure \ref{fig_test}, Figure \ref{fig_CV} and Table \ref{table2}, we have the following observations. Firstly, GR-CV performs the best in terms of test error, implying that our approach is more capable of generalization. Secondly, in terms of test error in Figure \ref{fig_test}, GR-CV is competitive with G-S. GR-CV is the winner in five data sets of all the twelve datasets where as G-S wins in eight datasets among the twelve datasets. Finally, comparing GR-CV with In-CV, the feasibility of the solution returned by GR-CV is significantly better than that by In-CV since Vio given by GR-CV is much smaller than that by In-CV. In terms of cpu time, it is obvious that In-CV takes less time than GR-CV since it only solves the relaxation problem (\ref{eq50}) once. Since G-S is basically solving a completely different type of problem to find the hyperparameter $C$, it is does not make sense to compare the cpu time between GR-CV and G-S.
\begin{table}[htbp]
\centering
\caption{Computational results for $T=3$.} \label{table2}
\begin{tabular}{ccccccc}
\toprule
{}&Data set & Method  & $E_{t}$ ($\%$) &$E_{C}$ ($\%$)& Vio& $(k,\ it)$ \\
\midrule
1 & heart  & GR-CV & \textbf{9.88} & \textbf{17.46}  &\textbf{1.51e$-$6} &(5, 24165)\\
{}&{}& In-CV &\textbf{9.88} & 17.95& 0.010 &(1,12418) \\
{}&{}& G-S &13.58 &\textbf{17.46} & $-$ & (27,425) \\
2 &breast  & GR-CV & \textbf{4.07}  &\textbf{5.42}& \textbf{4.98e$-$4}& (5,17092)  \\
{}&{}& In-CV & \textbf{4.07} & \textbf{5.42}& 0.006& (1,14971)\\
{}&{}& G-S &\textbf{4.07} & 6.25 & $-$ & (27,298) \\
3 & colon-cancer & GR-CV& \textbf{19.23}&\textbf{2.78}& \textbf{9.69e$-$5}& (5,2166) \\
{}&{}& In-CV &23.08&\textbf{2.78} & 0.005& (1,1102)\\
{}&{}& G-S & 26.92 & \textbf{2.78} & $-$ & (27,167)\\
4 & ionosphere & GR-CV&\textbf{0.95} & 27.61 &\textbf{0.03}&(5,96200) \\
{}&{}& In-CV & 1.90 &29.76& \textbf{0.03} &(1,29530)  \\
{}&{}& G-S &4.76 & \textbf{18.70} & $-$ & (27,522) \\
5 & australian & GR-CV&\textbf{14.29} & \textbf{14.44} &\textbf{3.03e$-$6} & (5,32583) \\
{}&{}& In-CV & 14.52 &14.81 &0.008 & (1,26703)\\
{}&{} & G-S & 14.52 & \textbf{14.44} & $-$ & (27,430) \\
6 & diabetes & GR-CV& \textbf{20.48} & \textbf{24.44} &\textbf{1.75e$-$5} &(5,33294) \\
{}&{}& In-CV & \textbf{20.48} & 25.18 & 0.005& (1,26558)\\
{}&{} & G-S & \textbf{20.48} & 25.19 &$-$& (27,416) \\
7 &splice & GR-CV&\textbf{23.29} & 29.01 &0.009 & (5,83306)  \\
{}&{}& In-CV & 26.29 &24.63&\textbf{0.005} &(1,24333) \\
{}&{}& G-S & \textbf{23.29} &\textbf{23.33} & $-$& (27,526) \\
8 & fourclass & GR-CV & \textbf{22.06} & 28.67 & \textbf{5.83e$-$5} & (5,17275)  \\
{}&{} &In-CV & 22.24 & 28.65 & 0.008 & (1,8989) \\
{}&{} &G-S & 22.24 & \textbf{23.33} & $-$ & (27,349) \\
9 & w1a & GR-CV & \textbf{0.00} & 23.33 & \textbf{4.26e$-$4}& (5,75793) \\
{}&{} & In-CV & \textbf{0.00} & \textbf{22.88} & 0.009 & (1,28810) \\
{}&{} &G-S & \textbf{0.00} & 30.00 & $-$ & (27,366)  \\
10 & w2a & GR-CV & \textbf{0.00} & 25.93 & \textbf{1.50e$-$4} & (5,88758)\\
{}&{} & In-CV & \textbf{0.00} & \textbf{22.11} & 0.009 &(1,31708)\\
{}&{} & G-S & \textbf{0.00} &35.67 & $-$ & (27,522)\\
11 & a1a & GR-CV & \textbf{19.50} & 15.33 & \textbf{7.64e$-$5} & (5,64349) \\
{}&{} & In-CV &\textbf{19.50} & 15.65 & 0.013& (1,36010) \\
{}&{}& G-S & 20.00 & \textbf{14.67} & $-$ & (27,533)  \\
12 &german. & GR-CV & \textbf{25.73} & 26.09 & \textbf{5.29e$-$5} & (5,33317) \\
{}&number &In-CV & 26.86 & 26.08 &0.068 & (1,24850)  \\
{}&{} &G-S &26.86 &  \textbf{25.60} & $-$ & (27,482)  \\
\bottomrule
\end{tabular}
\end{table}
\begin{figure}[htbp]
	\centering
	\includegraphics[width=0.80\textwidth]{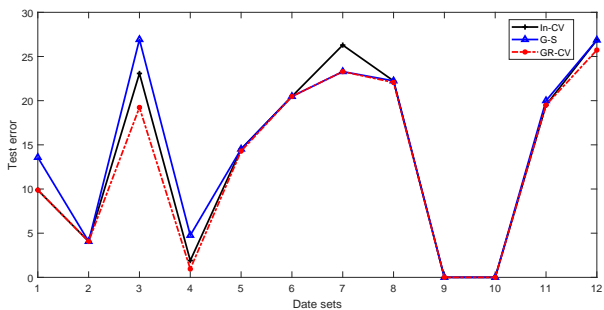}
	\caption{The comparison among the three methods on test error.}\label{fig_test}
\end{figure}
\begin{figure}[htbp]
	\centering
	\includegraphics[width=0.80\textwidth]{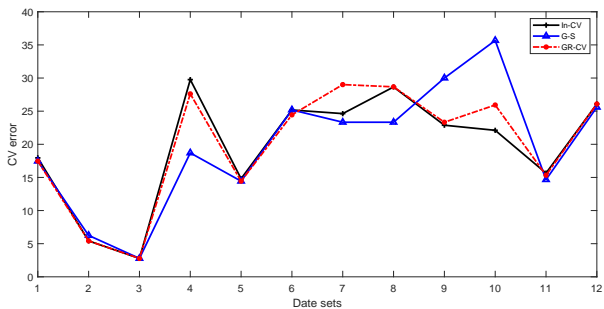}
	\caption{The comparison among the three methods on CV error.}\label{fig_CV}
\end{figure}
\begin{figure}[h]
	\centering
	\includegraphics[width=0.9\textwidth]{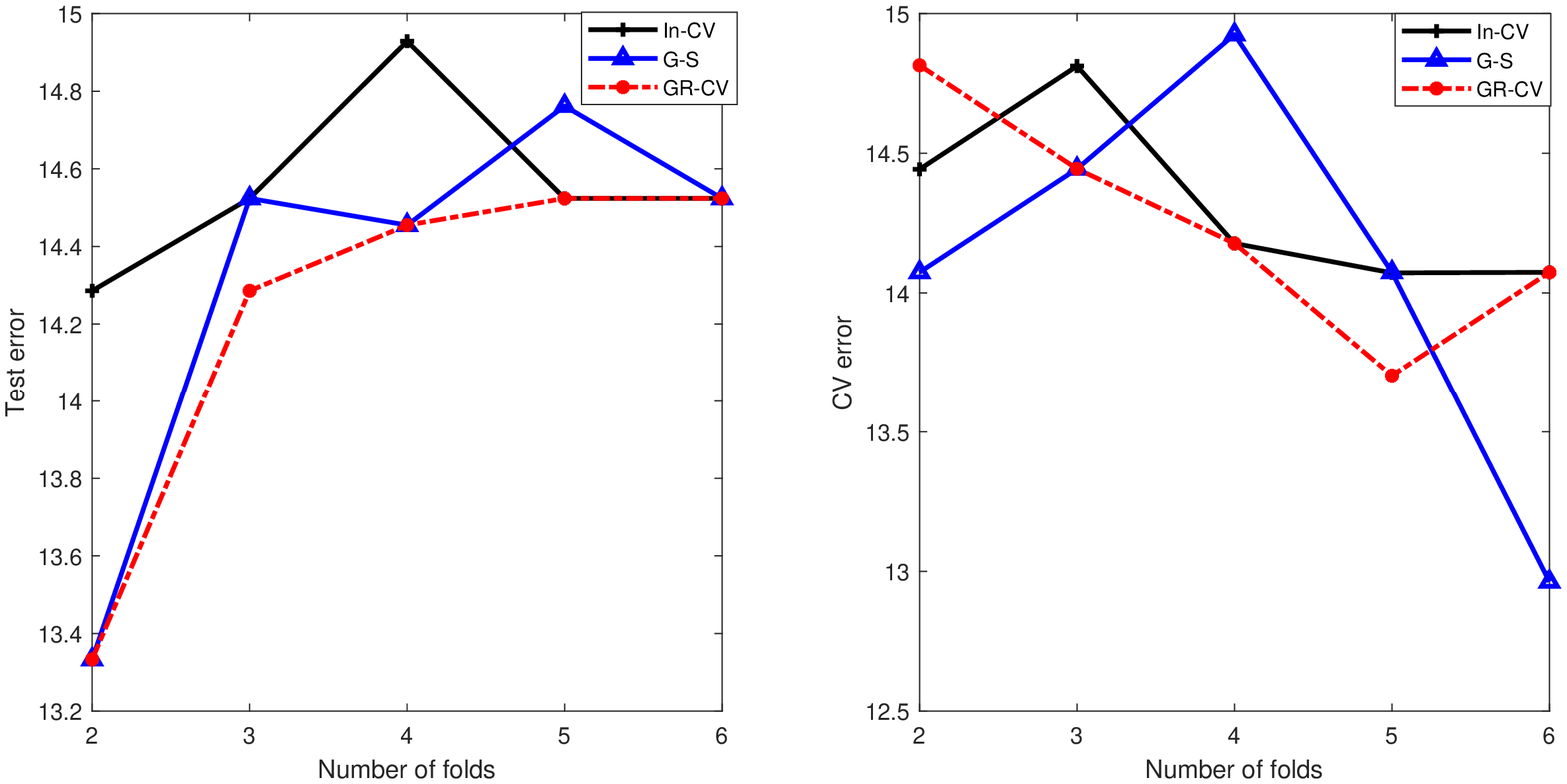}
	\caption{Effect of increasing the number of folds on test error and CV error.}\label{figCVtest}
\end{figure}

To further study the effect of increasing the number of folds on test error $E_{t}$ and CV error $E_{C}$ in the three methods, we report the results on the Australian data set in Figure \ref{figCVtest}. The results show that as $T$ changes, the test error for GR-CV is always the lowest, and the CV error for GR-CV is competitive with the other two methods. Meanwhile it is clear that larger number of folds can be successfully solved for GR-CV, the computing time grows with the number of folds because of the increasing number of variables and constraints for the MPEC to be solved. The ranges of the test error and CV error for different numbers of folds are not large, so $T=3$ represents a reasonable choice.

\section{Conclusion}\label{sec6}
We have proposed a bilevel optimization model for the hyperparameter selection for support vector classification in which the upper-level problem minimizes a T-fold cross validation error and the lower-level problems are T $l_{1}$-loss SVC problems on the training sets. We reformulated  { {the bilevel optimization problem}} into an MPEC, and  {proposed the GR-CV to solve it based on the GRM from \cite{scholtes2001convergence}.} We also proved that the MPEC-MFCQ automatically holds at each feasible point. Extensive numerical results on the data sets from the LIBSVM library demonstrated the superior generalization performance of  {the proposed method} over almost all the data sets used in this paper.

\section*{Acknowledgments.}
The work of AZ is supported by the EPSRC grant EP/V049038/1 and the Alan
Turing Institute under the EPSRC grant EP/N510129/1.

\bibliographystyle{plain-short}
{\footnotesize{
\nocite{*}

\begin{thebibliography}{54}
\ifx \bisbn   \undefined \def \bisbn  #1{ISBN #1}\fi
\ifx \binits  \undefined \def \binits#1{#1}\fi
\ifx \bauthor  \undefined \def \bauthor#1{#1}\fi
\ifx \batitle  \undefined \def \batitle#1{#1}\fi
\ifx \bjtitle  \undefined \def \bjtitle#1{#1}\fi
\ifx \bvolume  \undefined \def \bvolume#1{\textbf{#1}}\fi
\ifx \byear  \undefined \def \byear#1{#1}\fi
\ifx \bissue  \undefined \def \bissue#1{#1}\fi
\ifx \bfpage  \undefined \def \bfpage#1{#1}\fi
\ifx \blpage  \undefined \def \blpage #1{#1}\fi
\ifx \burl  \undefined \def \burl#1{\textsf{#1}}\fi
\ifx \doiurl  \undefined \def \doiurl#1{\url{https://doi.org/#1}}\fi
\ifx \betal  \undefined \def \betal{\textit{et al.}}\fi
\ifx \binstitute  \undefined \def \binstitute#1{#1}\fi
\ifx \binstitutionaled  \undefined \def \binstitutionaled#1{#1}\fi
\ifx \bctitle  \undefined \def \bctitle#1{#1}\fi
\ifx \beditor  \undefined \def \beditor#1{#1}\fi
\ifx \bpublisher  \undefined \def \bpublisher#1{#1}\fi
\ifx \bbtitle  \undefined \def \bbtitle#1{#1}\fi
\ifx \bedition  \undefined \def \bedition#1{#1}\fi
\ifx \bseriesno  \undefined \def \bseriesno#1{#1}\fi
\ifx \blocation  \undefined \def \blocation#1{#1}\fi
\ifx \bsertitle  \undefined \def \bsertitle#1{#1}\fi
\ifx \bsnm \undefined \def \bsnm#1{#1}\fi
\ifx \bsuffix \undefined \def \bsuffix#1{#1}\fi
\ifx \bparticle \undefined \def \bparticle#1{#1}\fi
\ifx \barticle \undefined \def \barticle#1{#1}\fi
\bibcommenthead
\ifx \bconfdate \undefined \def \bconfdate #1{#1}\fi
\ifx \botherref \undefined \def \botherref #1{#1}\fi
\ifx \url \undefined \def \url#1{\textsf{#1}}\fi
\ifx \bchapter \undefined \def \bchapter#1{#1}\fi
\ifx \bbook \undefined \def \bbook#1{#1}\fi
\ifx \bcomment \undefined \def \bcomment#1{#1}\fi
\ifx \oauthor \undefined \def \oauthor#1{#1}\fi
\ifx \citeauthoryear \undefined \def \citeauthoryear#1{#1}\fi
\ifx \endbibitem  \undefined \def \endbibitem {}\fi
\ifx \bconflocation  \undefined \def \bconflocation#1{#1}\fi
\ifx \arxivurl  \undefined \def \arxivurl#1{\textsf{#1}}\fi
\csname PreBibitemsHook\endcsname

\bibitem{cortes1995support}
\begin{barticle}
\bauthor{\bsnm{Cortes}, \binits{C.}},
\bauthor{\bsnm{Vapnik}, \binits{V.}}:
\batitle{Support-vector networks}.
\bjtitle{Machine Learning}
\bvolume{20}(\bissue{3}),
\bfpage{273}--\blpage{297}
(\byear{1995})
\end{barticle}
\endbibitem

\bibitem{Chauhan2019problem}
\begin{barticle}
\bauthor{\bsnm{Chauhan}, \binits{V.K.}},
\bauthor{\bsnm{Dahiya}, \binits{K.}},
\bauthor{\bsnm{Sharma}, \binits{A.}}:
\batitle{Problem formulations and solvers in linear svm: a review}.
\bjtitle{Artificial Intelligence Review}
\bvolume{52}(\bissue{2}),
\bfpage{803}--\blpage{855}
(\byear{2019})
\end{barticle}
\endbibitem

\bibitem{Vapnik2013nature}
\begin{bbook}
\bauthor{\bsnm{Vapnik}, \binits{V.}}:
\bbtitle{The Nature of Statistical Learning Theory}.
\bpublisher{Springer},
\blocation{New York}
(\byear{2013})
\end{bbook}
\endbibitem

\bibitem{chapelle2002choosing}
\begin{barticle}
\bauthor{\bsnm{Chapelle}, \binits{O.}},
\bauthor{\bsnm{Vapnik}, \binits{V.}},
\bauthor{\bsnm{Bousquet}, \binits{O.}},
\bauthor{\bsnm{Mukherjee}, \binits{S.}}:
\batitle{Choosing multiple parameters for support vector machines}.
\bjtitle{Machine Learning}
\bvolume{46}(\bissue{1}),
\bfpage{131}--\blpage{159}
(\byear{2002})
\end{barticle}
\endbibitem

\bibitem{duan2003evaluation}
\begin{barticle}
\bauthor{\bsnm{Duan}, \binits{K.B.}},
\bauthor{\bsnm{Keerthi}, \binits{S.S.}},
\bauthor{\bsnm{Poo}, \binits{A.N.}}:
\batitle{Evaluation of simple performance measures for tuning svm
  hyperparameters}.
\bjtitle{Neurocomputing}
\bvolume{51},
\bfpage{41}--\blpage{59}
(\byear{2003})
\end{barticle}
\endbibitem

\bibitem{Keerthi2007Efficient}
\begin{botherref}
\oauthor{\bsnm{Keerthi}, \binits{S.S.}},
\oauthor{\bsnm{Sindhwani}, \binits{V.}},
\oauthor{\bsnm{Chapelle}, \binits{O.}}:
An efficient method for gradient-based adaptation of hyperparameters in svm
  models.
MIT Press
(2007)
\end{botherref}
\endbibitem

\bibitem{kunapuli2008bilevel1}
\begin{bbook}
\bauthor{\bsnm{Kunapuli}, \binits{G.}}:
\bbtitle{A Bilevel Optimization Approach to Machine Learning}.
\bpublisher{Rensselaer Polytechnic Institute},
\blocation{New York}
(\byear{2008})
\end{bbook}
\endbibitem

\bibitem{couellan2015bi}
\begin{barticle}
\bauthor{\bsnm{Couellan}, \binits{N.}},
\bauthor{\bsnm{Wang}, \binits{W.J.}}:
\batitle{Bi-level stochastic gradient for large scale support vector machine}.
\bjtitle{Neurocomputing}
\bvolume{153},
\bfpage{300}--\blpage{308}
(\byear{2015})
\end{barticle}
\endbibitem

\bibitem{kunapuli2008bilevel}
\begin{barticle}
\bauthor{\bsnm{Kunapuli}, \binits{G.}},
\bauthor{\bsnm{Bennett}, \binits{K.P.}},
\bauthor{\bsnm{Hu}, \binits{J.}},
\bauthor{\bsnm{Pang}, \binits{J.-S.}}:
\batitle{Bilevel model selection for support vector machines}.
\bjtitle{Data mining and mathematical programming}
\bvolume{45},
\bfpage{129}--\blpage{158}
(\byear{2008})
\end{barticle}
\endbibitem

\bibitem{kunapuli2008classification}
\begin{barticle}
\bauthor{\bsnm{Kunapuli}, \binits{G.}},
\bauthor{\bsnm{Bennett}, \binits{K.P.}},
\bauthor{\bsnm{Hu}, \binits{J.}},
\bauthor{\bsnm{Pang}, \binits{J.-S.}}:
\batitle{Classification model selection via bilevel programming}.
\bjtitle{Optimization Methods \& Software}
\bvolume{23}(\bissue{4}),
\bfpage{475}--\blpage{489}
(\byear{2008})
\end{barticle}
\endbibitem

\bibitem{momma2002pattern}
\begin{bchapter}
\bauthor{\bsnm{Momma}, \binits{M.}},
\bauthor{\bsnm{Bennett}, \binits{K.P.}}:
\bctitle{A pattern search method for model selection of support vector
  regression}.
In: \bbtitle{Proceedings of the 2002 SIAM International Conference on Data
  Mining},
pp. \bfpage{261}--\blpage{274}
(\byear{2002})
\end{bchapter}
\endbibitem

\bibitem{bennett2006model}
\begin{bchapter}
\bauthor{\bsnm{Bennett}, \binits{K.P.}},
\bauthor{\bsnm{Hu}, \binits{J.}},
\bauthor{\bsnm{Ji}, \binits{X.Y.}},
\bauthor{\bsnm{Kunapuli}, \binits{G.}},
\bauthor{\bsnm{Pang}, \binits{J.-S.}}:
\bctitle{Model selection via bilevel optimization}.
In: \bbtitle{The 2006 IEEE International Joint Conference on Neural Network
  Proceedings},
pp. \bfpage{1922}--\blpage{1929}
(\byear{2006}).
\bcomment{IEEE}
\end{bchapter}
\endbibitem

\bibitem{Yu2020hyper}
\begin{botherref}
\oauthor{\bsnm{Yu}, \binits{T.}},
\oauthor{\bsnm{Zhu}, \binits{H.}}:
Hyper-parameter optimization: A review of algorithms and applications.
arXiv preprint arXiv:2003.05689
(2020)
\end{botherref}
\endbibitem

\bibitem{Luo2016review}
\begin{barticle}
\bauthor{\bsnm{Luo}, \binits{G.}}:
\batitle{A review of automatic selection methods for machine learning
  algorithms and hyper-parameter values}.
\bjtitle{Network Modeling Analysis in Health Informatics and Bioinformatics}
\bvolume{5}(\bissue{1}),
\bfpage{1}--\blpage{16}
(\byear{2016})
\end{barticle}
\endbibitem

\bibitem{okuno2018hyperparameter}
\begin{botherref}
\oauthor{\bsnm{Okuno}, \binits{T.}},
\oauthor{\bsnm{Takeda}, \binits{A.}},
\oauthor{\bsnm{Kawana}, \binits{A.}}:
Hyperparameter learning for bilevel nonsmooth optimization.
arXiv preprint arXiv:1806.01520
(2018)
\end{botherref}
\endbibitem

\bibitem{kunisch2013bilevel}
\begin{barticle}
\bauthor{\bsnm{Kunisch}, \binits{K.}},
\bauthor{\bsnm{Pock}, \binits{T.}}:
\batitle{A bilevel optimization approach for parameter learning in variational
  models}.
\bjtitle{SIAM Journal on Imaging Sciences}
\bvolume{6}(\bissue{2}),
\bfpage{938}--\blpage{983}
(\byear{2013})
\end{barticle}
\endbibitem

\bibitem{mooregradient}
\begin{botherref}
\oauthor{\bsnm{Moore}, \binits{G.}},
\oauthor{\bsnm{Bergeron}, \binits{C.}},
\oauthor{\bsnm{Bennett}, \binits{K.P.}}:
Gradient-type methods for primal SVM model selection.
\url{http://opt.kyb.tuebingen.mpg.de/papers/OPT2010-moore.pdf}.
Online; accessed 10-July-2021
\end{botherref}
\endbibitem

\bibitem{moore2009nonsmooth}
\begin{bchapter}
\bauthor{\bsnm{Moore}, \binits{G.}},
\bauthor{\bsnm{Bergeron}, \binits{C.}},
\bauthor{\bsnm{Bennett}, \binits{K.P.}}:
\bctitle{Nonsmooth bilevel programming for hyperparameter selection}.
In: \bbtitle{2009 IEEE International Conference on Data Mining Workshops},
pp. \bfpage{374}--\blpage{381}
(\byear{2009})
\end{bchapter}
\endbibitem

\bibitem{colson2007overview}
\begin{barticle}
\bauthor{\bsnm{Colson}, \binits{B.}},
\bauthor{\bsnm{Marcotte}, \binits{P.}},
\bauthor{\bsnm{Savard}, \binits{G.}}:
\batitle{An overview of bilevel optimization}.
\bjtitle{Annals of Operations Research}
\bvolume{153}(\bissue{1}),
\bfpage{235}--\blpage{256}
(\byear{2007})
\end{barticle}
\endbibitem

\bibitem{dempe2002foundations}
\begin{bbook}
\bauthor{\bsnm{Dempe}, \binits{S.}}:
\bbtitle{Foundations of Bilevel Programming}.
\bpublisher{Kluwer},
\blocation{Dordrecht}
(\byear{2002})
\end{bbook}
\endbibitem

\bibitem{dempebilevelbook}
\begin{bbook}
\bauthor{\bsnm{Dempe}, \binits{S.}},
\bauthor{\bsnm{Zemkoho}, \binits{A.B.}}:
\bbtitle{Bilevel Optimization Advances and Next Challenges}.
\bpublisher{Springer},
\blocation{New York}
(\byear{2020})
\end{bbook}
\endbibitem

\bibitem{mejia2019metaheuristic}
\begin{bchapter}
\bauthor{\bsnm{Mej{\'\i}a-de-Dios}, \binits{J.-A.}},
\bauthor{\bsnm{Mezura-Montes}, \binits{E.}}:
\bctitle{A metaheuristic for bilevel optimization using tykhonov regularization
  and the quasi-newton method}.
In: \bbtitle{2019 IEEE Congress on Evolutionary Computation},
pp. \bfpage{3134}--\blpage{3141}
(\byear{2019})
\end{bchapter}
\endbibitem

\bibitem{zemkoho2021theoretical}
\begin{barticle}
\bauthor{\bsnm{Zemkoho}, \binits{A.B.}},
\bauthor{\bsnm{Zhou}, \binits{S.L.}}:
\batitle{Theoretical and numerical comparison of the Karush--Kuhn--Tucker and
  value function reformulations in bilevel optimization}.
\bjtitle{Computational Optimization and Applications}
\bvolume{78}(\bissue{2}),
\bfpage{625}--\blpage{674}
(\byear{2021})
\end{barticle}
\endbibitem

\bibitem{fischer2019semismooth}
\begin{barticle}
\bauthor{\bsnm{Fischer}, \binits{A.}},
\bauthor{\bsnm{Zemkoho}, \binits{A.B.}},
\bauthor{\bsnm{Zhou}, \binits{S.L.}}:
\batitle{Semismooth newton-type method for bilevel optimization: Global
  convergence and extensive numerical experiments}.
\bjtitle{Optimization Methods \& Software}
(\byear{2021}).
\doiurl{10.1080/10556788.2021.1977810}
\end{barticle}
\endbibitem

\bibitem{lin2014solving}
\begin{barticle}
\bauthor{\bsnm{Lin}, \binits{G.-H.}},
\bauthor{\bsnm{Xu}, \binits{M.W.}},
\bauthor{\bsnm{Ye}, \binits{J.J.}}:
\batitle{On solving simple bilevel programs with a nonconvex lower level
  program}.
\bjtitle{Mathematical Programming}
\bvolume{144}(\bissue{1}),
\bfpage{277}--\blpage{305}
(\byear{2014})
\end{barticle}
\endbibitem

\bibitem{ye2010new}
\begin{barticle}
\bauthor{\bsnm{Ye}, \binits{J.J.}},
\bauthor{\bsnm{Zhu}, \binits{D.L.}}:
\batitle{New necessary optimality conditions for bilevel programs by combining
  the mpec and value function approaches}.
\bjtitle{SIAM Journal on Optimization}
\bvolume{20}(\bissue{4}),
\bfpage{1885}--\blpage{1905}
(\byear{2010})
\end{barticle}
\endbibitem

\bibitem{ochs2016techniques}
\begin{barticle}
\bauthor{\bsnm{Ochs}, \binits{P.}},
\bauthor{\bsnm{Ranftl}, \binits{R.}},
\bauthor{\bsnm{Brox}, \binits{T.}},
\bauthor{\bsnm{Pock}, \binits{T.}}:
\batitle{Techniques for gradient-based bilevel optimization with non-smooth
  lower level problems}.
\bjtitle{Journal of Mathematical Imaging and Vision}
\bvolume{56}(\bissue{2}),
\bfpage{175}--\blpage{194}
(\byear{2016})
\end{barticle}
\endbibitem

\bibitem{ochs2015bilevel}
\begin{bchapter}
\bauthor{\bsnm{Ochs}, \binits{P.}},
\bauthor{\bsnm{Ranftl}, \binits{R.}},
\bauthor{\bsnm{Brox}, \binits{T.}},
\bauthor{\bsnm{Pock}, \binits{T.}}:
\bctitle{Bilevel optimization with nonsmooth lower level problems}.
In: \bbtitle{International Conference on Scale Space and Variational Methods in
  Computer Vision},
pp. \bfpage{654}--\blpage{665}
(\byear{2015})
\end{bchapter}
\endbibitem

\bibitem{luo1996mathematical}
\begin{bbook}
\bauthor{\bsnm{Luo}, \binits{Z.-Q.}},
\bauthor{\bsnm{Pang}, \binits{J.-S.}},
\bauthor{\bsnm{Ralph}, \binits{D.}}:
\bbtitle{Mathematical Programs with Equilibrium Constraints}.
\bpublisher{Cambridge University Press},
\blocation{Cambridge}
(\byear{1996})
\end{bbook}
\endbibitem

\bibitem{bennett2008bilevel}
\begin{bchapter}
\bauthor{\bsnm{Bennett}, \binits{K.P.}},
\bauthor{\bsnm{Kunapuli}, \binits{G.}},
\bauthor{\bsnm{Hu}, \binits{J.}},
\bauthor{\bsnm{Pang}, \binits{J.-S.}}:
\bctitle{Bilevel optimization and machine learning}.
In: \bbtitle{IEEE World Congress on Computational Intelligence},
pp. \bfpage{25}--\blpage{47}
(\byear{2008})
\end{bchapter}
\endbibitem

\bibitem{wu2015inexact}
\begin{barticle}
\bauthor{\bsnm{Wu}, \binits{J.}},
\bauthor{\bsnm{Zhang}, \binits{L.W.}},
\bauthor{\bsnm{Zhang}, \binits{Y.}}:
\batitle{An inexact newton method for stationary points of mathematical
  programs constrained by parameterized quasi-variational inequalities}.
\bjtitle{Numerical Algorithms}
\bvolume{69}(\bissue{4}),
\bfpage{713}--\blpage{735}
(\byear{2015})
\end{barticle}
\endbibitem

\bibitem{harder2021reformulation}
\begin{barticle}
\bauthor{\bsnm{Harder}, \binits{F.}},
\bauthor{\bsnm{Mehlitz}, \binits{P.}},
\bauthor{\bsnm{Wachsmuth}, \binits{G.}}:
\batitle{Reformulation of the m-stationarity conditions as a system of
  discontinuous equations and its solution by a semismooth newton method}.
\bjtitle{SIAM Journal on Optimization}
\bvolume{31}(\bissue{2}),
\bfpage{1459}--\blpage{1488}
(\byear{2021})
\end{barticle}
\endbibitem

\bibitem{lee2015global}
\begin{barticle}
\bauthor{\bsnm{Lee}, \binits{Y.-C.}},
\bauthor{\bsnm{Pang}, \binits{J.-S.}},
\bauthor{\bsnm{Mitchell}, \binits{J.E.}}:
\batitle{Global resolution of the support vector machine regression parameters
  selection problem with lpcc}.
\bjtitle{EURO Journal on Computational Optimization}
\bvolume{3}(\bissue{3}),
\bfpage{197}--\blpage{261}
(\byear{2015})
\end{barticle}
\endbibitem

\bibitem{mangasarian1994misclassification}
\begin{barticle}
\bauthor{\bsnm{Mangasarian}, \binits{O.L.}}:
\batitle{Misclassification minimization}.
\bjtitle{Journal of Global Optimization}
\bvolume{5}(\bissue{4}),
\bfpage{309}--\blpage{323}
(\byear{1994})
\end{barticle}
\endbibitem

\bibitem{scholtes2001convergence}
\begin{barticle}
\bauthor{\bsnm{Scholtes}, \binits{S.}}:
\batitle{Convergence properties of a regularization scheme for mathematical
  programs with complementarity constraints}.
\bjtitle{SIAM Journal on Optimization}
\bvolume{11}(\bissue{4}),
\bfpage{918}--\blpage{936}
(\byear{2001})
\end{barticle}
\endbibitem

\bibitem{cristianini2000introduction}
\begin{bbook}
\bauthor{\bsnm{Cristianini}, \binits{N.}},
\bauthor{\bsnm{Shawe-Taylor}, \binits{J.}}:
\bbtitle{An Introduction to Support Vector Machines and Other Kernel-based
  Learning Methods}.
\bpublisher{Cambridge University Press},
\blocation{Cambridge}
(\byear{2000})
\end{bbook}
\endbibitem

\bibitem{galli2021study}
\begin{botherref}
\oauthor{\bsnm{Galli}, \binits{L.}},
\oauthor{\bsnm{Lin}, \binits{C.-J.}}:
A study on truncated newton methods for linear classification.
IEEE Transactions on Neural Networks and Learning Systems
(2021)
\end{botherref}
\endbibitem

\bibitem{hsieh2008dual}
\begin{bchapter}
\bauthor{\bsnm{Hsieh}, \binits{C.-J.}},
\bauthor{\bsnm{Chang}, \binits{K.-W.}},
\bauthor{\bsnm{Lin}, \binits{C.-J.}},
\bauthor{\bsnm{Keerthi}, \binits{S.S.}},
\bauthor{\bsnm{Sundararajan}, \binits{S.}}:
\bctitle{A dual coordinate descent method for large-scale linear SVM}.
In: \bbtitle{Proceedings of the 25th International Conference on Machine Learning},
pp. \bfpage{408}--\blpage{415}
(\byear{2008})
\end{bchapter}
\endbibitem

\bibitem{dempe2003annotated}
\begin{barticle}
\bauthor{\bsnm{Dempe}, \binits{S.}}:
\batitle{Annotated bibliography on bilevel programming and mathematical
  programs with equilibrium constraints}.
\bjtitle{Optimization}
\bvolume{52}(\bissue{3}),
\bfpage{333}--\blpage{359}
(\byear{2003})
\end{barticle}
\endbibitem

\bibitem{jane2005necessary}
\begin{barticle}
\bauthor{\bsnm{Ye}, \binits{J.J.}}:
\batitle{Necessary and sufficient optimality conditions for mathematical
  programs with equilibrium constraints}.
\bjtitle{Journal of Mathematical Analysis and Applications}
\bvolume{307}(\bissue{1}),
\bfpage{350}--\blpage{369}
(\byear{2005})
\end{barticle}
\endbibitem

\bibitem{flegel2005constraint}
\begin{botherref}
\oauthor{\bsnm{Flegel}, \binits{M.L.}}:
Constraint qualifications and stationarity concepts for mathematical programs
  with equilibrium constraints.
PhD thesis,
Universit{\"a}t W{\"u}rzburg
(2005)
\end{botherref}
\endbibitem

\bibitem{guo2015solving}
\begin{barticle}
\bauthor{\bsnm{Guo}, \binits{L.}},
\bauthor{\bsnm{Lin}, \binits{G.-H.}},
\bauthor{\bsnm{Ye}, \binits{J.J.}}:
\batitle{Solving mathematical programs with equilibrium constraints}.
\bjtitle{Journal of Optimization Theory and Applications}
\bvolume{166}(\bissue{1}),
\bfpage{234}--\blpage{256}
(\byear{2015})
\end{barticle}
\endbibitem

\bibitem{jara2018study}
\begin{barticle}
\bauthor{\bsnm{Jara-Moroni}, \binits{F.}},
\bauthor{\bsnm{Pang}, \binits{J.-S.}},
\bauthor{\bsnm{W{\"a}chter}, \binits{A.}}:
\batitle{A study of the difference-of-convex approach for solving linear
  programs with complementarity constraints}.
\bjtitle{Mathematical Programming}
\bvolume{169}(\bissue{1}),
\bfpage{221}--\blpage{254}
(\byear{2018})
\end{barticle}
\endbibitem

\bibitem{judice2012algorithms}
\begin{barticle}
\bauthor{\bsnm{J{\'u}dice}, \binits{J.J.}}:
\batitle{Algorithms for linear programming with linear complementarity
  constraints}.
\bjtitle{TOP}
\bvolume{20}(\bissue{1}),
\bfpage{4}--\blpage{25}
(\byear{2012})
\end{barticle}
\endbibitem

\bibitem{li2015superlinearly}
\begin{barticle}
\bauthor{\bsnm{Li}, \binits{J.L.}},
\bauthor{\bsnm{Huang}, \binits{R.S.}},
\bauthor{\bsnm{Jian}, \binits{J.B.}}:
\batitle{A superlinearly convergent {QP}-free algorithm for mathematical
  programs with equilibrium constraints}.
\bjtitle{Applied Mathematics and Computation}
\bvolume{269},
\bfpage{885}--\blpage{903}
(\byear{2015})
\end{barticle}
\endbibitem

\bibitem{yu2019solving}
\begin{barticle}
\bauthor{\bsnm{Yu}, \binits{B.}},
\bauthor{\bsnm{Mitchell}, \binits{J.E.}},
\bauthor{\bsnm{Pang}, \binits{J.-S.}}:
\batitle{Solving linear programs with complementarity constraints using
  branch-and-cut}.
\bjtitle{Mathematical Programming Computation}
\bvolume{11}(\bissue{2}),
\bfpage{267}--\blpage{310}
(\byear{2019})
\end{barticle}
\endbibitem

\bibitem{anitescu2000solving}
\begin{botherref}
\oauthor{\bsnm{Anitescu}, \binits{M.}}:
On solving mathematical programs with complementarity constraints as nonlinear
  programs.
Preprint ANL/MCS-P$864$-$1200$, Argonne National Laboratory, Argonne, IL
\textbf{3}
(2000)
\end{botherref}
\endbibitem

\bibitem{facchinei2007finite}
\begin{bbook}
\bauthor{\bsnm{Facchinei}, \binits{F.}},
\bauthor{\bsnm{Pang}, \binits{J.-S.}}:
\bbtitle{Finite-dimensional Variational Inequalities and Complementarity
  Problems}.
\bpublisher{Springer},
\blocation{NewYork}
(\byear{2007})
\end{bbook}
\endbibitem

\bibitem{fletcher2006local}
\begin{barticle}
\bauthor{\bsnm{Fletcher}, \binits{R.}},
\bauthor{\bsnm{Leyffer}, \binits{S.}},
\bauthor{\bsnm{Ralph}, \binits{D.}},
\bauthor{\bsnm{Scholtes}, \binits{S.}}:
\batitle{Local convergence of sqp methods for mathematical programs with
  equilibrium constraints}.
\bjtitle{SIAM Journal on Optimization}
\bvolume{17}(\bissue{1}),
\bfpage{259}--\blpage{286}
(\byear{2006})
\end{barticle}
\endbibitem

\bibitem{fukushima2002implementable}
\begin{barticle}
\bauthor{\bsnm{Fukushima}, \binits{M.}},
\bauthor{\bsnm{Tseng}, \binits{P.}}:
\batitle{An implementable active-set algorithm for computing a b-stationary
  point of a mathematical program with linear complementarity constraints}.
\bjtitle{SIAM Journal on Optimization}
\bvolume{12}(\bissue{3}),
\bfpage{724}--\blpage{739}
(\byear{2002})
\end{barticle}
\endbibitem

\bibitem{hoheisel2013theoretical}
\begin{barticle}
\bauthor{\bsnm{Hoheisel}, \binits{T.}},
\bauthor{\bsnm{Kanzow}, \binits{C.}},
\bauthor{\bsnm{Schwartz}, \binits{A.}}:
\batitle{Theoretical and numerical comparison of relaxation methods for
  mathematical programs with complementarity constraints}.
\bjtitle{Mathematical Programming}
\bvolume{137}(\bissue{1}),
\bfpage{257}--\blpage{288}
(\byear{2013})
\end{barticle}
\endbibitem

\bibitem{dempe2012karush}
\begin{barticle}
\bauthor{\bsnm{Dempe}, \binits{S.}},
\bauthor{\bsnm{Zemkoho}, \binits{A.B.}}:
\batitle{On the {K}arush--{K}uhn--{T}ucker reformulation of the bilevel
  optimization problem}.
\bjtitle{Nonlinear Analysis: Theory, Methods \& Applications}
\bvolume{75}(\bissue{3}),
\bfpage{1202}--\blpage{1218}
(\byear{2012})
\end{barticle}
\endbibitem

\bibitem{Yan2020efficient}
\begin{barticle}
\bauthor{\bsnm{Yan}, \binits{Y.Q.}},
\bauthor{\bsnm{Li}, \binits{Q.N.}}:
\batitle{An efficient augmented lagrangian method for support vector machine}.
\bjtitle{Optimization Methods and Software}
\bvolume{35}(\bissue{4}),
\bfpage{855}--\blpage{883}
(\byear{2020})
\end{barticle}
\endbibitem

\bibitem{philip2002user}
\begin{botherref}
\oauthor{\bsnm{Gill}, \binits{P.E.}},
\oauthor{\bsnm{Murray}, \binits{W.}},
\oauthor{\bsnm{Saunders}, \binits{M.A.}}:
User's guide for snopt version 6, a fortran package for large-scale nonlinear
  programming.
University of California, California
(2002)
\end{botherref}
\endbibitem

\end{thebibliography}

}
}
\end{document}